\documentclass[11pt]{amsart}     
\usepackage{amssymb,amsthm,graphicx,here}
\newtheorem{theorem}{Theorem}
\newtheorem{definition}{Definition}
\newtheorem{example}{Example}
\newtheorem{proposition}{Proposition}
\newtheorem{lemma}{Lemma}

\newtheorem{remark}{Remark}
\begin{document}
\title[\text{Local singularities on torus curves of type (2,5) }]
{Classification of local singularities on torus curves of type (2,5) }
\author{ M. Kawashima }
\address{\vtop{
\hbox{Department of Mathematics}
\hbox{Tokyo University of science}
\hbox{wakamiya-cho 26, shinjuku-ku  }
\hbox{Tokyo 162-0827}
\hbox{\rm{E-mail}:{\rm j1107702@ed.kagu.tus.ac.jp}}
}}
\begin{abstract}
In this paper, we consider curves of degree 10 of torus type (2,5),
 $C:\,f_5(x,y)^2+f_2(x,y)^5=0$. Assume that $f_2(0,0)=f_5(0,0)=0$. Then
 $O=(0,0)$ is a singular point of $C$ which is called an inner singularity.
In this paper, we give a topological classification of singularities of $(C,O)$.
\end{abstract}

\maketitle 
\section{Introduction} 
A plane curve $C\subset \Bbb P^2$ is called a 
$curve \ of \ torus \ type \ (p,q)$ 
if there is a defining polynomial $F$ of $C$
which can be written as $F=F_{p}^q+F_{q}^p$,
where $F_p$, $F_q$ are homogeneous polynomials of $X,Y,Z$
 of degree $p$ and $q$ respectively. 
In \cite{Pho}, D.T. Pho classified the local and global configurations
of the singularities of sextics of torus type $(2,3)$.
In this paper, 
we will classify local inner singularities of
torus type $(2,5)$.
We assume $C$ is a reduced curve of torus type $(2,5)$.
Using affine coordinates $x=X/Z,\, y=Y/Z$, 
$C$ is defined  as
\[C:\, f_2(x,y)^5+f_5(x,y)^2=0,\] 
where 
$f_2(x,y)=F_2(x,y,1)$, $f_5(x,y)=F_5(x,y,1)$.
Put $C_{2}:=\{f_2(x,y)=0\}$ and $C_{5}:=\{f_5(x,y)=0\}$.
We assume that 
the origin $O=(0,0)$ is an intersection point of
$C_{2}$ and $C_{5}$ and 
$O$ is an isolated singularity of $C$.
We classify the local singularity type of $(C,O)$
following the method of \cite{Pho}.
If $f_2(x,y)=\ell(x,y)^2$ for some linear form $\ell$,
the curve $C$ is called a {\em linear torus curve} 
and it consists of two quintics,
as $f_2^5+f_5^2=(f_5+\ell^5 \sqrt{-1})(f_5-\ell^5\sqrt{-1})$. 

First we recall the following notations.  
\begin{eqnarray*}
& A_n:x^{n+1}+y^2=0 \ \ \ (n \geq 1)\ \\
& D_n:x^{n-1}+xy^2=0 \ \ \ (n \geq 4) \\ 
& E_{6}:x^3+y^4=0,\,  \  \ E_7:x^3+xy^3=0,\, \ \ E_8:x^3+y^5=0 \\
& \ B_{n,m}:x^n+y^m=0\quad \text{(Brieskorn-Pham type)} 
\end{eqnarray*}
In the case $\gcd(m,n)>1$, the equation of $B_{n,m}$
can contain other monomials on the Newton boundary.
For example, $B_{2,4}$ has the form $C_{t} : x^2+txy^2+y^4=0$,
for $t\not=\pm 2$. $C_{t}$ is  topologically  equivalent
to  $B_{2,4}$ (Oka, \cite{Oka-milnor}).
In these notations, we note that 
$A_n=B_{n+1,2}$ and $E_6=B_{3,4}$.

We remark that every non-degenerate singularity
in the sense of Newton boundary is 
a union of Brieskorn-Pham type singularities.
For example,
$C_{p,q}:\, x^p+x^2y^2+y^q,\, p,q\ge 4,\,p+q\ge 9 $
is the union of two singularities:
$x^{p-2}+y^2=0$ and $x^2+y^{q-2}=0$.
So we introduce the notation:
$B_{p-2,2}\circ B_{2,q-2}$ to express $C_{p,q}$.

For the classification, 
we use the local intersection multiplicity
$\iota:=I(C_2,C_5;O)$ effectively. 
The complete classifications is given by  Theorem \ref{res} in \S5.\\


This paper consists of the following sections:\\

\S2. Preliminaries.

\S3. Some lemmas for tours curves of type $(p,q)$.

\S4. Calculation of the local singularities.

\S5. List of classification.

\S6. Linear torus curve of type $(2,5)$.

\S7. Appendix.

\section{preliminaries}
\subsection{Toric modification}
Throughout this paper,
we follow the notations of Oka \cite{okabook}.
First we recall a toric modification. 
Let 
\begin{equation*}
\sigma= \begin{pmatrix}
         \alpha & \beta \\
         \gamma & \delta
         \end{pmatrix} 
\end{equation*}
be a unimodular integral $2 \times 2$ matrix.
We associate to $\sigma$ a birational morphism 
$\pi_{\sigma}:{\mathbb{C}^*}^2 \rightarrow {\mathbb{C}^*}^2$
by $\pi_{\sigma}(x,y)=(x^{\alpha} y^{\beta}, x^{\gamma} y^{\delta})$. 
If $\alpha, \gamma \geq 0$ (respectively. $\beta, \delta \geq 0$),
this map can be extend to $x=0$ (resp. $y=0$ ).
Note that $\{\pi_{\sigma}\}$'s satisfy the properties 
$\pi_{\sigma} \circ \pi_{\tau} = \pi_{\sigma \tau}$ 
and $(\pi_{\sigma})^{-1}=\pi_{{\sigma}^{-1}}$.

Let $N$ be a free $\Bbb{Z}$-module of rank 
two with a fixed basis $\{E_{1}, E_{2}\}$. 
Through this basis, 
we identify $N_{\Bbb{R}}:=N\otimes \Bbb{R}$ with $\mathbb{R}^2$. 
Thus $N$ can be understood as the set of integral points in $\Bbb{R}^2$. 
We denote a vector in $N$ by a column vector. 
Hereafter we fix two special vectors  
$E_1=\sideset{^t}{}{\mathop{(}}1,0)$ and $E_2=\sideset{^t}{}{\mathop{(}}0,1)$.
Let $N^{+}$ be the space of positive vectors of $N$.
Let $\{P_1, \cdots, P_m \}$ be given positive primitive 
integral vectors in $N^{+}$. 
Let $P_i=\sideset{^t}{}{\mathop{(}}a_i,b_i)$ and assume that 
$\det(P_i,P_{i+1}$)$
>0$ for each $i=0, \cdots, m$.
Here $P_0=E_1$, $P_{m+1}=E_2$.
We associated to $\{P_0,P_{1}, \cdots, P_{m+1} \}$ 
a simplicial cone subdivision $\Sigma^{*}$ of $N_{\Bbb{R}}$
which has $m+1$ cones 
Cone$(P_i,P_{i+1})$ of dimension two where 
\[{\rm{Cone}}(P_i,P_{i+1}):=\{tP_i+sP_{i+1} \mid t,s \geq 0\}.\]
We call $\{P_0, \cdots, P_{m+1} \}$ 
{\em{the vertices of $\Sigma^{*}$}}.
We say that $\Sigma^{*}$ is a $regular \ simplicial \ cone$ 
$subdivision$ of $N^{+}$ 
if $\det(P_i,P_{i+1})=1$ for each $i=0, \cdots, m$.

Assume that $\Sigma^{*}$ is a given regular simplicial cone subdivision
with vertices $\{P_0, P_1,$ $\cdots, P_{m+1} \}$, $(P_0=E_1, P_{m+1}=E_2)$
and let $P_i=\sideset{^t}{}{\mathop{(}}a_i,b_i)$.
For each cone, Cone$(P_i,P_{i+1})$, we associate the unimodular matrix 
\begin{equation*}
\sigma_i := \begin{pmatrix}
         a_i & a_{i+1} \\
         b_i & b_{i+1}
         \end{pmatrix}
\end{equation*}
We identify Cone$(P_i,P_{i+1})$ with
the unimodular matrix $\sigma_i$.
Let $(x,y)$ be a fixed system of coordinates of $\mathbb{C}^2$.
Then we consider, for each $\sigma_i$, 
an affine space $\mathbb{C}_{\sigma_{i}}^2$
of dimension two with coordinates $(x_{\sigma_i}, y_{\sigma_i})$ 
and the birational map 
$\pi_{\sigma_i} \colon \mathbb{C}_{\sigma_i}^2 \rightarrow \mathbb{C}^2$.
First we consider the disjoint union 
of $\mathbb{C}_{\sigma_i}^2$ for $i=0,\dots,m$
and we define the variety $X$ as the quotient 
of this union by the following identification: Two points 
$(x_{\sigma_i}, y_{\sigma_i}) \in \mathbb{C}^2_{\sigma_i}$ and 
$(x_{\sigma_j}, y_{\sigma_j}) \in \mathbb{C}^2_{\sigma_j}$ are identified 
if and only if the birational map $\pi_{{\sigma_j}^{-1}\sigma_i}$ is well
defined at the point $(x_{\sigma_i},y_{\sigma_i})$ and 
$\pi_{{\sigma_j}^{-1}\sigma_i}(x_{\sigma_i},y_{\sigma_i})
=(x_{\sigma_j},y_{\sigma_j})$.
As can be easily checked, $X$ is non-singular and the maps 
$\{\pi_{\sigma_i}\colon \mathbb{C}^2_{\sigma_i}
\rightarrow \mathbb{C}^2 \mid 0 \leq i \leq m \}$ 
glue into a proper analytic map $\pi \colon X \rightarrow \mathbb{C}^2$.
\begin{definition}
The map $\pi:X \rightarrow \mathbb{C}^2$ is called
the toric modification associated 
with $\{\Sigma^*, (x,y), O\}$ where 
is a regular simplicial cone subdivision of 
$N^{+}$ and $(x,y)$ is a coordinate system of
$\mathbb{C}^2$ centered at the origin $O$.
\end{definition}
Recall that  this modification has the following properties.
\newline
1. $\{\mathbb{C}_{\sigma_i}^2,(x_{\sigma_i},y_{\sigma_i})\}, (0 \leq i \leq m)$
give coordinate charts of $X$ and 
we call them {\em the toric coordinate charts} of $X$.
\newline
2. Two affine divisors 
$\{y_{\sigma_{i-1}}=0\} \subset \mathbb{C}^2_{\sigma_{i-1}}$
and $\{x_{\sigma_i}=0\} \subset \mathbb{C}^2_{\sigma_i}$
glue together to make a compact divisor isomorphic to
$\mathbb{P}^1$ for $0 \leq i \leq m$.
We denote this divisor by $\hat{E}(P_i)$. 
\newline
3. $\pi^{-1}(O)=\bigcup_{i=1}^{m}\hat{E}(P_i)$
and $\pi:X-\pi^{-1}(O) \rightarrow 
\mathbb{C}^2-\{O\}$ is isomorphism. 
The non-compact divisor $x_{\sigma_0}=0$ 
(respectively. $y_{\sigma_m}=0$) is 
mapped isomorphically onto the divisor $x=0$ (resp. $y=0$). 
\newline
4. $\hat{E}(P_i) \cap \hat{E}(P_j) \not = \emptyset$ 
if and only if  $i-j = \pm 1$.
If $i-j = \pm 1$, they intersect transversely  
at a point.   
\subsection{Toric modification with respect to an analytic function}
Let $\mathcal{O}$ be the ring of germs of analytic functions
at the origin.
Recall that $\mathcal{O}$ is isomorphic to the ring of convergent power
series in $x,y$.
Let $f\in \mathcal O$ be a germ of complex analytic function
and suppose that $f(O)=0$.
Let $f(x,y)= \sum c_{i,j}x^iy^j$ be the Taylor expansion of $f$ 
at the origin. 
We assume that $f(x,y)$ is reduced as a germ.
{\em{The Newton polygon}} $\Gamma_+(f;x,y)$ of $f$,
with respect to the coordinate system $(x,y)$,
is the convex hull of the union 
$\bigcup_{i,j}\{(i,j)+\mathbb{R}_{+}^2 \}$
where the union is taken for $(i,j)$ such that $c_{i,j}\not =0$
and {\em{the Newton boundary}} $\Gamma(f;x,y)$ is 
the union of compact faces of
the Newton polygon $\Gamma_+(f;x,y)$. 
For each compact face $\Delta$ of $\Gamma(f;x,y)$, 
{\em{the face function}} $f_{\Delta}(x,y)$ is defined 
 by $f_{\Delta}(x,y):=\sum_{(i,j) \in \Delta} c_{i,j}x^iy^j$.
 
In the space $M$ where the Newton polygon
$\Gamma_+(f;x,y)$ is contained, 
we use $(\nu_1,\nu_2)$ as the coordinates. 
For any positive weight vector $P={}^t(a,b)\in N$, 
we consider $P$ as a linear function 
on $M$ by $P(\nu_1,\nu_2)=a\nu_1+b\nu_2$. 
We define $d(P;f)$ to be the smallest value of 
the restriction of $P$ to 
the Newton polygon $\Gamma_+(f;x,y)$ and $\Delta(P;f)$
be the face where this smallest value is taken.
For simplicity we shall write 
$f_{P}$ instead of $f_{\Delta(P;f)}$.
By the definition, $f_{P}$ is weighted 
homogeneous polynomial of degree $d(P;f)$ with weight 
$P={}^t(a,b)$.
For each face $\Delta \in \Gamma(f;x,y)$ there is 
a unique primitive integral vector 
$P=\sideset{^t}{}{\mathop{(}}a,b)$
such that $\Delta=\Delta(P;f)$.
The Newton boundary $\Gamma(f;x,y)$
has a finite number of faces. \\

\hspace{4.5cm}{\small{\input{./in-1.tex}}}\\

Let $\Delta_1, \dots, \Delta_m$ these faces 
and let $P_i=\sideset{^t}{}{\mathop{(}}a_i, b_i)$
be the corresponding positive primitive integral vector,
i.e.,
$\Delta_i=\Delta (P_i;f)$.
We call $P_i$ the weight vector of the face $\Delta_{i}$.
Then we can write   
\[f_{P_i}(x,y)=cx^{r_i} y^{s_i} \prod_{j=1}^{k_i}(y^{a_i}
+{\gamma}_{i,j} x^{b_i})^{\nu_{i,j}},\quad c\ne0\] 
with distinct non-zero complex numbers 
$\gamma_{i,1},\cdots, \gamma_{i,k_{i}}$.
We define 
\[
 \tilde{f}_{P_i}(x,y)=f_{P_i}(x,y)/x^{r_i}y^{s_i} 
=c\,\prod_{j=1}^{k_i}(y^{a_i}+{\gamma}_{i,j}x^{b_i})^{\nu_{i,j}}.
\] 

The polynomial $\sum_{(i,j) \in \Gamma(f;x,y)}c_{i,j}x^iy^j$
is called {\em{the Newton principal part}} of $f(x,y)$
and we denote by  $\mathcal{N}(f;x,y)$.      
We say that $f(x,y)$ is {\em{convenient}} 
if the intersection $\Gamma (f;x,y)$ 
with each axe is non-empty. 
Note that $\tilde{f}_{\Delta_i}$ is always convenient.
We say that $f$ is {\em{non-degenerate on a face $\Delta_i$}}
if the function
$f_{{\Delta}_{i}}:{\mathbb{C}^*}^2 \rightarrow \mathbb{C}$
has no critical points.
This is equivalent to $\nu_{i,j}=1$ for all $j=1,\cdots, k_i$.
We say that $f$ is {\em{non-degenerate}} if $f$ 
is non-degenerate on any face $\Delta_i$ $(i=1,\dots, m)$.

We introduce an equivalence relation $\sim$ in $N^{+}$ 
which is 
defined by $P \sim Q$ if and only if 
$\Delta(P;f)=\Delta(Q;f)$.
The equivalence classes
define a conical subdivision of $N^{+}$.
This gives a simplicial cone subdivision of $N^{+}$ 
with $m+2$ vertices $\{P_0, \cdots P_{m+1} \}$
 with $P_0=E_1$, $P_{m+1}=E_2$.
We denote this subdivision by $\Gamma^{*}(f;x,y)$
and we call it {\em{the dual Newton diagram}} of $f$
with respect to the system of coordinates $(x,y)$.
$\Gamma^{*}(f;x,y)$ has $m+1$ two-dimensional cones
Cone$(P_i,P_{i+1})$, $i=0, \cdots ,m$.
Note that these cones are not regular in general.  
\begin{definition}
A regular simplicial cone subdivision $\Sigma^{*}$ 
is admissible for $f(x,y)$ 
if $\Sigma^{*}$ is a subdivision of $\Gamma^{*}(f;x,y)$.
The corresponding toric modification 
$\pi:X \rightarrow \mathbb{C}^2$ is called an admissible 
toric modification for $f(x,y)$
with respect to the system of coordinate $(x,y)$.
\end{definition}
There exists a unique canonical regular
simplicial cone subdivision (Lemma 3,3 of \cite{MR88m:32023}) . 
We call the corresponding toric modification
{\em the canonical toric modification}
with respect to $f(x,y)$. 
Let $C$ be a germ of a reduced curve defined by $f(x,y)=0$ and 
let $\pi :X \to \Bbb{C}^2$ be a good resolution.
Recall that the  dual graph  of
the resolution  is defined as follows.
Let $E_1,\dots, E_r$ be the exceptional divisors and put 
$\pi^{*} f^{-1}(0)= \sum_{i=1}^{r} m_{i}E_{i}+\sum_{j}^{s} \widetilde{C_{j}}$.
To each $E_i$,  we associate a vertex $v_i$ of $\mathcal{G}(\pi)$
denoted it by a black circle.
We give an edge joining $v_i$ and $v_j$ 
if $E_{i} \cap E_{j} \not = \emptyset$.
For the extended dual graph $\widetilde{\mathcal{G}}(\pi)$, 
we add vertices $w_i $
to each irreducible components $C_i,\,i=1, \cdots , s$ and 
we join $w_j$ and $v_i$ if $E_i\cap C_j \not = \emptyset$
by a dotted arrow line. 
It is also important to remember 
the multiplicities of $\pi^{*}f$ along $E_i$,
which we denote by $m_i$.
So we put weights $m_i$ to each vertex and call $\mathcal{G}(\pi)$ with weight
{\em{the weighted dual graph of the resolution $\pi:X\to \Bbb C^2$.}}
\begin{example}\label{exam-1}
{\rm{
Consider the curve $C:y^2-x^3=0$.
The Newton boundary consists of one face $\Delta$. 
Then $P=\sideset{^t}{}{\mathop{(}}2,3)$ 
is the weight vector correspond to the face $\Delta$ 
and $\Gamma^{*}(f;,x,y)$ has three vertices
$\{E_1,\, P, \, E_2\}$. 
We take a regular simplicial cone subdivision $\Sigma^{*}$
of $\Gamma^{*}(f;x,y)$
and we consider an admissible toric modification
$\pi:X \rightarrow \mathbb{C}^2$ 
associated to $\{\Sigma^*,(x,y),O\}$. 
Vertices of $\Sigma^{*}$ are $\{E_1 ,T_1, P,T_2, E_2 \}$
where $T_1={}^t(1,1)$ and $T_2={}^t(1,2)$ are new vertices
which are added to $\Gamma^{*}(f;x,y)$ in order to make $\Sigma^*$ regular. 

The proper transform $\widetilde{C}$ intersects transversely with 
the exceptional divisor $\hat{E}(P)$ at $1$-point.
Take a toric coordinate $(u,v)$ so that $u=0$ defines $\hat{E}(P)$ and 
let $\xi$ be the intersection point.
Then $\widetilde{C}$ is defined by $v-1=0$
and $\widetilde{C}$ is smooth at $\xi$.
$(\pi^{*}f)=2\hat{E}(T_1)+6\hat{E}(P)+3\hat{E}(T_2)+\widetilde{C}$. 
{\small{
\begin{figure}[H]
\centering
\unitlength 0.1in
\begin{picture}( 33.6000,  9.1000)(  2.0000,-10.6000)
\put(27.0000,-6.3000){\makebox(0,0)[lb]{\huge{$\bullet$}}}%
%
\special{pn 8}%
\special{pa 2760 550}%
\special{pa 3160 550}%
\special{fp}%
\put(31.0000,-6.3000){\makebox(0,0)[lb]{\huge{$\bullet$}}}%
%
\special{pn 8}%
\special{pa 3160 550}%
\special{pa 3560 550}%
\special{fp}%
\put(35.0000,-6.3000){\makebox(0,0)[lb]{\huge{$\bullet$}}}%
%
\special{pn 8}%
\special{pa 3170 540}%
\special{pa 3170 940}%
\special{dt 0.045}%
\special{sh 1}%
\special{pa 3170 940}%
\special{pa 3190 874}%
\special{pa 3170 888}%
\special{pa 3150 874}%
\special{pa 3170 940}%
\special{fp}%
\put(31.1000,-11.1000){\makebox(0,0)[lb]{$\bigcirc$}}%
\put(27.4000,-4.5000){\makebox(0,0)[lb]{2}}%
\put(31.4000,-4.5000){\makebox(0,0)[lb]{6}}%
\put(35.4000,-4.5000){\makebox(0,0)[lb]{3}}%
%
\special{pn 8}%
\special{pa 270 380}%
\special{pa 470 940}%
\special{fp}%
%
\special{pn 8}%
\special{pa 260 780}%
\special{pa 1200 780}%
\special{fp}%
%
\special{pn 8}%
\special{pa 1000 880}%
\special{pa 1170 470}%
\special{fp}%
\put(2.1000,-6.6000){\makebox(0,0)[lb]{2}}%
\put(11.6000,-7.1000){\makebox(0,0)[lb]{3}}%
\put(8.6000,-9.4000){\makebox(0,0)[lb]{6}}%
%
\special{pn 8}%
\special{pa 720 1060}%
\special{pa 720 260}%
\special{fp}%
\special{sh 1}%
\special{pa 720 260}%
\special{pa 700 328}%
\special{pa 720 314}%
\special{pa 740 328}%
\special{pa 720 260}%
\special{fp}%
\put(7.7000,-4.9000){\makebox(0,0)[lb]{$\widetilde{C}$}}%
\put(2.0000,-3.2000){\makebox(0,0)[lb]{$\hat{E}(T_1)$}}%
\put(12.3000,-4.8000){\makebox(0,0)[lb]{$\hat{E}(T_2)$}}%
\put(12.5000,-9.1000){\makebox(0,0)[lb]{$\hat{E}(P)$}}%
\put(20.6000,-6.8000){\makebox(0,0)[lb]{$\longleftrightarrow$}}%
\end{picture}%
\caption{}
\label{b-2}
\end{figure}
}}
}}
\end{example}
\section{Some lemmas for tours curves of type (p,q).}
\subsection{Notations}
Through out this paper, we use the same notations in \cite{Oka-milnor},
unless otherwise stated.
We also  use the fact that 
the topological equivalence class
of a non-degenerate germ depends only on 
its Newton boundary (Theorem 2.1 of  \cite{Oka-milnor}). 
In the process of classification of topological type of 
an inner singularity of torus curve of type (2.5),
we have the following possibilities:

(1) $(C,O)$ is  non-degenerate and $\Gamma(f)$ has one face,

(2) $(C,O)$ is  non-degenerate (in some local coordinate system)
but  the Newton boundary  has several faces, or

(3) $(C,O)$ has  some degenerate faces  for any choice of
coordinate systems $(x,y)$.

To make the expression of these singularity classes simpler, 
we introduce some notations.
The class of singularities (1) can be expressed by the class of
$B_{n,m}$:
\[
 B_{n,m}\quad : \quad x^n+y^m+(\text{higher terms})=0
\]

The class of singularities (2) can be understand
a union of singularities of type (1).
For example,
$x^n+x^2y^2+y^m=0$ is topologically equivalent to 
$(x^{n-2}+y^2)(x^2+y^{m-2})=0$ and  thus we denote this class by
$B_{n-2,2}\circ B_{2,m-2}$, as is already introduced in \S 1.

The last class (3) is most complicated.
For example, we consider the singularity germ
$(C,O)$ is defined by 
$f(x,y)=(\lambda x^3+y^2)^2+x^3y^3$ then 
$\mathcal{N}(f;x,y)=(\lambda x^3+y^2)^2$
and $\Gamma(f;x,y)$ consists of one face $\Delta$ 
with weight vector is $P={}^t(2,3)$ 
and $f$ is degenerate on $\Delta$.
We take a regular simplicial subdivision $\Sigma^{*}$
as in example \ref{exam-1}. 
Then we take a toric modification $\pi_1:X_1 \rightarrow \Bbb{C}^2$ 
we can see that $\widetilde{C_1}$ intersects transversely with $\hat{E}(P)$ at 
$\xi_1=(0,-\lambda)$ in toric coordinate $(u,v)$ 
correspond to Cone$(P, T_2)$ and 
$\hat{E}(P)=\{u=0\}$ with multiplicity $12$. (Figure \ref{figure-1}).
{\small{
\begin{figure}[H]
\centering
\unitlength 0.1in
\begin{picture}( 14.0800, 11.1000)(  0.6000,-12.9000)
%
\special{pn 8}%
\special{pa 286 428}%
\special{pa 526 1104}%
\special{fp}%
%
\special{pn 8}%
\special{pa 270 930}%
\special{pa 1406 930}%
\special{fp}%
%
\special{pn 8}%
\special{pa 1166 1032}%
\special{pa 1372 536}%
\special{fp}%
\put(2.1200,-7.6600){\makebox(0,0)[lb]{4}}%
\put(13.8000,-8.1000){\makebox(0,0)[lb]{6}}%
\put(10.0000,-11.2000){\makebox(0,0)[lb]{12}}%
\put(6.9000,-5.8000){\makebox(0,0)[lb]{$\widetilde{C}$}}%
\put(0.6000,-3.5000){\makebox(0,0)[lb]{$\hat{E}(T_1)$}}%
\put(14.4400,-5.4900){\makebox(0,0)[lb]{$\hat{E}(T_2)$}}%
\put(14.6800,-10.6800){\makebox(0,0)[lb]{$\hat{E}(P)$}}%
\special{pn 8}%
\special{pa 482 576}%
\special{pa 488 580}%
\special{pa 496 586}%
\special{pa 502 590}%
\special{pa 510 596}%
\special{pa 516 600}%
\special{pa 524 606}%
\special{pa 530 610}%
\special{pa 538 616}%
\special{pa 544 620}%
\special{pa 550 626}%
\special{pa 558 630}%
\special{pa 564 636}%
\special{pa 570 640}%
\special{pa 576 646}%
\special{pa 582 650}%
\special{pa 590 656}%
\special{pa 596 660}%
\special{pa 602 666}%
\special{pa 608 670}%
\special{pa 614 676}%
\special{pa 620 680}%
\special{pa 626 686}%
\special{pa 632 690}%
\special{pa 638 696}%
\special{pa 644 700}%
\special{pa 650 706}%
\special{pa 656 710}%
\special{pa 662 716}%
\special{pa 666 720}%
\special{pa 672 726}%
\special{pa 678 730}%
\special{pa 684 736}%
\special{pa 688 740}%
\special{pa 694 746}%
\special{pa 700 750}%
\special{pa 704 756}%
\special{pa 710 760}%
\special{pa 714 766}%
\special{pa 720 770}%
\special{pa 724 776}%
\special{pa 730 780}%
\special{pa 734 786}%
\special{pa 738 790}%
\special{pa 744 796}%
\special{pa 748 800}%
\special{pa 752 806}%
\special{pa 756 810}%
\special{pa 760 816}%
\special{pa 766 820}%
\special{pa 770 826}%
\special{pa 774 830}%
\special{pa 778 836}%
\special{pa 782 840}%
\special{pa 784 846}%
\special{pa 788 850}%
\special{pa 792 856}%
\special{pa 796 860}%
\special{pa 800 866}%
\special{pa 802 870}%
\special{pa 806 876}%
\special{pa 808 880}%
\special{pa 812 886}%
\special{pa 814 890}%
\special{pa 818 896}%
\special{pa 820 900}%
\special{pa 822 906}%
\special{pa 824 910}%
\special{pa 828 916}%
\special{pa 830 920}%
\special{pa 832 926}%
\special{pa 832 930}%
\special{pa 834 936}%
\special{pa 836 940}%
\special{pa 836 946}%
\special{sp}%
\special{pn 8}%
\special{pa 1188 576}%
\special{pa 1180 580}%
\special{pa 1174 586}%
\special{pa 1166 590}%
\special{pa 1160 596}%
\special{pa 1152 600}%
\special{pa 1146 606}%
\special{pa 1138 610}%
\special{pa 1132 616}%
\special{pa 1126 620}%
\special{pa 1118 626}%
\special{pa 1112 630}%
\special{pa 1106 636}%
\special{pa 1098 640}%
\special{pa 1092 646}%
\special{pa 1086 650}%
\special{pa 1080 656}%
\special{pa 1074 660}%
\special{pa 1068 666}%
\special{pa 1060 670}%
\special{pa 1054 676}%
\special{pa 1048 680}%
\special{pa 1042 686}%
\special{pa 1036 690}%
\special{pa 1030 696}%
\special{pa 1024 700}%
\special{pa 1020 706}%
\special{pa 1014 710}%
\special{pa 1008 716}%
\special{pa 1002 720}%
\special{pa 996 726}%
\special{pa 992 730}%
\special{pa 986 736}%
\special{pa 980 740}%
\special{pa 976 746}%
\special{pa 970 750}%
\special{pa 964 756}%
\special{pa 960 760}%
\special{pa 954 766}%
\special{pa 950 770}%
\special{pa 944 776}%
\special{pa 940 780}%
\special{pa 936 786}%
\special{pa 930 790}%
\special{pa 926 796}%
\special{pa 922 800}%
\special{pa 916 806}%
\special{pa 912 810}%
\special{pa 908 816}%
\special{pa 904 820}%
\special{pa 900 826}%
\special{pa 896 830}%
\special{pa 892 836}%
\special{pa 888 840}%
\special{pa 884 846}%
\special{pa 880 850}%
\special{pa 876 856}%
\special{pa 874 860}%
\special{pa 870 866}%
\special{pa 866 870}%
\special{pa 864 876}%
\special{pa 860 880}%
\special{pa 858 886}%
\special{pa 854 890}%
\special{pa 852 896}%
\special{pa 848 900}%
\special{pa 846 906}%
\special{pa 844 910}%
\special{pa 842 916}%
\special{pa 840 920}%
\special{pa 838 926}%
\special{pa 836 930}%
\special{pa 834 936}%
\special{pa 834 940}%
\special{pa 832 946}%
\special{sp}%
%
\special{pn 8}%
\special{pa 1094 636}%
\special{pa 1218 562}%
\special{fp}%
\special{sh 1}%
\special{pa 1218 562}%
\special{pa 1152 578}%
\special{pa 1172 588}%
\special{pa 1172 612}%
\special{pa 1218 562}%
\special{fp}%
\put(7.3000,-11.3000){\makebox(0,0)[lb]{$\xi_1$}}%
\end{picture}%
\caption{}
\label{figure-1}
\end{figure}
}}
To express the strict transform $\widetilde C$ at $\xi_1$,
we choose the coordinate
$(u,v_1)$ where $v_1=v+\lambda$,
we call this coordinates $(u,v_1)$ 
{\em the translated toric coordinates} 
for $(\widetilde C,\xi_1)$.
Now we found that $(\widetilde C, \xi_1)$ 
is 
defined by $B_{3,2}:v_1^2-\lambda u^3+(\text{higher terms})=0$
where $u=0$ defines the exceptional divisor which contains $\xi_1$. 
Observe that
$\hat{E}(P)$ is defined by $u=0$ and the tangent cone of $\widetilde{C}$ 
is $v_1^2=0$.
Thus the tangent cone is transverse to $\hat E(P)$ at $\xi_1$.
Again we take a toric blow-up $\pi_2: X_2\to (X_1,\xi_1)$.
This is essentially the same as the one for 
the cusp singularity $v_1^2-\lambda u^3=0$. 
As $\pi_1^{*}f(0)$ 
is non-degenerate in $(u,v_1)$, $\pi_2$ gives a good resolution of
$(\widetilde C, \xi_1)$ and therefore the composition
$\pi_1\circ\pi_2: X_2\to \Bbb C^2$
gives a good resolution of $(C,O)$.  
The resolution graph is simply obtained by adding a bamboo for this blowing
up (See (1), Figure \ref{e-1}).
We denote this class of singularity by 
$(B_{3,2}^2)^{B_{3,2}}$.

Sometime, we may need to take a coordinate change of the type 
$(u,v_2)$ where $v_2=v_1+h(u)$ for some polynomial $h(u)$. The important
point here is that we do not change the first coordinate $u$, as it
defines the exceptional divisor $\hat E(P)$. We call such a coordinate
system
$(u,v_2)$ {\em{admissible translated toric coordinates at $\xi_1$}}.


\begin{example}
{\rm{
Consider $f(x,y)=(y^5-xy^2+x^2y+x^5)^2+(x-2y)^5(x-3y)^5$.
Then $\mathcal{N}(f;x,y)=-31y^{10}-2xy^7+x^2y^2(x-y)+2x^7y+2x^{10}$ and 
$\Gamma(f;x,y)$ consists three faces $\Delta_i\ (i=1,2,3)$ with weight vectors
$P_1={}^t(3,1),\ P_2={}^t(1,1)$ and $P_3={}^t(1,3)$. 
Note that  $f(x,y)$ is degenerate on $\Delta(P_2;f)$.
By adding vertices $T_1={}^t(2,1)$ and $T_2={}^t(1,2)$, 
we get the canonical regular subdivision.
We take the associated  toric modification and we can see easily that
the strict transform $\widetilde{C}$
splits into three germs $\widetilde{C_1}$,
$\widetilde{C_2}$, $\widetilde{C_3}$ so that for $i=1,3$, $\widetilde{C_i}$,
intersects transversely with 
exceptional divisor $\hat{E}(P_i)$ at two points 
and  $\widetilde{C_i}$ are smooth at
these points. 
The germ $\widetilde{C_2}$ which intersects with $\hat{E}(P_2)$ is 
still singular and intersects with $\hat{E}(P_2)$ transversely at
$\xi_1=(0,1)$ in the toric coordinate $(u,v)$
corresponding to Cone$(P_2,T_2)$.
Thus taking the translated toric coordinate
$(u,v_1),\ v_1=v-1$ at $\xi_1$, we can write the defining equation of 
 $\widetilde{C_2}$ as
$v_1^2-4u^2v_1-239u^4+(\text{higher terms})=0$, while the exceptional divisor
$\hat E(P_2)$ is defined by $u=0$.
Note that $(\widetilde{C_2}, \xi_1)=B_{4,2}$.
Thus $\pi_1^*f$ is non-degenerate and
we need one more toric modification centered at $\xi_1$.
Then  the resolution is given by (2), Figure \ref{e-1}.  
We denote this class of singularity as 
$(C,O)\sim B_{6,2} \circ (B_{1,1}^2)^{B_{4,2}} \circ B_{2,6}$.
{\small{
\begin{figure}[H]
\centering
\unitlength 0.1in
\begin{picture}( 44.3000, 17.3000)(  1.9000,-18.2000)
%
\special{pn 8}%
\special{pa 4312 1526}%
\special{pa 4612 1612}%
\special{dt 0.045}%
\special{sh 1}%
\special{pa 4612 1612}%
\special{pa 4554 1574}%
\special{pa 4562 1598}%
\special{pa 4542 1614}%
\special{pa 4612 1612}%
\special{fp}%
%
\special{pn 8}%
\special{pa 4312 1526}%
\special{pa 4450 1808}%
\special{dt 0.045}%
\special{sh 1}%
\special{pa 4450 1808}%
\special{pa 4438 1738}%
\special{pa 4426 1760}%
\special{pa 4402 1756}%
\special{pa 4450 1808}%
\special{fp}%
\put(46.2000,-16.5100){\special{rt 0 0  0.7404}\makebox(0,0)[lb]{$\bigcirc$}}%
\special{rt 0 0 0}%
\put(44.1300,-18.7700){\special{rt 0 0  0.7319}\makebox(0,0)[lb]{$\bigcirc$}}%
\special{rt 0 0 0}%
\put(38.4100,-11.6500){\makebox(0,0)[lb]{\huge$\bullet$}}%
%
\special{pn 8}%
\special{pa 3896 1096}%
\special{pa 3896 900}%
\special{fp}%
\put(38.4100,-9.2300){\makebox(0,0)[lb]{\huge$\bullet$}}%
%
\special{pn 8}%
\special{pa 3896 852}%
\special{pa 3896 656}%
\special{fp}%
\put(38.4100,-6.8000){\makebox(0,0)[lb]{\huge$\bullet$}}%
%
\special{pn 8}%
\special{pa 3902 580}%
\special{pa 3766 298}%
\special{dt 0.045}%
\special{sh 1}%
\special{pa 3766 298}%
\special{pa 3776 366}%
\special{pa 3788 346}%
\special{pa 3812 348}%
\special{pa 3766 298}%
\special{fp}%
%
\special{pn 8}%
\special{pa 3902 578}%
\special{pa 4020 288}%
\special{dt 0.045}%
\special{sh 1}%
\special{pa 4020 288}%
\special{pa 3976 342}%
\special{pa 4000 336}%
\special{pa 4012 356}%
\special{pa 4020 288}%
\special{fp}%
%
\special{pn 8}%
\special{pa 3896 1104}%
\special{pa 4092 1298}%
\special{fp}%
\put(40.6000,-13.6900){\makebox(0,0)[lb]{\huge$\bullet$}}%
%
\special{pn 8}%
\special{pa 4092 1298}%
\special{pa 4288 1494}%
\special{fp}%
\put(42.5600,-15.6500){\makebox(0,0)[lb]{\huge$\bullet$}}%
%
\special{pn 8}%
\special{pa 3470 1534}%
\special{pa 3168 1620}%
\special{dt 0.045}%
\special{sh 1}%
\special{pa 3168 1620}%
\special{pa 3238 1622}%
\special{pa 3220 1606}%
\special{pa 3226 1582}%
\special{pa 3168 1620}%
\special{fp}%
%
\special{pn 8}%
\special{pa 3468 1534}%
\special{pa 3332 1814}%
\special{dt 0.045}%
\special{sh 1}%
\special{pa 3332 1814}%
\special{pa 3378 1764}%
\special{pa 3354 1766}%
\special{pa 3342 1746}%
\special{pa 3332 1814}%
\special{fp}%
\put(31.6000,-16.5900){\special{rt 0 0  5.5428}\makebox(0,0)[rb]{$\bigcirc$}}%
\special{rt 0 0 0}%
\put(33.6700,-18.8500){\special{rt 0 0  5.5513}\makebox(0,0)[rb]{$\bigcirc$}}%
\special{rt 0 0 0}%
%
\special{pn 8}%
\special{pa 3884 1112}%
\special{pa 3688 1306}%
\special{fp}%
\put(37.2000,-13.7700){\makebox(0,0)[rb]{\huge$\bullet$}}%
%
\special{pn 8}%
\special{pa 3688 1306}%
\special{pa 3494 1502}%
\special{fp}%
\put(35.2400,-15.7300){\makebox(0,0)[rb]{\huge$\bullet$}}%
\put(38.7200,-12.6700){\makebox(0,0)[lb]{6}}%
\put(35.9000,-12.2000){\makebox(0,0)[lb]{8}}%
\put(37.0800,-6.4800){\makebox(0,0)[lb]{10}}%
\put(33.4800,-14.2400){\makebox(0,0)[lb]{10}}%
\put(41.3900,-12.2000){\makebox(0,0)[lb]{8}}%
\put(43.5000,-14.2400){\makebox(0,0)[lb]{10}}%
\put(37.4700,-9.0700){\makebox(0,0)[lb]{8}}%
\put(39.9500,-2.6000){\makebox(0,0)[lb]{$\bigcirc$}}%
\put(36.6500,-2.6000){\makebox(0,0)[lb]{$\bigcirc$}}%
\put(1.9000,-13.9000){\makebox(0,0)[lb]{\huge{$\bullet$}}}%
%
\special{pn 8}%
\special{pa 250 1310}%
\special{pa 650 1310}%
\special{fp}%
\put(5.9000,-13.9000){\makebox(0,0)[lb]{\huge{$\bullet$}}}%
%
\special{pn 8}%
\special{pa 650 1310}%
\special{pa 1050 1310}%
\special{fp}%
\put(9.9000,-13.9000){\makebox(0,0)[lb]{\huge{$\bullet$}}}%
%
\special{pn 8}%
\special{pa 1100 890}%
\special{pa 1500 1090}%
\special{fp}%
%
\special{pn 8}%
\special{pa 650 1300}%
\special{pa 650 900}%
\special{fp}%
\put(5.8000,-9.4000){\makebox(0,0)[lb]{\huge$\bullet$}}%
\put(9.8000,-9.4000){\makebox(0,0)[lb]{\huge{$\bullet$}}}%
%
\special{pn 8}%
\special{pa 1040 860}%
\special{pa 640 860}%
\special{fp}%
\put(14.3000,-11.7000){\makebox(0,0)[lb]{\huge{$\bullet$}}}%
\put(14.4000,-7.0000){\makebox(0,0)[lb]{$\bigcirc$}}%
%
\special{pn 8}%
\special{pa 1100 830}%
\special{pa 1410 670}%
\special{dt 0.045}%
\special{sh 1}%
\special{pa 1410 670}%
\special{pa 1342 684}%
\special{pa 1364 694}%
\special{pa 1360 718}%
\special{pa 1410 670}%
\special{fp}%
\put(5.7000,-16.0000){\makebox(0,0)[lb]{12}}%
\put(1.9000,-16.0000){\makebox(0,0)[lb]{4}}%
\put(10.5000,-16.0000){\makebox(0,0)[lb]{6}}%
\put(5.7000,-7.5000){\makebox(0,0)[lb]{14}}%
\put(9.7000,-7.5000){\makebox(0,0)[lb]{30}}%
\put(14.5000,-13.2000){\makebox(0,0)[lb]{15}}%
\put(7.7000,-19.8000){\makebox(0,0)[lb]{(1)}}%
\put(38.3000,-19.9000){\makebox(0,0)[lb]{(2)}}%
\end{picture}%
\caption{}
\label{e-1}
\end{figure}
}}}}
\end{example}
\subsection{Some lemmas for general torus curves}
First we prepare some lemmas for the general torus curves of type $(p,q)$.
Let $C=\{f=0\}$ be a  curve of torus type $(p,q)$  
which can be written as $f=f_q^p+f_p^q$ where 
$f_p$ and $f_q$ are polynomial of degree 
$p$ and $q$ respectively.
We assume that $p\ge q\ge 2$.
We put $C_q:=\{f_q=0\}$ and $C_p:=\{f_p=0\}$.
Suppose that $O\in C_q\cap C_p$
and let $\iota$ be the 
local intersection multiplicity $I(C_{p},C_{q};O)$.
We recall a following key lemmas.  
\begin{lemma}[Lemma 1 of \cite{BenoitTu}]\label{Lemma-BT}
Suppose that 
$C_p$ is non-singular at $O$ and 
then the singularity $(C,O)$
is topologically equivalent to the Brieskorn-Pham singularity 
$B_{p\iota,q}$.
\end{lemma}
\begin{lemma}[Lemma 4.3 of \cite{cuspidal}]\label{cuspidal}
Suppose that $(C_p,O)$ is singular with  multiplicity $m$
and $(C_{q},O)$ is smooth.
If $C_q$ intersects transversely with $C_p$ and $p<qm$, 
the singularity $(C,O)$
is topologically equivalent to
$B_{mq,p}$.
\end{lemma}
\begin{lemma}\label{termination1}
Suppose that $(C_p,O)$ is singular with
 multiplicity $m$ and 
the tangent cone of $C_p$ consist of $m$ distinct lines.
Then $C_p$ consists of $m$ smooth components at $O$.
Consider the local factorization $f_p=\prod_{i=1}^m g_i$.
\begin{enumerate}
\item Suppose that $C_q$ is smooth at $O$ and $p < qm$. 
\begin{enumerate}
\item
If $\iota < \dfrac{p}{p-q}(m-1)$, then $ (C,O)\sim B_{q\iota,p}$.
\item
If $\iota > \dfrac{p}{p-q}(m-1)$, $(C,O)\sim B_{\beta,\ q}\circ
     B_{q(m-1),\ p-q}$ where $\beta=p(\iota-m+1)-q(m-1)$.
\item If $\iota= \dfrac{p}{p-q}(m-1)$ 
and the coefficients are generic,
then $ (C,O)\sim B_{q\iota,p}$.
\end{enumerate}
\item \label{casII-1}
Suppose that $m=2$, $C_q$ is smooth  and  $p>2q$. Then $(C,O)\sim B_{p(\iota-1)-q,q}\circ B_{q,p-q}$.
\item Suppose that 
$(C_q,O)$ is singular with multiplicity 2 and
the tangent cone $T_OC_q$ consists of 
distinct two lines  
and let $f_q=h_1h_2$ be the local factorization
and we put  $\ell_i=\{h_i=0\}$, $i=1,2$.
Let $\nu_{i}=I(\ell_i,C_p;O)$, $i=1,2$.
We assume that $g_1(x_2,y_2)=x_2,\,g_2(x_2,y_2)=y_2$ for a local
      coordinate system $(x_2,y_2)$ and
\[I(\ell_1,g_i;O)=1,\quad (i\ne 1) \quad I(\ell_2,g_j;O)=1,\quad (j\ne 2)\]
and put $\nu_1=I(\ell_1,g_1;O)$ and $\nu_2=I(\ell_2,g_2;O)$.
Assume that  $2p>qm$. Then  
\[(C,O)\sim B_{\beta_1,2}\circ 
{(B_{m-2,\ m-2}^q)}^{(m-2)B_{2p-qm,q}}
\circ B_{2,\beta_2},\quad \beta_i =p(\nu_i+1)-q(m-1)\]
\end{enumerate}
\end{lemma}
\begin{proof} 
By the assumption, 
the tangent cone of $C_p$ consists of 
$m$ distinct lines. 
This implies  $(C_p,O)$ has 
$m$ smooth components which are transverse each other.

First we consider the case (1).
First we take a local coordinate system $(x_1,y_1)$ so that
$C_q$ is defined by $ y_1=0$.  
Let $f_p(x_1,y_1)=\prod_{i=1}^{m}g_i(x_1,y_1)$ be the factorization in
$\mathcal{O}$ such that
$g_1(x_1,y_1)=y_1+\alpha_1\,x_1^\nu+{\text{(higher terms)}}$,
$g_i(x_1,y_1)=y_1-\alpha_ix_1+{\text{(higher terms)}}$ with 
$\alpha_i\ne 0$
$(2\le i \le m)$.
In this expression, we have $\iota=\nu+m-1$.
In the case of (a): $\iota<\frac{p}{p-q}(m-a)$, the Newton boundary of 
$f_p(x_1,y_1)^q,\,f_q(x_1,y_1)^p$ are  as the left hand side of Figure \ref{f-1}. Thus it is easy
 to see that 
$f(x_1,y_1)$ is non-degenerate in this coordinate and we have
$ (C,O)\sim B_{q\iota,p}$.

In the case of $(b): \iota >\frac{p}{p-q}(m-a)$, $f(x_1,y_1)$ are
 degenerate in this coordinate. We take another coordinate:
$(x_2,y_2)$ with $x_2=x_1,\,y_2=g_1(x_1,y_1)$ so that $y_2|f_p$. Then  
in this coordinate $(x_2,y_2)$, the Newton boundaries of $f_p(x_2,y_2)$
and $f_q(x_2,y_2)$ are given as the right hand side of Figure \ref{f-1}
 so that
$f(x_2,y_2)$ are now non-degenerate and the assertion follows.
\begin{figure}[H]
\centering
\unitlength 0.1in
\begin{picture}( 49.5200, 12.4700)(  5.5000,-16.1700)
%
\special{pn 8}%
\special{pa 956 432}%
\special{pa 956 1612}%
\special{fp}%
%
\special{pn 8}%
\special{pa 896 1496}%
\special{pa 2370 1496}%
\special{fp}%
%
\special{pn 8}%
\special{pa 932 1124}%
\special{pa 982 1124}%
\special{fp}%
%
\special{pn 8}%
\special{pa 1570 1524}%
\special{pa 1570 1462}%
\special{fp}%
%
\special{pn 8}%
\special{pa 930 910}%
\special{pa 982 910}%
\special{fp}%
%
\special{pn 8}%
\special{pa 930 496}%
\special{pa 982 496}%
\special{fp}%
\put(6.8000,-5.5000){\makebox(0,0)[lb]{$qm$}}%
\put(8.0000,-9.7000){\makebox(0,0)[lb]{$p$}}%
\put(8.0000,-11.8000){\makebox(0,0)[lb]{$q$}}%
\put(14.1000,-16.8000){\makebox(0,0)[lb]{$q(m-1)$}}%
\put(15.7100,-11.2400){\makebox(0,0){$\bullet$}}%
\put(9.5500,-4.9500){\makebox(0,0){$\bullet$}}%
%
\special{pn 8}%
\special{pa 956 488}%
\special{pa 1572 1116}%
\special{dt 0.045}%
\special{pa 1572 1116}%
\special{pa 2228 1500}%
\special{dt 0.045}%
%
\special{pn 8}%
\special{pa 2220 1520}%
\special{pa 2220 1456}%
\special{fp}%
\put(21.9200,-16.7100){\makebox(0,0)[lb]{$q\iota$}}%
\put(9.5500,-9.1200){\makebox(0,0){$\bullet$}}%
%
\special{pn 8}%
\special{pa 950 910}%
\special{pa 2220 1500}%
\special{fp}%
\put(22.1500,-14.9500){\makebox(0,0){$\bullet$}}%
\put(49.5200,-8.9600){\makebox(0,0)[lb]{: $\Gamma_{\ }(f;x_2,y_2)$}}%
%
\special{pn 8}%
\special{pa 4750 840}%
\special{pa 4910 840}%
\special{fp}%
\put(49.5000,-5.4000){\makebox(0,0)[lb]{: $\Gamma_{\ }(f_p^q;x_2,y_2)$}}%
%
\special{pn 8}%
\special{pa 4750 490}%
\special{pa 4910 490}%
\special{dt 0.045}%
\put(49.5000,-7.1200){\makebox(0,0)[lb]{: $\Gamma_{\ }(f_q^p;x_2,y_2)$}}%
%
\special{pn 8}%
\special{pa 4730 640}%
\special{pa 4920 640}%
\special{da 0.070}%
%
\special{pn 8}%
\special{pa 3424 438}%
\special{pa 3424 1618}%
\special{fp}%
%
\special{pn 8}%
\special{pa 3366 1500}%
\special{pa 5502 1500}%
\special{fp}%
%
\special{pn 8}%
\special{pa 3404 1128}%
\special{pa 3452 1128}%
\special{fp}%
%
\special{pn 8}%
\special{pa 4040 1530}%
\special{pa 4040 1468}%
\special{fp}%
%
\special{pn 8}%
\special{pa 3402 786}%
\special{pa 3452 786}%
\special{fp}%
%
\special{pn 8}%
\special{pa 3402 500}%
\special{pa 3452 500}%
\special{fp}%
\put(31.7100,-5.5000){\makebox(0,0)[lb]{$qm$}}%
\put(32.4100,-8.4000){\makebox(0,0)[lb]{$p$}}%
\put(32.4100,-11.8000){\makebox(0,0)[lb]{$q$}}%
\put(38.8100,-16.8300){\makebox(0,0)[lb]{$q(m-1)$}}%
\put(40.4100,-11.2800){\makebox(0,0){$\bullet$}}%
\put(34.2500,-4.9900){\makebox(0,0){$\bullet$}}%
%
\special{pn 8}%
\special{pa 3432 508}%
\special{pa 4048 1138}%
\special{dt 0.045}%
\special{pa 4048 1138}%
\special{pa 4048 1138}%
\special{dt 0.045}%
\special{pa 4048 1138}%
\special{pa 5288 1138}%
\special{dt 0.045}%
\put(34.2500,-7.8600){\makebox(0,0){$\bullet$}}%
%
\special{pn 8}%
\special{pa 3430 788}%
\special{pa 5286 1504}%
\special{da 0.070}%
%
\special{pn 8}%
\special{pa 5280 1528}%
\special{pa 5280 1464}%
\special{fp}%
\put(52.5600,-16.8000){\makebox(0,0)[lb]{$p\nu$}}%
\put(52.7900,-14.9800){\makebox(0,0){$\bullet$}}%
%
\special{pn 8}%
\special{pa 3432 790}%
\special{pa 4048 1132}%
\special{fp}%
\special{pa 4048 1132}%
\special{pa 5286 1504}%
\special{fp}%
\put(19.5000,-5.4000){\makebox(0,0)[lb]{: $\Gamma_{\ }(f_p^q;x_1,y_1)$}}%
%
\special{pn 8}%
\special{pa 1750 490}%
\special{pa 1908 490}%
\special{dt 0.045}%
\put(19.5000,-7.2000){\makebox(0,0)[lb]{: $\Gamma_{\ }(f\ ;x_1,y_1)$}}%
%
\special{pn 8}%
\special{pa 1750 660}%
\special{pa 1910 660}%
\special{fp}%
%
\special{pn 8}%
\special{pa 1570 1500}%
\special{pa 1570 1100}%
\special{dt 0.045}%
%
\special{pn 8}%
\special{pa 960 1120}%
\special{pa 1570 1120}%
\special{dt 0.045}%
%
\special{pn 8}%
\special{pa 4040 1500}%
\special{pa 4040 1130}%
\special{dt 0.045}%
\special{pa 4040 1130}%
\special{pa 3420 1130}%
\special{dt 0.045}%
\end{picture}%
\caption{}
\label{f-1}
\end{figure}
For the proof of (2), we take a local coordinate system so that 
$f_p(x,y)=c\,x_1y_1,\,c\ne 0$. Then we may assume that 
$f_q(x_1,y_1)=y_1+c'\, x^{\iota-1}+\text{(higher terms)}$.
The the assertion is immediate from  Newton boundary argument.

Next we consider the case (3).
We have chosen a local coordinate system $(x_2,y_2)$ so that
$g_1(x_2,y_2)=c\,x_2\,(c\ne 0)$ and $g_2(x_2,y_2)=y_2$. Put
$g_i(x_2,y_2)=y_2-\alpha_ix_2+{\text{(higher terms)}}$
$(3\le i \le m)$.
By the assumption, we can write
\[
 h_1(x_2,y_2)=x_2+d_1\,y_2^{\nu_1}+\text{(higher terms)},
\quad
 h_2(x_2,y_2)=y_2+d_2\,x_2^{\nu_2}+\text{(higher terms)}
\]
with $d_1,\,d_2\ne 0$.
Then the Newton principal part of $f$ can be written as: 
\[\mathcal{N}(f,x_2,y_2)=d_1^py_2^{p(\nu_1+1)}+
c^q\, x_2^qy_2^q\prod_{i=3}^{m}(y_2-\alpha_ix_2)^q
+d_2^px_2^{p(\nu_2+1)}.\]

Newton boundary of $f$ consists of three faces $\Delta_1$, $\Delta_2$ 
and $\Delta_3$ and $f$ is non-degenerate on $\Delta_1,\,\Delta_3$ but 
 degenerate on $\Delta_2$.
We take an admissible toric blowing-up $\pi:X\to \Bbb{C}^2$ with respect
 to
some regular simplicial cone subdivision $\Sigma^*=\{P_0,\dots, P_k\}$. 
Assume  that $P_\gamma={}^t(1,1)$, the weight vector of the homogeneous
 face
$\Delta_2$ and put $P_{\gamma+1}={}^t(n,n+1)$.
Then the pull-back of $f_p$ and $f_q$ to the coordinate chart 
$(\Bbb{C}^2_{\sigma}, (u_{\gamma},v_{\gamma}))$ with 
$\sigma=\text{\rm Cone}(P_\gamma,P_{\gamma+1})$ are given by  
\begin{equation*}
\begin{split}
\pi_{\sigma}^*f_p(u_\gamma,v_\gamma) &= u_{\gamma}^m\,
v_\gamma\prod_{i=3}^{m}\left\{g_i (u_\gamma v_\gamma^{n},u_\gamma
v_\gamma^{n+1})/u_{\gamma}\right \}\\
\pi_{\sigma}^*f_q(u_\gamma,v_\gamma) &= u_\gamma^2
(h_1(u_\gamma v_\gamma^{n},u_\gamma
v_\gamma^{n+1})/u_\gamma)\times (h_2(u_\gamma v_\gamma^{n},u_\gamma
v_\gamma^{n+1})/u_\gamma)
\end{split}
\end{equation*}
where
\begin{equation*}
\begin{split}
g_i (u_\gamma v_\gamma^{n},u_\gamma
v_\gamma^{n+1})/u_{\gamma}&\equiv v_\gamma-\alpha_i\,\,\mod\,
u_\gamma\\
h_1(u_\gamma v_\gamma^{n},u_\gamma
v_\gamma^{n+1})/u_\gamma&\equiv 1\,\,\mod\, u_\gamma\\
 h_2(u_\gamma v_\gamma^{n},u_\gamma
v_\gamma^{n+1})/u_\gamma&\equiv v_\gamma\,\,\mod\, u_\gamma.
\end{split}
\end{equation*}
Thus we get 
\begin{equation*}
\begin{split}
\pi_{\sigma}^{*}(f(u_\gamma,v_\gamma))&=u_\gamma^{qm}(\bar{f}_p^{q}(u_\gamma,
v_\gamma)+u_\gamma^{2p-qm}\bar{f}_q^{p}(u_\gamma,v_\gamma))\\
&=u_\gamma^{qm}\left(\prod_{i=3}^m(v_\gamma-\alpha_i)^q+u^{2p-qm}
\bar{f}_q(u_\gamma,v_\gamma)+\text{(higher terms)}
\right)
\end{split}
\end{equation*}
and $\bar f_q(\alpha_i,0)\ne 0$.
We put $\xi_i=(0,\alpha_i)$ and
we take a local coordinates $(u_{\gamma},v_{\gamma, i})$ at $\xi_i$ with 
$v_{\gamma, i}=v_{\gamma}-\alpha_i$.
6The above expression implies  $(\widetilde{C},\xi_i)\sim B_{2p-qm,q}$. 
\end{proof}
The assertion (1) of Lemma \ref{termination1} can be generalized 
for the case where the tangent cone of $C_p$ may have some factors 
with multiplicity as follows.
\begin{lemma}\label{termination4}
Suppose that $C_p$ is singular at $O$ and
$C_q$ is smooth at $O$.
Let $m$ be the multiplicity of $(C_p,O)$. 
We take a local coordinate system $(x_1,y_1)$ so that $C_q$
is defined by $y_1=0$. 
We assume that $p< qm$ and 

--either $y_1 \not | \ (f_p)_m $ (this implies that $C_q$ intersects
 transversely with $T_OC_p$),  or 

--$y_1=0$ is a simple tangent line of  $T_OC_p$  (this is equivalent to
 $y_1|(f_p)_m,\, y_1^2 \not | \ (f_p)_m$ ).

where $(f_p)_m$ is the homogeneous part of degree $m$ of $f_p$.

Then we have:
\begin{enumerate}
\item
If $\iota< \dfrac{p}{p-q}(m-1)$, then $ (C,O)\sim B_{q\iota,p}$.
\item
If $\iota > \dfrac{p}{p-q}(m-1)$, $(C,O)\sim B_{\beta,\ q}\circ
     B_{q(m-1),\ p-q}$ where $\beta=p(\iota-m+1)-q(m-1)$.
\item If $\iota= \dfrac{p}{p-q}(m-1)$ and the coefficients are 
generic, then $ (C,O)\sim B_{q\iota,p}$.
\end{enumerate}
\end{lemma}
\begin{proof}
The proof is completely parallel to that of Lemma \ref{termination1}.
\end{proof} 

\section{Calculation of the local singularities.}
We come back to our original situation
of a torus curve of type $(2,5)$:
\[
 C:\, f_5(x,y)^2+f_2(x,y)^5=0,\quad C_5:  \, f_5(x,y)=0,
\,\, C_2:\,f_2(x,y)=0.
\]
First we consider the case that $C_2$ is reduced  in section 4 and 5.
Next consider the case of 
$C$ being a linear torus curve in section 6.   
For the classification, we start from the following generic equations:
\[
 f_2(x,y)=\sum_{i+j\le 2}a_{ij}x^i y^j,\quad
f_5(x,y)=\sum_{i+j\le 5} b_{ij}x^i y^j.
\]
Hereafter  $x,y$ are the affine coordinates $x=X/Z,\,y=Y/Z$ 
on $\Bbb{C}^2:=\Bbb{P}^2\setminus\{Z=0\}$.
As we assume that $C_2,\,C_5$ pass through the origin, we have
$a_{00}=b_{00}=0$.
We study the inner singularity $O\in C_{2}\cap C_{5}$.
We denote hereafter the multiplicities of 
$C_2$ and $C_5$ at the origin $O$ by $m_2$ and $m_5$ respectively
and the intersection multiplicity $I(C_2,C_5;O)$
of $C_2$ and $C_5$  at $O$ by $\iota$.
By  B\'ezout theorem,  
we have  the inequalities: 
\[1 \leq \iota \leq 10,\quad m_2m_5 \leq \iota\]

For the classification of possible topological 
types of the singularity $(C,O)$, 
we divide the situations into the 5 cases,
corresponding to the values  of $m_5$.
Then for a fixed $m_5$, 
we consider the subcases, 
corresponding to $\iota, \, m_5\le  \iota \le 10$
taking the  geometry of
the intersection of $C_2$ and $C_5$ at $O$ into account. 
And each case has several sub-cases 
by type of the singularity of $(C_5,O)$.
\begin{enumerate}
\item Case I. $m_5=1$. The quintic $C_5$ is smooth. 
\item Case II. $m_5=2$. We divide this case into two subcases  (a) and
      (b) by type of tangent cone of $C_5$.\\
\indent
(a) The tangent cone of $C_5$ is distinct two lines i.e., $(C_5,O)\sim
      A_1$. \\
\indent
(b) The tangent cone of $C_5$ is a single line with multiplicity 2.\\
\item Case III. $m_5=3$. We divide this case into three subcases by type
      of tangent cone of $C_{5}$.\\
\indent
(a) The tangent cone of $C_5$ is distinct three lines. \\
\indent
(b) The tangent cone of $C_5$ is a double line and another line. \\
 \indent
(c) The tangent cone of $C_5$ is a single line with multiplicity 3. \\ 

\item Case IV. $m_5=4$. We divide this case into five subcases by type of tangent cone of $C_{5}$.

(a) The tangent cone of $C_5$ is distinct four lines. \\
\indent
(b) The tangent cone of $C_5$ is a double line  and distinct two
      line. \\
\indent
(c) The tangent cone of $C_5$ is a triple line and another line.\\
\indent
(d) The tangent cone of $C_5$ is a  single line with multiplicity 4. \\
\indent
(e) The tangent cone of $C_5$ is two double line.\\ 
\item Case V. $m_5=5$.
The quintic $C_5$ is five lines.
\end{enumerate}
\subsection{Case I: $m_{5} = 1$.} Quintic $C_5$ is smooth.
 This case is  determined by  Lemma \ref{Lemma-BT} as follows.
\begin{proposition}\label{caseI}
Suppose that $C_5$ is smooth at $O$ and let
$\iota=I(C_5,C_2;O)$ be the local intersection multiplicity.
Then $(C,O) \sim B_{5\iota,2}$
 for $\iota=1, \cdots, 10$.
\end{proposition}

\subsection{Case II: $m_{5}$ = 2.}
We divide Case II into two subcases $(a)$ and $(b)$ by type of tangent cone 
$T_{O}C_5$ of $C_5$.\\
(a) $T_{O}C_5$ consists  of distinct two lines i.e., $(C_5,O)\sim A_1$. \\
(b) $T_{O}C_5$  consists of a line with multiplicity 2.
\subsubsection{Case II-(a).} We consider the subcase  (a).
In this case, we assume that the tangent cone of $C_5$
is given $xy=0$, $f_5(x,y)=xy+\text{(higher terms)}$.
\begin{proposition}\label{caseII-1}
Under the situation in $(a)$, we have the following  possibilities.
\begin{enumerate}
\item We assume that the conic is smooth and 
the tangent cone of $C_{2}$ is $y=0$
($a_{10}=0,\,a_{01}\ne 0$) for simplicity. 
The possibility of $(C,O)$ are $B_{5\iota-7,2}\circ B_{2,3}$ for
      $2\le\iota\le 10$. 
\item For a conic which consists of two lines $\ell_1,\,\ell_2$
 (i.e., $a_{01}=a_{10}=0$),
the generic singularity is  $B_{8,2}\circ B_{2,8}$. 
Further degeneration occurs when these lines are tangent to one or
both tangent cones of $C_5$. 
Put $\iota_i=I(\ell_i,C_5;O)$ $(i=1,2)$. Then 
$4\le\iota_1+\iota_2\le10$ and the corresponding singularity is
$B_{5\iota_2-2,2}\circ B_{2,5\iota_1-2}$.
\end{enumerate}
\end{proposition}
\begin{proof}
Both assertions are immediate from Lemma \ref{casII-1} and Lemma \ref{termination1}
\end{proof}
\subsubsection{Case II-(b)} Now we consider the case (II-b).
So we assume that $m_5=2$ and  the tangent cone of $C_5$ is given by
 $L:y^2=0$
(with multiplicity $2$).
We divide this subcase  (II-b) into two subcases: 
(b-1) $m_2=1$ and (b-2) $m_2=2$.\\\\
\noindent {\bf{(b-1)}} 
Assume that $m_5=2,\,m_2=1$ and the tangent cone of $C_5$ is $y^2=0$.
\begin{proposition}\label{caseII-2}
Suppose that the tangent cone of $C_5$ 
is a line with multiplicity 2 and 
$C_2$ is smooth.
Then we have the following  possibilities.
\begin{enumerate}
\item If $C_2$ and $C_5$ are transverse at $O$ $(\iota =2)$, 
then $(C,O)\sim B_{5,4}$.
\item If $(C_5,O)\sim B_{3,2}$ and $\iota=3$, we have $(C,O)\sim (B_{3,2}^2)^{B_{3,2}}$.
\item If $(C_5,O)\sim B_{4,2}$, 
      we have
$(C,O)\sim(B_{4,2}^2)^{ (B_{5\iota-18,2}+B_{2,2})}$ for $\iota
      =4,\cdots,10$.
      (Here the upper $(B_{5\iota-18,2}+B_{2,2})$ implies we have
       two non-degenerate singularities $B_{5\iota-18,2},\,B_{2,2}$
       sitting on two different points on the exceptional
       divisor $\widehat E(P)$, $P={}^t(1,2)$.
      after one toric modification.)
\item If $(C_5,O)\sim B_{5,2}$, 
      we have
\begin{enumerate}
\item $(C,O)\sim B_{10,4}$ generically and 
      $B_{k,2}\circ B_{5,2},\,6\le k\le 15$ for $\iota =4$,
      
\item ${(B_{5,2}^2)}^{B_{5,2}}$ for $\iota =5$.
      \end{enumerate}
\item If $(C_5,O)\sim B_{6,2}$, we have
      \begin{enumerate}
      \item $(C,O)\sim B_{10,4}$ for $\iota =4$ and 
      \item $(C,O)\sim {(B_{6,2}^2)}^{(B_{5\iota-27,2}+B_{3,2})}$
      for $\iota=6,\dots, 10$.
      \end{enumerate}
      
\item If $(C_5,O)\sim B_{7,2}$, we have 
     \begin{enumerate}
     \item $(C,O)\sim B_{10,4}$ for $\iota =4$ and 
     \item $(C,O)\sim {(B_{7,2}^2)}^{B_{5\iota-28,2}}$ for $\iota=6,7$.
     \end{enumerate}
     
\item If $(C_5,O)\sim B_{8,2}$, we have 
    \begin{enumerate}
    \item  $(C,O)\sim B_{10,4}$ for $\iota =4$,
    \item  $(C,O)\sim B_{15,4}$ for $\iota=6$,
    \item  ${(B_{8,2}^2)}^{(B_{5\iota-36,2}+B_{4,2})}$
                                    for $\iota=8,9,10$.
    \end{enumerate}
\item If $(C_5,O)\sim B_{9,2}$, we have 
    \begin{enumerate}
    \item  $(C,O)\sim B_{10,4}$ for $\iota =4$,
    \item  $(C,O)\sim B_{15,4}$ for $\iota=6$,
    \item  $(C,O)\sim {(B_{9,2}^2)}^{ B_{5\iota-35,2}}$ for $\iota=8,9$.
    \end{enumerate}
\item If $(C_5,O)\sim B_{10,2}$, we have 
     \begin{enumerate}
     \item $(C,O)\sim B_{10,4}$ for $\iota =4$,
     \item $(C,O)\sim B_{15,4}$ for $\iota=6$,
     \item $(C,O)\sim B_{20,4}$ and $B_{k,2}\circ B_{10,2}$ $(k=11,12)$
                                                         for  $\iota=8$,
     \item ${(B_{10,2}^2)}^{2B_{5,2}}$ for  $\iota=10$.
     \end{enumerate}
\item If $(C_5,O)\sim B_{11,2}$, we have
     \begin{enumerate}
     \item $(C,O)\sim B_{10,4}$ for $\iota =4$,
     \item $(C,O)\sim B_{15,4}$ for $\iota=6$,
     \item $(C,O)\sim B_{20,4}$ for $\iota=8$,
     \item $(C,O)\sim {(B_{11,2}^2)}^{ B_{6,2}}$ for $\iota=10$.
     \end{enumerate}    
\item If $(C_5,O)\sim B_{12,2}$, we have
      \begin{enumerate}
      \item $(C,O)\sim B_{10,4}$ for $\iota =4$,
      \item $(C,O)\sim B_{15,4}$ for $\iota=6$,
      \item $(C,O)\sim B_{20,4}$ for $\iota=8$,
      \item $(C,O)\sim {(B_{12,2}^2)}^{2{{B}}_{1,2}}$ for $\iota=10$.
      \end{enumerate}
\item If $(C_5,O)\sim B_{13,2}$, we have
      \begin{enumerate}
           \item 
      $(C,O)\sim B_{10,4}$ for $\iota =4$,
           \item
      $(C,O)\sim B_{15,4}$ for $\iota=6$,
      \item $(C,O)\sim B_{20,4}$ for $\iota=8$,
           \item
      $(C,O)\sim B_{25,4}$ for $\iota=10$.
      \end{enumerate}
\end{enumerate}
\end{proposition}
\begin{proof}
We can proceed the classification mainly using 
the local intersection multiplicity $\iota=I(C_2,C_5;O)$
and geometry of $C_2$ and $C_5$. 
By the assumption $m_5=2$,
we have $\iota \geq 2$. 
When $\iota=2$,
$C_2$ intersects transversely with $C_5$ at the origin and 
then $(C,O)\sim B_{5,4}$.  Thus hereafter we assume that $\iota\ge 3$.

Assume that $(C_5,O)\sim A_{\ell-1}$. Then 
 by a triangular change of local coordinates $(x,y_1)$ where
$y_1=y+c_2x^2+\dots+c_{k-1}x^{k-1}$, we can write $f_5$ as
\[
 f_5(x,y_1)=\alpha y_1^2+\beta x^{\ell}+\text{(higher terms)},\quad
\ell\ge 2(k-1),\,\alpha,\beta\ne 0.
\]
 A simple computation
shows that $\ell\le 13$.
Now we can write $f_2$ in this coordinates as
\[
 f_2(x,y_1)=\gamma y_1+\delta x^{\nu}+\text{(higher terms)},\,\, \nu\ge
 2,\,\gamma,\delta\ne 0.
\]
(Here $\delta=0$ if $\nu=\infty$, i.e., $y_1|f_2$.)
First we notice that

$(\star):\quad \iota\ge\min\,(2\nu,\,\ell)$ and the equality holds except the case
$\ell=2\nu$ and
$\alpha \delta^2+\beta\gamma^2=0$.

\vspace{.3cm}
We observe that

1. If $5\nu<2\ell$, $f$ is non-degenerate in this coordinates and
$(C,O)\sim B_{5\nu,4}$.

2. If $5\nu = 2\ell$ (in this situation, 
the possible pairs which satisfy this condition are
 $(\nu, \ell)=(2,5),\ (4,10)$),
we have 
$\mathcal{N}(f,x,y)=\alpha^2y_1^4+2\alpha \beta x^{\ell}y_1^2+(\beta^2+\delta^5)x^{2\ell}$ and 
if $(\beta^2+\delta^5)\ne0$, we have $(C,O)\sim B_{2\ell ,4}$.

If $(\beta^2+\delta^5)= 0$, we can see that the Newton boundary has 
two faces with
$R=(\ell,2)$ as the common vertex of the faces
and $f$ can be non-degenerate on these faces,
after taking  a suitable triangular change of coordinates
$(x,y_2)$. 

3. If  $5\nu>2\ell$, it is easy to see that
the Newton principal part of $f(x,y_1)$ is given by 
$(\alpha y_1^2+\beta x^{\ell})^2$ which implies that $f(x,y_1)$ is
degenerate
in this coordinate. So we need to take first a toric modification
$\pi:X\to {\Bbb C}^2$ with respect to the canonical regular simplicial
cone subdivision $\{P_0,\dots,P_m\}$ and we
have to see the total transform  equation in $X$.

The weight vector of $A_{\ell-1}$ is given as
$P={}^t(2,\ell),\,{}^t(1,\ell/2)$ for  $\ell$ is odd or even
respectively.
Note also the germ $A_{\ell-1}$ has two smooth components if $\ell$ is even.
Thus the description for the toric modification has to be divided
 in two cases.

{\bf Case 1}. $\ell$ is odd. The weight vector of $\Gamma(f_5;x,y_1)$ is
given by 
$P={}^t(2,\ell)$. We may assume that 
$P=P_s$ and we consider the cone
$\sigma:={\rm Cone}(P_s,P_{s+1})$
and we consider the corresponding unimodular matrix
\[
 \sigma:=\left(
\begin{matrix}
2&a\\
\ell&b\end{matrix}\right),\quad 2b-a\ell =1
\]
Let $(u,v)$ be the toric coordinates of this chart. Then in this
coordinates,
we have $x=u^2v^a,\, y_1=u^\ell v^b$.
Then we can write
\begin{eqnarray*}
&\pi^*f_5(u,v)=u^{2\ell}v^{a\ell}\tilde f_5(u,v),\quad
\tilde f_5(u,v)=\alpha v+\beta+ h_5(u,v)\\
&\pi^* f_2(u,v)=u^{\mu}v^{\mu'}\tilde f_2(u,v),\,\mu=\min(\ell,2\nu),\,\mu'=\min(b,a\nu)
\end{eqnarray*}
Putting $\xi=(0,-\beta/\alpha),\, \eta=5\mu-4\ell$,
 we can write as
\begin{eqnarray*}
& \pi^*f(u,v)=u^{4\ell}v^{2a\ell}\tilde  f(u,v)\\
&\tilde  f(u,v)=
\tilde f_5(u,v)^2+u^{\eta}v^{5\mu'-2a\ell}
\tilde f_2(u,v)^5
\end{eqnarray*}
Thus using  admissible translated toric coordinates $(u,v_2)$,
$v_2=v_1+h(u),\, v_1=v+\beta/\alpha$ for some polynomial $h$,
the strict transform is defined as
\[
 \alpha^2 v_2^2+\epsilon u^{\eta'}+\text{(higher terms)}=0,\,\epsilon \ne 0
\]
which implies $(\tilde C,\xi)\sim B_{\eta',2}$ and the tangent cone is
transverse to the exceptional divisor $u=0$ where $\eta'\ge\eta$.

{\bf Case 2}. $\ell$ is even. 
The weight vector of $\Gamma(f_5;x,y_1)$ is given by 
$P={}^t(1,k)$ where $\ell =2k$. 
We may consider that there is a cone corresponding to
a unimodular matrix
\[
 \sigma:=\left(
\begin{matrix}
1&0\\
k&1\end{matrix}\right)
\]
Let $(u,v)$ be the toric coordinates of this chart. Then in this
coordinates,
we have $x=u,\, y_1=u^k v$.
Then we can write
\begin{eqnarray*}
&\pi^* f_5(u,v)=u^{\ell}\tilde f_5(u,v),\quad
\tilde f_5(u,v)=\alpha(v+\alpha_1)(v+\alpha_2) +uv h_5(u,v)\\
&\pi^* f_2(u,v)=u^{\mu}\tilde f_2(u,v),\,\mu=\min(k,\nu)
\end{eqnarray*}
Putting $\xi_i=(0,\alpha_i),\, (\i=1,2), \
\eta=5\mu-4\ell$, we can write as
\[
 \pi^*f(u,v)=u^{2\ell}\left(
\tilde f_5(u,v)^2+u^{\eta}\tilde f_2(u,v)^5
\right)
\]
Then the strict transform $\widetilde{C}$ has two components.
Thus using admissible translated toric coordinates $(u,v_i')$,
$v_i'=v_i+h(u),\, v_i=v+\alpha_i$ $(i=1,2)$ for some polynomial $h$,
 the total transform $\pi^{*}f$ 
is described as 
\[
\pi^{*}f(u,v_i')=u^{2\ell}\Big{(}\alpha^2(\alpha_1-\alpha_2)^2 v_i'^2+\epsilon u^{\eta'}+\text{(higher terms)}\Big{)}=0,\ \epsilon\ne 0
\]
which implies $(\tilde C,\xi_i)\sim B_{\eta',2}$ where $\eta'\ge \eta$.

 Putting the above strategy into consideration, we will explain
 some more detail for several cases.

First we consider Case (2): $(C_5,O)\sim B_{3,2},\,\iota=3$
 ($\ell=3,\,\nu\ge 2$).
 
We have  $f_5(x,y)=\alpha\,y^2+\beta x^3+ (\text{higher terms})$. 
 We have to consider the toric modification in the toric coordinate chart:
\[
 \sigma=\left(\begin{matrix}
2&1\\
3&2
\end{matrix}\right),\,\, \pi(u,v)=(u^2v,u^3v^2)
\] in above observation.
Then taking the translated coordinates $(u,v_1),\, v_1=v+\beta/\alpha$, we
 have
\begin{eqnarray*}
&\pi^*f_5(u,v_1)=u^6(v_1-\beta/\alpha)^3(\alpha v_1+ c u+ h_5(u,v_1))\\
&\pi^*f_2(u,v_1)=u^3(v_1-\beta/\alpha)^2(\delta u+\gamma+h_2(u,v_1))\\
&\pi^*f(u,v_1)=u^{12}(v_1-\beta/\alpha)^6 \tilde f(u,v_1)\\
&\tilde f(u,v_1)=(\alpha v_1+ c u+ h_5(u,v_1))^2+u^3 (v_1-\beta/\alpha)^4(\delta u+\gamma+h_2(u,v_1))^5
\end{eqnarray*}
Now we can see that
\[
 \tilde f(u,v_2)=v_2^2+c' u^3+\text{(higher terms)}, \quad c'\ne0
\quad
v_2=\alpha v_1+cu
\]
which implies the corresponding singularity is 
$(B_{3,2}^2)^{B_{3,2}}$.

Next we consider the case (3): $(C_5,O)\sim B_{4,2}$ and $\iota\ge 4$.
Thus
 \begin{eqnarray*}
& f_5(x,y_1)=\alpha\,y_1^2+\beta\, x^4+ (\text{higher terms})\\
& f_2(x,y_1)=\gamma y_1+\delta x^\nu+ (\text{higher terms}),\, \nu\ge 2.
 \end{eqnarray*}
 Note that $\nu\ge 2$ and the case $\iota>4$ only if $\nu=2$ and
 $\alpha\,\delta^2+\beta\,\gamma^2=0$. Thus for simplicity, we assume
 that $\nu=2$. For the simplicity of the calculation, we put:
 \begin{eqnarray*}
&  f_5(x,y_1)=\alpha y_1^2+{\beta}x^4+\text{(higher terms)}
   =\alpha ( y_1 +\alpha_1 x)(y_1 +\alpha_2 x)+\text{(higher terms)},
 \\
& f_2(x,y_1)=\gamma y_1+\delta x^2+\text{(higher terms)}).
 \end{eqnarray*}
 The corresponding toric chart is associated with:
\[
 \sigma=\left(\begin{matrix}
1&0\\
2&1
\end{matrix}\right),\,\, \pi(u,v)=(u,u^2v)
\] by the above consideration.
Note that $\iota=4$ if and only if $\alpha\,\delta^2+\beta\,\gamma^2\ne 0$.
Then taking the translated coordinates 
$(u,v_1),\, v_1=v+{\alpha_1}$
(respectively $(u,v_2), \,v_2=v+{\alpha_1})$, we
 have
\begin{eqnarray*}
&\pi^*f_5(u,v_1)=u^4(\alpha(\alpha_1-\alpha_2)v_1+c_1u+ h_5(u,v_1)),\\
&\pi^*f_2(u,v_1)=u^2{((\gamma v_1-\gamma \alpha_1+\delta) +uh_2(u,v_1))},\\
&\pi^*f(u,v_1)=u^{8} \tilde f(u,v_1),\quad
(resp.\ \pi^*f(u,v_2)=u^{8} \tilde f(u,v_2)) \\
&\tilde f(u,v_1)=((\alpha(\alpha_1-\alpha_2)v_1+c_1u)^2+
{(\delta-\gamma \alpha_1)^5}\,u^2+
\text{(higher terms)}\\
& (resp.\ \tilde f(u,v_2)=
(\alpha(\alpha_2-\alpha_1)v_1+c_2 u)^2+(\delta-\gamma \alpha_2)^5\,u^2+\text{(higher terms)})
\end{eqnarray*}
where $c_i$  $(i=1,2)$ are constant. 
Then if $\iota=4$, we have $\alpha\,\delta^2+\beta\,\gamma^2\ne 0$ and we see that 
$(\widetilde C,\xi_i)=A_1$ for $i=1,2$.
Thus 
$(C,O)\sim (B_{4,2}^2)^{2B_{2,2}}$  and resolution graph is given by Figure \ref{iota=4}.
{\small{
\begin{figure}[H]
\centering
\unitlength 0.1in
\begin{picture}( 13.3200, 11.9400)(  6.6000,-14.2300)
\put(6.6000,-10.8000){\makebox(0,0)[lb]{\huge{$\bullet$}}}%
%
\special{pn 8}%
\special{pa 720 1000}%
\special{pa 1120 1000}%
\special{fp}%
\put(10.6000,-10.8000){\makebox(0,0)[lb]{\huge{$\bullet$}}}%
%
\special{pn 8}%
\special{pa 1120 1000}%
\special{pa 1520 800}%
\special{fp}%
\put(14.5000,-8.7000){\makebox(0,0)[lb]{\huge$\bullet$}}%
\put(14.5000,-12.7000){\makebox(0,0)[lb]{\huge$\bullet$}}%
%
\special{pn 8}%
\special{pa 1120 1000}%
\special{pa 1520 1200}%
\special{fp}%
%
\special{pn 8}%
\special{pa 1524 788}%
\special{pa 1842 546}%
\special{dt 0.045}%
\special{sh 1}%
\special{pa 1842 546}%
\special{pa 1778 570}%
\special{pa 1800 578}%
\special{pa 1802 602}%
\special{pa 1842 546}%
\special{fp}%
%
\special{pn 8}%
\special{pa 1526 788}%
\special{pa 1918 862}%
\special{dt 0.045}%
\special{sh 1}%
\special{pa 1918 862}%
\special{pa 1856 830}%
\special{pa 1866 852}%
\special{pa 1850 868}%
\special{pa 1918 862}%
\special{fp}%
\put(18.8800,-5.6700){\special{rt 0 0  6.0917}\makebox(0,0)[lb]{$\bigcirc$}}%
\special{rt 0 0 0}%
\put(19.6400,-9.5100){\special{rt 0 0  6.0838}\makebox(0,0)[lb]{$\bigcirc$}}%
\special{rt 0 0 0}%
%
\special{pn 8}%
\special{pa 1570 1206}%
\special{pa 1958 1102}%
\special{dt 0.045}%
\special{sh 1}%
\special{pa 1958 1102}%
\special{pa 1888 1100}%
\special{pa 1906 1116}%
\special{pa 1898 1140}%
\special{pa 1958 1102}%
\special{fp}%
%
\special{pn 8}%
\special{pa 1574 1206}%
\special{pa 1908 1424}%
\special{dt 0.045}%
\special{sh 1}%
\special{pa 1908 1424}%
\special{pa 1864 1370}%
\special{pa 1864 1394}%
\special{pa 1842 1404}%
\special{pa 1908 1424}%
\special{fp}%
\put(19.9200,-11.4000){\special{rt 0 0  0.1799}\makebox(0,0)[lb]{$\bigcirc$}}%
\special{rt 0 0 0}%
\put(19.1600,-15.2300){\special{rt 0 0  0.1915}\makebox(0,0)[lb]{$\bigcirc$}}%
\special{rt 0 0 0}%
\put(10.9000,-9.0000){\makebox(0,0)[lb]{8}}%
\put(6.9000,-9.0000){\makebox(0,0)[lb]{4}}%
\put(14.6000,-7.0000){\makebox(0,0)[lb]{10}}%
\put(14.6000,-14.3000){\makebox(0,0)[lb]{10}}%
\end{picture}%
\caption{}
\label{iota=4}
\end{figure}
}}
The case $\nu>2$ gives the same conclusion as above.
 
If $\nu=2,\, \iota>4$ and assuming
 $\alpha\,\delta^2+\beta\,\gamma^2\ne 0$, we have 
 $(\widetilde C,\xi_2)=A_1$ but  $(\widetilde C,\xi_1)$
 is bigger that $A_1$. Thus we have to take a triangular change of coordinates
$(u,v_1')$ so that $\widetilde C$ is defined at $\xi_1$ as 
$(v_1')^2+c'\, u^k+\text{(higher terms)}=0$.
The explicit computation shows that the possible $k$ are
$5\iota-18,\, 4\le \iota\le 10$.

Next we consider the assertion (4): 
$(C_5,O)\sim B_{5,2}$ $\iota\ge 4$. We put as above
\[f_5(x,y_1)= \alpha y_1^2+{\beta} x^5+{\text{(higher terms)}},\quad 
f_2(x,y_1)=\gamma y_1+{\delta} x^\nu+{\text{(higher terms)}}.\] 
Note that $\iota=4$ if and only if $\nu=2$.
 If $\nu>2$, $\iota=5$ and  $\mathcal{N}(f,x,y_1)=(\alpha y_1^2+{\beta} x^5)^2$ and we have to
 take a toric modification. If $\nu=2$, then $\iota=4$ and 
as 
$\mathcal{N}(f,x,y_1)=y_1^4+2{\beta} x^5y_1^2+({\beta}^2+{\delta}^5)x^{10}$,
 we see that  $(C,O)\sim B_{10,4}$ if $\iota=4$ and ${\beta}^2+{\delta}^5\ne0$.
 If ${\beta}^2+{\delta}^5=0$, the Newton boundary has two faces.

--First we consider the case $\iota=4$ (so ${\delta}\ne 0$) and  ${\beta}^2+{\delta}^5=0$.
Then 
$\mathcal{N}(f,x,y_1)=y_1^4+2{\beta} x^5y_1^2+\gamma_{8}x^8y_1+\gamma_{11} x^{11}$ 
and $\Gamma(f;x,y_1)$ consists of two face
${\Delta}_1$ and ${\Delta}_2$.
Clearly $f$ is non-degenerate on ${\Delta}_1$.
If $f$ is degenerate on ${\Delta}_2$, we take a suitable triangulated
 change of coordinates $(x,y_2)$ so that 
$\mathcal{N}(f;x,y_2)=\alpha^2 y_2^4+2\alpha {\beta} xy_2^2+\gamma' x^{11+k},\,
 k=0,\dots,9$. 
This implies $(C,O)\sim B_{k+6,2}\circ B_{5,2}$.

--Secondly, we consider the case $\iota\ge 5$ (i.e., $\nu>2$).
This case, by the previous consideration, we see that $\iota=5$.
Then $\mathcal{N}(f,x,y_1)=(\alpha y_1^2+{\beta} x^5)^2$ and 
$\Gamma(f;x,y_1)$ consists of one face ${\Delta}$ 
with the weight vector $P={^t(2,5)}$ 
and $f$ is degenerate on ${\Delta}$. 
We consider the toric modification 
with respect to the canonical regular subdivision $\Sigma^*$
of $\Gamma^*(f;x,y_1)$. The toric coordinate chart which
intersects the strict transform $\widetilde C$
is described by  a unimodular matrix 
\[
 \sigma=\left(\begin{matrix}
2&1\\
5&3
\end{matrix}\right),\,\, \pi(u,v)=(u^2v,u^5v^3).
\]
Then we taking  admissible translated toric coordinates $(u,v_2),\ 
v_2=v_1+h(u), v_1=\alpha v+{\beta}$ for a suitable polynomial $h$,
we have 
\[\pi^*f(u,v_2)=u^{20}{(v_2-{\beta/\alpha}+h(u))}\Big{(}\alpha^2 v_2^2+{\beta''} u^{5}+
{\text{(higher terms)}}\Big{)},\quad{\beta''}\ne 0.\]
Thus we can get $(C,O)\sim (B_{5,2}^2)^{B_{5,2}}$.
Hence  we have the assertion (4) of Proposition \ref{caseII-2}.

The assertions $(5),\cdots, (12)$ can be shown in a similar manner.
 \end{proof}
\begin{remark}
{\em{Note that the singularity $B_{25,4}$ in case $(12)$ has Milnor number $72$ and 
$72$ is the maximum Milnor number of irreducible curve of degree $10$. 
Thus in this case, $C$ is a rational curve.}}
\end{remark}

The classification  Proposition \ref{caseII-2}
 can be rewritten as follows from the viewpoint of $\iota$.
 \newtheorem{propositionbis}[proposition]{Proposition-bis}
\begin{enumerate}
\item If $\iota =2$, we have $(C,O)\sim B_{5,4}$.
\item If $\iota =3$, we have $(C,O)\sim (B_{3,2}^2)^{B_{3,2}}$.
\item If $\iota =4$, we have 
      $(C,O)\sim (B_{4,2}^2)^{2B_{2,2}}$, $B_{10,4}$ and 
      $B_{k,2}\circ B_{5,2}\, (6\le k\le15)$. 
\item If $\iota =5$, we have 
      $(C,O)\sim (B_{4,2}^2)^{(B_{7,2} +B_{2,2})}$ 
      and $(B_{5,2}^2)^{B_{5,2}}$. 
\item If $\iota =6$, we have 
      $(C,O)\sim  (B_{4,2}^2)^{(B_{12,2}+B_{2,2})}$, 
      $(B_{6,2}^2)^{2B_{3,2}}$, $(B_{7,2}^2)^{B_{2,2}}$ and $B_{15,4}$. 
\item If $\iota =7$, we have 
      $(C,O)\sim  (B_{4,2}^2)^{(B_{17,2}+B_{2,2})}$, 
      $(B_{6,2}^2)^{ (B_{8,2}+B_{3,2})}$ and 
      $({B_{7,2}^2})^{ B_{7,2}}$.  
\item If $\iota =8$, we have 
       $(C,O)\sim (B_{4,2}^2)^{(B_{22,2}+B_{2,2})}$, 
      $(B_{6,2}^2)^{(B_{13,2}+B_{3,2})}$,
      $({B_{8,2}^2})^{2B_{4,2}}$, \\  
      $(B_{9,2}^2)^{B_{5,2}}$,$B_{20,4}$,
      $B_{11,2} \circ B_{10,2}$ and $B_{12,2}\circ B_{10,2}$.
\item If $\iota =9$, we have  
      $(C,O)\sim (B_{4,2}^2)^{(B_{27,2}+B_{2,2})}$, 
      $(B_{6,2}^2)^{(B_{18,2}+B_{3,2})}$,
      $(B_{8,2}^2)^{(B_{9,2}+B_{4,2})}$ and $(B_{9,2}^2)^{B_{10,2}}$. 
\item If $\iota =10$, we have 
      $(C,O)\sim (B_{4,2}^2)^{(B_{32,2}+B_{2,2})}$,
      $(B_{6,2}^2)^{(B_{23,2}+B_{3,2})}$, 
      $(B_{8,2}^2)^{(B_{14,2}+B_{4,2})}$, 
      $(B_{10,2}^2)^{2B_{5,2}}$,
      $(B_{11,2}^2)^{ B_{6,2}}$,
      $(B_{12,2}^2)^{2{B}_{1,2}}$ and $B_{25,4}$. 
\end{enumerate}
\noindent {\bf(b-2)} 
Assume that $C_2$ is a union of two lines meeting at the origin $(m_2=2)$ 
and the tangent cone of $C_5$ is $y^2=0$.
\begin{proposition}\label{caseII-3}
Suppose that  the tangent cone of $C_5$ 
is a line with multiplicity 2 and $C_2$ is a union 
of two lines meeting at the origin. 
Then the germ $(C,O)$ can take   following singularities.
\begin{enumerate}
\item If  $(C_5,O)\sim B_{3,2}$, we have
         \begin{enumerate}
          \item  $(C,O)\sim {(B_{3,2}^2)}^{B_{8,2}}$ for $\iota=4$ and 
          \item  $(C,O)\sim (B_{3,2}^2)^{B_{13,2}}$ for $\iota=5$. 
         \end{enumerate} 
\item If  $(C_5,O)\sim B_{4,2}$, we have  
         \begin{enumerate}
          \item  $(C,O)\sim {(B_{4,2}^2)}^{2B_{2,2}}$ for $\iota=4$, 
          \item  $(C,O)\sim {(B_{4,2}^2)}^{2B_{7,2}}$ for $\iota=6$, and 
          \item  $(C,O)\sim B_{16,2} \circ (B_{2,1}^2)^{B_{7,2}}$ for $\iota=7$.
         \end{enumerate} 
\item If  $(C_5,O)\sim B_{5,2}$, we have  
         \begin{enumerate}
          \item  $(C,O) \sim B_{10,4}$
                 or $B_{k,2}\circ B_{5,2}\, (6\le k\le15)$ for $\iota=4$,
          \item  $(C,O) \sim (B_{5,2}^2)^{B_{10,2}}$ for $\iota=6$ and 
          \item  $(C,O) \sim (B_{5,2}^2)^{B_{15,2}}$ for $(\iota=7)$.
         \end{enumerate} 
\item If  $(C_5,O)\sim B_{6,2}$, we have  
         \begin{enumerate}
          \item  $(C,O)\sim  B_{10,4}$ for $\iota=4$ and  
          \item  $(C,O)\sim {(B_{6,2}^2)}^{2B_{3,2}}$ for $\iota=6$.
          \end{enumerate}        
\item If  $(C_5,O)\sim B_{7,2}$, we have 
         \begin{enumerate}
          \item  $(C,O)\sim B_{10,4}$ for $\iota=4$ and 
          \item  $(C,O)\sim {(B_{7,2}^2)}^{B_{2,2}}$ for $\iota=6$.
          \end{enumerate}  
\item If  $(C_5,O)\sim B_{k,2}$, we have 
         \begin{enumerate}
          \item  $(C,O)\sim B_{10,4}$ for $\iota=4$ and 
          \item  $(C,O)\sim B_{15,4} \ (8\le k\le 13)$ for $\iota=6$ 
         \end{enumerate}  
\end{enumerate}
\end{proposition}
\begin{proof}
By taking a local coordinates $(x,y_1)$, we can assume   
\[
f_5(x,y_1)=\alpha y_1^2+\beta x^{\ell}+{(\text{higher terms})},
\quad \alpha,\, \beta\ne0. 
\]
Now we assume  that $f_2(x,y_1)=\ell_1(x,y_1)\ell_2(x,y_1)$ where
\[
\ell_1=y_1+c_\nu x^\nu+\text{(higher
terms)},\,\,\ell_2=c_2(y_1+\gamma x)+\text{(higher terms)},\, c_{\nu},c_2\ne 0
\]
We put $\iota_1=I(\ell_1,C_5;O)\le 5$,
$\iota_2=I(\ell_2,C_5;O)$.
As $\gamma\ne 0$, we have 
$\iota_2=2$. 
Hence $\iota=\iota_1+2$ and  
we have $4\le \iota \le 7$.

Comparing Newton boundaries of $f_5^2$ and $f_2^5$, and 
applying an similar argument as in  case (b-1),
we get assertions of Proposition \ref{caseII-3}.
\end{proof}

\subsection{Case III : $m_{5}$ = 3.}
We divide Case III into three cases by type of the tangent cone of $C_{5}$.\\
(a) $T_{O}C_5$ consists of distinct three lines. \\
(b) $T_{O}C_5$ consists of a double line and a simple line. \\
(c) $T_{O}C_5$ consists of is a  line with multiplicity 3.\\  

First we remark that if $C_2$ is smooth and 
$C_2$ intersects transversely with $C_5$ at the origin ($\iota=3$),
we have $(C,O) \sim B_{6,5}$ by Lemma \ref{cuspidal} in \S3.
So hereafter, we consider the case $C_2$ and 
the tangent cone of $C_5$ does not 
intersect transversely.
\noindent
\subsubsection{Case III-(a)}: We first consider  Case III-(a).
We assume that 
the tangent cone of $C_5$ consists of distinct three lines.
\begin{proposition}
Under the situation of  Case III-(a), $(C,O)\sim B_{6,5}$ if $\iota=3$.
For $\iota\ge 4$ we have the following possibilities of $(C,O)$.  
\begin{enumerate}
\item Assume that  $C_2$ is smooth and  tangent to  $y=0$.
      Then $(C,O)$ can be  $B_{5\iota-14,2}\circ B_{4,3}$ for
      $\iota = 4,5, \cdots 10$.
\item Assume that $C_2$ consists of two distinct lines $\ell_1,\,\ell_2$. 
      Put $\iota_i=I(\ell_i,C_5;O)\ge3$ $(i=1,2)$. 
      Then 
      $(C,O)\sim B_{5\iota_2- 9,2}\circ {(B_{1,1}^2)}^{B_{4,2}}
      \circ B_{2,5 \iota_1 -9}$ with $\iota_1+\iota_2=6,7,\cdots,10$.
\end{enumerate}
\end{proposition}
\begin{proof}
The assertion is immediate from Lemma \ref{termination1}.
\end{proof}
\subsubsection{Case III-(b)}
In this case,
we may assume that the tangent cone of $C_5$ consists of
a double line $\{y=0\}$ and a single line  $\{x=0\}$.\\

\noindent \textbf{(b-1)}
Assume first that $C_2$ is smooth $(m_2=1)$.
If $\iota=3$, $(C,O)\sim B_{6,5}$ by Lemma \ref{cuspidal}.
Therefore 
we consider the case $\iota \ge 4$.
The common tangent line of $C_2$  is either $\{x=0\}$ or $\{y=0\}$.
When the common tangent line is $\{x=0\}$, 
 $(C,O)$ is described by 
Lemma \ref{termination4}.

So we assume that common tangent cone is $\{y=0\}$.  
If $\iota=4$, we have 
$f_5(x,y)=b_{12}xy^2+b_{40}x^4+({\text{higher terms}})$
 and $\iota=4$ if and only if $b_{40}\ne0$.
Hence we have $(C,O)\sim B_{8,5}$.

Next we consider the case $\iota \ge 5$ and 
we take a local coordinates system $(x,y_1)$ so that 
$C_2$ is defined by $y_1=0$ and we have
$f_5(x,y_1)=\beta_{12}xy_1^2+\beta_{31}x^3y_1+\beta_{50}x^5+({\text{higher
terms}})$ with $\beta_{12}\ne 0$.
First we assume that $\iota=5$. 
Then $\beta_{50}\ne 0$ and we factor $\mathcal{N}(f;x,y_1)$ as
\[\mathcal{N}(f;x,y_1)=y_1^5+x^2(\beta_{12}y_1^2+\beta_{31}x^2y_1+\beta_{50}x^4)^2
=\prod_{i=1}^{5}(y_1+\alpha_ix^2)\] and we see that
$\Gamma(f;x,y)$ consists of one face with weight vector $P={^t(1,2)}$. 
Then we have several cases:
\begin{enumerate}
\item  $\alpha_1,\dots,  \alpha_5$ are all distinct.
\item $\alpha_1= \alpha_2$  and  $\alpha_3,\alpha_4,\alpha_5$
are mutually distinct and different from $\alpha_1$.
\item  $\alpha_1= \alpha_2=\alpha_3$ and  $\alpha_4,\alpha_5$
are mutually distinct and different from $\alpha_1$.
\item  $\alpha_1= \alpha_2$, $\alpha_3=\alpha_4$ and 
      $\alpha_1\ne \alpha_3$ and  $\alpha_5$
is different from $\alpha_1$, $\alpha_3$.
\end{enumerate}
We can see by an easy computation that the other cases are not possible.
(By a direct computation, we see that if $\mathcal N(f;x,y_1)=0$ has a root with
multiplicity
4, $\beta_{50}=0$ and  the intersection number jumps up to 6.)
\begin{lemma}\label{c1}
Under the above situation, we further assume that $\iota=5$.  
\begin{enumerate}
\item If  $\alpha_1,\dots,  \alpha_5$ are all distinct,
        $(C,O)\sim B_{10,5}$.
\item If $\alpha_1= \alpha_2$  and  $\alpha_3,\alpha_4,\alpha_5$
      are mutually distinct and different from $\alpha_1$,
        $(C,O)\sim B_{k,2}\circ B_{6,3}$, $(5\le k \le12)$.
\item If $\alpha_1= \alpha_2=\alpha_3$ and  $\alpha_4,\alpha_5$
      are mutually distinct and different from $\alpha_1$, 
        $(C,O)\sim B_{k,3}\circ B_{4,2}$ $(k=7,\cdots,11)$ or
      $B_{3,1}\circ B_{5,2}\circ B_{4,2}$ or 
        $B_{3,1}\circ B_{7,2}\circ B_{4,2}$ or  
      $B_{k,2}\circ B_{3,1}\circ B_{4,2}$ $(k=7,8,9)$.
\item If $\alpha_1= \alpha_2$, $\alpha_3=\alpha_4$ and 
      $\alpha_1\ne \alpha_3$ and  $\alpha_5$
       is different from $\alpha_1$, $\alpha_3$,   
       $(C,O)\sim B_{k_2+4,2}\circ B_{2,1}\circ (B_{2,1}^2)^{B_{{k_1},2}}$
       where $(k_1,k_2)$ moves  in the set
\newline
       $\{(k_1,k_2); 13-k_2\ge  k_1\ge k_2-4,\, k_2=5,\dots,7\}\cup\{(4,8),(5,9)\}$.
       \end{enumerate}
\end{lemma}

\begin{proof}
The case (1) is clear.
We  consider the case (2) and we may assume $\alpha_1=\alpha_2$.
Then we see that the Newton boundary 
$\Gamma(f;x,y_1)$ has two faces $\Delta_1$ and $\Delta_2$
and $f$ is non-degenerate on  $\Delta_1$.
Taking a suitable 
triangular change of coordinates,
we can make $f$  non-degenerate on $\Delta_2$.
Hence this gives series $(C,O)\sim B_{k,2}\circ B_{6,3}$, $k=5,\cdots, 12$.
We can consider the cases (3) and (4) similarly.
\end{proof}
\begin{remark}
{\rm{
In (4) of Lemma \ref{c1},  we have the following 
 symmetry.
Let
 \[
  f(x,y_1)=(y_1+\alpha_1x^2)^2(y_1+\alpha_3x^2)^2(y_1+\alpha_5x^2)
+\text{(higher
 terms)}\] 
be a defining polynomial of  $(C,O)$. 
First we take a change of coordinates $(x,y_2)= (x,y_1+\alpha_1x^2)$ and 
we take further
changes of  coordinates of type $y_2\to y_2+c\,x^{j},\, 2\le j\le [k_2/2]$ if
 necessary and  we can assume 
\[
 f(x,y_2)=y_2^2\,(y_2+(\alpha_3-\alpha_1) x^2)^2(y_2+(\alpha_5-\alpha_1)x^2)
+\beta x^{k_2+6}+\text{(higher
 terms)}
\]
 and its Newton boundary consists of two faces 
$\Delta_1$ and $\Delta_2$ and $f$ is non-degenerate on $\Delta_2$ 
but degenerate on $\Delta_1$.
(Here ``higher terms'' are linear combinations of monomials above the
 Newton boundary.)
Let $P_1={}^t(1,2)$ and $P_2$ be the weight vectors corresponding to
$\Delta_1,\Delta_2$ respectively.
The  ${\rm Cone}(E_1, P_1)$ needs one vertex $T_1={}^t(1,1)$ to be regular.
For the cones ${\rm Cone}(P_1,P_2)$ and ${\rm Cone}(P_2,E_2)$,
the subdivision changes for $k_2$ being
even or odd.

-- For $k_2=2m+1$, $P_2={}^t(2,2m+1)$ and 
 ${\rm Cone}(P_1,P_2)$ and ${\rm Cone}(P_2,E_2)$
are  subdivided into regular fans by adding vertices 
$\{T_i={}^t(1,i),\,3\le i\le m\}$  and $S={}^t(1,m+1)$.

-- For $k_2=2m$, $P_2={}^t(1,m))$ and 
${\rm Cone}(P_1,P_2)$ is subdivided into a  regular fan by adding vertices
$\{T_i={}^t(1,i),\,3\le i \le m-1\}$.
The ${\rm Cone}(P_2,E_2)$ is already regular.
Note that in any case, the corresponding resolution is minimal. 
In the second case, 
$\hat E(P_2)^2=-1$ but it intersects with two components of $C$.
{\footnotesize{
\begin{figure}[H]
\unitlength 0.1in
\begin{picture}( 56.5400,  8.1100)(  0.1000, -9.4500)
\put(0.1000,-6.2400){\makebox(0,0)[lb]{\huge{$\bullet$}}}%
%
\special{pn 8}%
\special{pa 54 568}%
\special{pa 346 568}%
\special{fp}%
\put(3.0100,-6.2400){\makebox(0,0)[lb]{\huge{$\bullet$}}}%
%
\special{pn 8}%
\special{pa 346 568}%
\special{pa 636 568}%
\special{fp}%
\put(5.9300,-6.2400){\makebox(0,0)[lb]{\huge{$\bullet$}}}%
%
\special{pn 8}%
\special{pa 644 568}%
\special{pa 934 568}%
\special{fp}%
\put(8.9100,-6.2400){\makebox(0,0)[lb]{\huge{$\bullet$}}}%
%
\special{pn 8}%
\special{pa 934 568}%
\special{pa 1226 568}%
\special{fp}%
\put(11.8200,-6.2400){\makebox(0,0)[lb]{\huge{$\bullet$}}}%
%
\special{pn 8}%
\special{pa 1226 568}%
\special{pa 1386 568}%
\special{dt 0.045}%
%
\special{pn 8}%
\special{pa 1518 568}%
\special{pa 1692 568}%
\special{dt 0.045}%
\put(15.8200,-6.2400){\makebox(0,0)[lb]{\huge{$\bullet$}}}%
%
\special{pn 8}%
\special{pa 1656 568}%
\special{pa 1946 568}%
\special{fp}%
\put(25.5400,-8.3600){\makebox(0,0)[lb]{\huge{$\bullet$}}}%
\put(19.1700,-6.1200){\makebox(0,0)[lb]{\huge{$\bullet$}}}%
%
\special{pn 8}%
\special{pa 1980 562}%
\special{pa 2272 562}%
\special{fp}%
\put(22.0100,-6.1200){\makebox(0,0)[lb]{\huge{$\bullet$}}}%
%
\special{pn 8}%
\special{pa 2294 548}%
\special{pa 2548 406}%
\special{dt 0.045}%
\special{sh 1}%
\special{pa 2548 406}%
\special{pa 2480 420}%
\special{pa 2500 432}%
\special{pa 2500 456}%
\special{pa 2548 406}%
\special{fp}%
\put(25.5400,-4.2200){\makebox(0,0)[lb]{$\bigcirc$}}%
\put(9.6300,-8.5000){\makebox(0,0)[lb]{$k_2=2m+1$}}%
%
\put(22.4500,-9.2900){\makebox(0,0)[lb]{}}%
%
\special{pn 8}%
\special{pa 358 566}%
\special{pa 358 262}%
\special{fp}%
\put(3.0600,-3.0400){\makebox(0,0)[lb]{\huge{$\bullet$}}}%
%
\special{pn 8}%
\special{pa 352 246}%
\special{pa 642 246}%
\special{fp}%
\put(5.9800,-3.0400){\makebox(0,0)[lb]{\huge{$\bullet$}}}%
%
\special{pn 8}%
\special{pa 648 246}%
\special{pa 940 246}%
\special{fp}%
\put(8.9600,-3.0400){\makebox(0,0)[lb]{\huge{$\bullet$}}}%
%
\special{pn 8}%
\special{pa 940 246}%
\special{pa 1232 246}%
\special{fp}%
\put(11.8700,-3.0400){\makebox(0,0)[lb]{\huge{$\bullet$}}}%
%
\special{pn 8}%
\special{pa 1232 246}%
\special{pa 1392 246}%
\special{dt 0.045}%
\put(2.3000,-7.9400){\makebox(0,0)[lb]{$\hat{E}(P_1)$}}%
\put(20.7200,-8.0200){\makebox(0,0)[lb]{$\hat{E}(P_2)$}}%
\put(31.2000,-6.4000){\makebox(0,0)[lb]{\huge{$\bullet$}}}%
%
\special{pn 8}%
\special{pa 3164 584}%
\special{pa 3456 584}%
\special{fp}%
\put(34.1100,-6.4000){\makebox(0,0)[lb]{\huge{$\bullet$}}}%
%
\special{pn 8}%
\special{pa 3456 584}%
\special{pa 3746 584}%
\special{fp}%
\put(37.0300,-6.4000){\makebox(0,0)[lb]{\huge{$\bullet$}}}%
%
\special{pn 8}%
\special{pa 3754 584}%
\special{pa 4044 584}%
\special{fp}%
\put(40.0100,-6.4000){\makebox(0,0)[lb]{\huge{$\bullet$}}}%
%
\special{pn 8}%
\special{pa 4044 584}%
\special{pa 4336 584}%
\special{fp}%
\put(42.9200,-6.4000){\makebox(0,0)[lb]{\huge{$\bullet$}}}%
%
\special{pn 8}%
\special{pa 4336 584}%
\special{pa 4496 584}%
\special{dt 0.045}%
%
\special{pn 8}%
\special{pa 4628 584}%
\special{pa 4802 584}%
\special{dt 0.045}%
\put(46.9200,-6.4000){\makebox(0,0)[lb]{\huge{$\bullet$}}}%
%
\special{pn 8}%
\special{pa 4766 584}%
\special{pa 5056 584}%
\special{fp}%
\put(50.2700,-6.2800){\makebox(0,0)[lb]{\huge{$\bullet$}}}%
%
\special{pn 8}%
\special{pa 5090 578}%
\special{pa 5382 578}%
\special{fp}%
\put(53.1100,-6.2800){\makebox(0,0)[lb]{\huge{$\bullet$}}}%
%
\special{pn 8}%
\special{pa 5404 564}%
\special{pa 5658 422}%
\special{dt 0.045}%
\special{sh 1}%
\special{pa 5658 422}%
\special{pa 5590 436}%
\special{pa 5610 448}%
\special{pa 5610 472}%
\special{pa 5658 422}%
\special{fp}%
\put(56.6400,-4.3800){\makebox(0,0)[lb]{$\bigcirc$}}%
%
\special{pn 8}%
\special{pa 5404 592}%
\special{pa 5640 762}%
\special{dt 0.045}%
\special{sh 1}%
\special{pa 5640 762}%
\special{pa 5598 708}%
\special{pa 5598 732}%
\special{pa 5574 740}%
\special{pa 5640 762}%
\special{fp}%
\put(42.4000,-8.7000){\makebox(0,0)[lb]{$k_2=2m$}}%
%
\put(53.5500,-9.4500){\makebox(0,0)[lb]{}}%
%
\special{pn 8}%
\special{pa 3468 582}%
\special{pa 3468 278}%
\special{fp}%
\put(34.1600,-3.2000){\makebox(0,0)[lb]{\huge{$\bullet$}}}%
%
\special{pn 8}%
\special{pa 3462 262}%
\special{pa 3752 262}%
\special{fp}%
\put(37.0800,-3.2000){\makebox(0,0)[lb]{\huge{$\bullet$}}}%
%
\special{pn 8}%
\special{pa 3758 262}%
\special{pa 4050 262}%
\special{fp}%
\put(40.0600,-3.2000){\makebox(0,0)[lb]{\huge{$\bullet$}}}%
%
\special{pn 8}%
\special{pa 4050 262}%
\special{pa 4342 262}%
\special{fp}%
\put(42.9700,-3.2000){\makebox(0,0)[lb]{\huge{$\bullet$}}}%
%
\special{pn 8}%
\special{pa 4342 262}%
\special{pa 4502 262}%
\special{dt 0.045}%
\put(33.4000,-8.1000){\makebox(0,0)[lb]{$\hat{E}(P_1)$}}%
\put(51.8200,-8.1800){\makebox(0,0)[lb]{$\hat{E}(P_2)$}}%
\put(56.6000,-9.0000){\makebox(0,0)[lb]{$\bigcirc$}}%
%
\special{pn 8}%
\special{pa 2320 560}%
\special{pa 2590 770}%
\special{fp}%
\end{picture}%
\caption{}
\end{figure}
}}
After  taking a toric modification with respect to the canonical
 subdivision, 
we have 
$(\widetilde{C},\xi_1)\sim B_{k_1,2}$. Hence 
$(C,O)\sim  B_{k_2,2}\circ B_{2,1}\circ (B_{2,1}^2)^{B_{k_1,2}}$.

Using canonical subdivision for the second toric modification, 
we see that resolution graph has 
three branches with center $\hat E(P_1)$:
one branch with a single vertex which corresponds to 
$\hat E(T_1)$. The second branch  corresponds to the vertices in 
${\rm Cone}(P_1,P_2)$ and ${\rm Cone}(P_2,E_2)$ 
(respectively ${\rm Cone}(P_1,P_2)$) for $k_2$ is odd (resp. even).
The third branch corresponds to the vertices for the second toric modification.

To see the relation between the second and third branches,
we take change of coordinates 
$(x,y_3)= (x,y_2+(\alpha_3-\alpha_1) x^2)$ from the beginning.
After a finite number of triangular changes of coordinates,
 we arrive to a expression
$(C,O)\sim  B_{k_2',2}\circ B_{2,1}\circ (B_{2,1}^2)^{B_{k_1',2}}$.
By an easy calculation and the minimality of the resolution graph so
 obtained,
we see that $k_1'=k_2-4,\, k_2'=k_1+4$. Thus 
$B_{k_2,2}\circ B_{2,1}\circ (B_{2,1}^2)^{B_{k_1,2}}\sim
 B_{k_1+4,2}\circ B_{2,1}\circ (B_{2,1}^2)^{B_{k_2-4,2}} $. Therefore
 in the classification, we can assume that $k_1+4\ge k_2$.
}}
\end{remark}

Next we consider the case  $\iota\ge 6$.
\begin{lemma}\label{c2}
Under case III-b, we assume that $\iota=6$. Then the topological type of
 $(C,O)$  is  generically $B_{9,2}\circ B_{6,3}$ and it can degenerate into  
  ${(B_{5,2}^2)}^{{B}_{1,2}}\circ B_{2,1}$ or 
     $B_{9,2} \circ B_{2,1}  \circ  (B_{2,1}^2)^{ B_{k,2}}$,
 $(k=1,\cdots, 8)$. 
\end{lemma}
\begin{proof}
We take a local coordinates system $(x,y_1)$ so that 
$C_2$ is defined by $y_1=0$ and
$f_5(x,y_1)=\beta_{12}xy_1^2+\beta_{31}x^3y_1+\beta_{60}x^6+({\text{higher terms}})$.
Then  
$f(x,y_1)$ is written as 
\begin{eqnarray*}
& f(x,y_1)=y_1^5+(\beta_{12}xy_1^2+\beta_{31}x^3y_1+\beta_{60}x^6)^2
+{\text{(higher terms)}}\hspace{5.2cm}\\
& =y_1^5+(\beta_{12}xy_1^2+\beta_{31}xy_1^3)^2+(\beta_{31}x^3y_1+\beta_{60}x^6)^2-
\beta_{31}^2x^6y_1^2+{\text{(higher terms)}}.
\end{eqnarray*}
If $\beta_{31}\ne 0$,  $\Gamma(f;x,y_1)$ consists of two faces $\Delta_1$, $\Delta_2$
and
$f_{\Delta_1}(x,y_1)=y_1^5+x^2y_1^2(\beta_{12}y_1+\beta_{31}x^2)^2$
and
$f_{\Delta_2}(x,y)=(\beta_{31}x^3y_1+\beta_{60}x^6)^2$.
In this case, we first take a triangular change of
 coordinates
of type 
$(x,y_2)=(x,y_1+c_3x^3+c_4x^4)$ so that the face $\Delta$ changes into a
 non-degenerate face
$\Delta_2'$ (new face after change of the coordinate) and
\[
 f_{\Delta_2'}(x,y_2)= \beta_{31}x^6\,(y_2^2+c_3'x^9),\, c_3\ne 0.
\]
If $f$ is non-degenerate on $\Delta_1$,
we have $(C,O)\sim B_{9,2}\circ B_{6,3}$.
\newline
If $f$ is degenerate on $\Delta_1$  
 ($\beta_{31}=4\beta_{12}^3/27,\,  \beta_{31}\ne 0$),
then 
$f_{\Delta_1}(x,y_1)=\alpha^2
 y_1^2(9y_1+\beta_{12}^2x^2)(9y_1+4\beta_{12}^2x^2)^2$
where $\alpha$ is a non-zero constant.
To analyze the singularity on the face $\Delta_1$, we take a toric modification:
let $P={}^t(1,2)$ be the weight vector corresponding to $\Delta_1$ and 
we take a toric modification with respect to an admissible regular 
simplicial cone subdivision $\Sigma^{*}$, $\pi:X\to \Bbb{C}^2$.  
We may assume that $\sigma=$Cone$(P_1,T_1)$ is a cone in $\Sigma^{*}$ where 
$T_1={}^t(1,3)$.
We take the toric coordinates $(u,v)$ of the chart $\Bbb{C}^2_{\sigma}$. 
Then we have 
 $\pi_\sigma(u,v)=(uv,u^2v^3)$ and 
\[
\pi_\sigma^{*}f(u,v)=\alpha^2 u^{10} v^{12}\tilde f(u,v)=\alpha^2 u^{10} v^{12}\Big{(}(9v+\beta_{12}^2)(9v+4\beta_{12}^2)^2+{\text{(higher terms)}}\Big{)}
\]
 and 
 the strict transform $\widetilde C$ splits into two components.
We see that one of the components of $\widetilde C$ which correspond to
the non-degenerate component of $f_{\Delta_1}$
is smooth and transversely with $\hat E (P)=\{u=0\}$.
To see the other 
component of $\widetilde C$,
we take the  translated toric coordinates $(u,v_1),\ v_1=9v+4\beta_{12}^2$. 
Then 
  $\tilde f(u,v_1)=cv_1^2+\gamma_1u+{\text{(higher terms)}}$
  where $c$ is a non-zero constant.
Hence if $\gamma_1\ne0$, we get 
$(C,O)\sim B_{9,2}\circ B_{2,1}\circ {(B_{2,1}^2)}^{B_{1,2}}$. 
If $\gamma_1=0$, taking a triangular change of  coordinates
of the type $(u,v_2)=(u,v_1+d_1u+\cdots+c_ku^j),\, j=[k/2]$, 
we can easily see that 
$(C,O)\sim B_{9,2}\circ B_{2,1}\circ {(B_{2,1}^2)}^{ B_{k,2}}$,
$(k=2,\cdots,8)$.  

Next we consider the case  $\beta_{31}=0$.
Then 
\[
\mathcal{N}(f;x,y_1)=y_1^5+x^2(\beta_{12}y_1^2+\beta_{60} x^5)^2
\] 
where $\beta_{60}\ne0$ since $\iota=6$
and
$\Gamma(f;x,y_1)$ has two faces 
$\Delta_1$ and $\Delta_2$
and the corresponding face functions are given by
$f_{\Delta_1}(x,y_1)=y_1^5+\beta_{12}^2x^2y_1^4$ and 
$f_{\Delta_2}(x,y_1)=x^2(\beta_{12}y_1^2+\beta_{60}x^5)^2$.
Thus $f$ is non-degenerate on $\Delta_1$ and 
 degenerate on $\Delta_2$. 
Then  taking a toric modification which is the same as (4) of
Proposition \ref{caseII-2},
we get $(C,O)\sim {(B_{5,2}^2)}^{{B}_{1,2}}\circ B_{2,1}$.
\end{proof}
 
For the remaining cases $\iota\ge 7$, we can carry out the classification
in the exact same way. So we can summarize the result as
follows. 
\begin{proposition} \label{CaseIII-2} 
Suppose that $C_2$ is smooth, $m_5=3$ and the tangent cone of the 
quintic $C_5$ consists of a double line $L_1$ and a single line $L_2$.

 If $\iota=3$, we have $(C,O)\sim B_{6,5}$.

If $\iota \ge 4$, $C_2$ is tangent to either $L_1$ or 
$L_2$ and  we have the following possibility.

\noindent
{\rm{(I)}} 
Assume that  the common tangent cone is $L_2$.\\  
         Then the germ $(C,O)$ can be of type  $B_{3,4}\circ B_{2,5\iota-14}$ 
         for $4\le \iota \le10$. (See Lemma \ref{termination4}).
         
\noindent{\rm{(II)}}
 Assume that the common tangent cone is $L_1$.\\  
         Then the germ $(C,O)$ can be of type  $B_{2\iota,5}$ for $\iota=4,5$ and 
$B_{5\iota-21,2}\circ B_{6,3}$ for $\iota = 6,\cdots,10$. 

Further degenerations are given for fixed $\iota$ 
by the following list.
\begin{enumerate}
\item If $\iota=5$, 
we have 
\begin{enumerate}
\item $(C,O)\sim  B_{k,2} \circ B_{6,3}\ (5\le k \le 12)$
\item $(C,O)\sim B_{k,3}\circ B_{4,2}$ $(k=7,\cdots,11)$,
   $B_{3,1}\circ B_{5,2}\circ B_{4,2}$,
   $B_{3,1}\circ B_{7,2}\circ B_{4,2}$ and 
   $B_{k,2}\circ B_{3,1}\circ B_{4,2}\ (k=7,8,9)$.
\item $(C,O)\sim B_{k_2+4,2}\circ B_{2,1}\circ (B_{2,1}^2)^{B_{k_1,2}}$
\end{enumerate}
where $(k_1,k_2)$ is in
$\{(k_1,k_2); 13-k_2\ge  k_1\ge k_2-4,\, k_2=5,\dots,7\}\cup\{(4,8),(5,9)\}$.
\item If $\iota=6$, we have 
\begin{enumerate}
\item  $(C,O)\sim {(B_{5,2}^2)}^{{B}_{1,2}}\circ B_{2,1}$
\item  $(C,O)\sim  B_{9,2} \circ B_{2,1}  \circ  (B_{2,1}^2)^{B_{k,2}}\ (k=1,\cdots,8)$
   \end{enumerate}
\item If $\iota=7$, we have 
\begin{enumerate}
\item $(C,O)\sim {(B_{6,2}^2)}^{2{B}_{1,2}}\circ B_{2,1}$ and 
 $B_{13,4}\circ B_{2,1}$.
\item $(C,O)\sim B_{14,2} \circ B_{2,1}\circ (B_{2,1}^2)^{B_{k,2}}\ (k=1,\cdots, 7)$
   \end{enumerate}
\item If $\iota=8$, we have 
  \begin{enumerate}
\item $(C,O)\sim {(B_{6,2}^2)}^{B_{6,2}+B_{1,2}}\circ B_{2,1}$ and 
      ${(B_{7,2}^2)}^{B_{3,2}}\circ B_{2,1}$
\item $(C,O)\sim  B_{19,2} \circ B_{2,1}  \circ  (B_{2,1}^2)^{B_{k,2}}\ (k=1,\cdots, 6)$
   \end{enumerate} 
\item If $\iota=9$, we have 
\begin{enumerate}
\item $(C,O)\sim {(B_{6,2}^2)}^{B_{11,2}+B_{1,2}}\circ B_{2,1}$, \ 
${(B_{8,2}^2)}^{2B_{2,2}}\circ B_{2,1}$, \
$B_{18,4}\circ B_{2,1}$ and \
$B_{11,2}\circ B_{9,2}\circ B_{2,1}$
\item $(C,O)\sim B_{24,2} \circ B_{2,1}\circ (B_{2,1}^2)^{ B_{k,2}},\ (k=1,\cdots, 4)
  $  \end{enumerate}  
\item If $\iota=10$, we have 
\begin{enumerate}
\item $(C,O)\sim {(B_{6,2}^2)}^{B_{16,2}+B_{1,2}}\circ B_{2,1}$,\
${(B_{8,2}^2)}^{B_{7,2}+B_{2,2}}\circ B_{2,1}$ and \ 
${(B_{9,2}^2)}^{B_{5,2}}\circ B_{2,1}$
\item $(C,O)\sim B_{29,2} \circ B_{2,1}  \circ  (B_{2,1}^2)^{B_{k,2}},\ (k=1,\cdots, 5, \ k \ne 4)$
\end{enumerate}
\end{enumerate}
\end{proposition}    
\noindent {\bf{(b-2)}}
Assume that $C_2$ is a union of two lines passing through the origin
 $(m_2=2)$ and
 the tangent cone of $C_5$ consists of
a double line $\{y=0\}$ and a single line is $\{x=0\}$.\\
So we assume that
$f_5(x,y)=b_{12} xy^2+b_{40}x^4+b_{04}y^4+({\text{higher terms}})$. 
We assume two lines of $C_2$ are defined by  $\ell_1:y+\alpha_1 x=0$, 
$\ell_2:\alpha_2 y+x=0$ and we put 
$\iota_{i}=I(\ell_i,C_5;O)\ge3$ $(i= 1,2)$.
Then we have $\iota=\iota_1+\iota_2 \geq 6$.
If $\iota=6$, we have $(\iota_1,\iota_2)=(3,3)$
$(\alpha_1,\alpha_2 \, \ne0)$.
If $\iota\ge7$, 
we have several possibility of $(\iota_1,\iota_2)$.
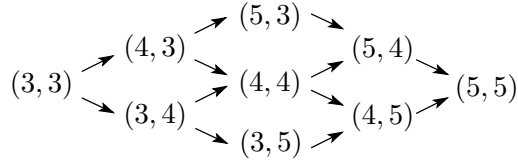
\begin{center}
\unitlength 0.1in
\begin{picture}( 24.3700,  7.7000)(  5.2000,-13.3000)
\put(8.3500,-10.0000){\makebox(0,0){$(3,3)$}}%
%
\special{pn 8}%
\special{pa 1046 930}%
\special{pa 1206 850}%
\special{fp}%
\special{sh 1}%
\special{pa 1206 850}%
\special{pa 1136 862}%
\special{pa 1158 874}%
\special{pa 1154 898}%
\special{pa 1206 850}%
\special{fp}%
%
\special{pn 8}%
\special{pa 1046 1070}%
\special{pa 1208 1146}%
\special{fp}%
\special{sh 1}%
\special{pa 1208 1146}%
\special{pa 1156 1100}%
\special{pa 1160 1124}%
\special{pa 1138 1136}%
\special{pa 1208 1146}%
\special{fp}%
\put(14.3000,-8.1000){\makebox(0,0){$(4,3)$}}%
%
\special{pn 8}%
\special{pa 1636 730}%
\special{pa 1796 650}%
\special{fp}%
\special{sh 1}%
\special{pa 1796 650}%
\special{pa 1726 662}%
\special{pa 1748 674}%
\special{pa 1744 698}%
\special{pa 1796 650}%
\special{fp}%
%
\special{pn 8}%
\special{pa 1636 870}%
\special{pa 1798 946}%
\special{fp}%
\special{sh 1}%
\special{pa 1798 946}%
\special{pa 1746 900}%
\special{pa 1750 924}%
\special{pa 1728 936}%
\special{pa 1798 946}%
\special{fp}%
\put(14.3000,-11.6000){\makebox(0,0){$(3,4)$}}%
%
\special{pn 8}%
\special{pa 1640 1090}%
\special{pa 1800 1010}%
\special{fp}%
\special{sh 1}%
\special{pa 1800 1010}%
\special{pa 1732 1022}%
\special{pa 1752 1034}%
\special{pa 1750 1058}%
\special{pa 1800 1010}%
\special{fp}%
%
\special{pn 8}%
\special{pa 1640 1230}%
\special{pa 1802 1306}%
\special{fp}%
\special{sh 1}%
\special{pa 1802 1306}%
\special{pa 1750 1260}%
\special{pa 1754 1284}%
\special{pa 1734 1296}%
\special{pa 1802 1306}%
\special{fp}%
\put(20.3000,-6.4500){\makebox(0,0){$(5,3)$}}%
\put(20.3000,-9.9500){\makebox(0,0){$(4,4)$}}%
\put(20.3500,-13.1000){\makebox(0,0){$(3,5)$}}%
\put(31.6500,-10.2000){\makebox(0,0){$(5,5)$}}%
\put(26.2000,-8.1500){\makebox(0,0){$(5,4)$}}%
\put(26.2000,-11.6500){\makebox(0,0){$(4,5)$}}%
%
\special{pn 8}%
\special{pa 2246 960}%
\special{pa 2406 880}%
\special{fp}%
\special{sh 1}%
\special{pa 2406 880}%
\special{pa 2336 892}%
\special{pa 2358 904}%
\special{pa 2354 928}%
\special{pa 2406 880}%
\special{fp}%
%
\special{pn 8}%
\special{pa 2246 630}%
\special{pa 2408 706}%
\special{fp}%
\special{sh 1}%
\special{pa 2408 706}%
\special{pa 2356 660}%
\special{pa 2360 684}%
\special{pa 2338 696}%
\special{pa 2408 706}%
\special{fp}%
%
\special{pn 8}%
\special{pa 2246 1330}%
\special{pa 2406 1250}%
\special{fp}%
\special{sh 1}%
\special{pa 2406 1250}%
\special{pa 2336 1262}%
\special{pa 2358 1274}%
\special{pa 2354 1298}%
\special{pa 2406 1250}%
\special{fp}%
%
\special{pn 8}%
\special{pa 2246 1030}%
\special{pa 2408 1106}%
\special{fp}%
\special{sh 1}%
\special{pa 2408 1106}%
\special{pa 2356 1060}%
\special{pa 2360 1084}%
\special{pa 2338 1096}%
\special{pa 2408 1106}%
\special{fp}%
%
\special{pn 8}%
\special{pa 2796 1150}%
\special{pa 2956 1070}%
\special{fp}%
\special{sh 1}%
\special{pa 2956 1070}%
\special{pa 2886 1082}%
\special{pa 2908 1094}%
\special{pa 2904 1118}%
\special{pa 2956 1070}%
\special{fp}%
%
\special{pn 8}%
\special{pa 2796 870}%
\special{pa 2958 946}%
\special{fp}%
\special{sh 1}%
\special{pa 2958 946}%
\special{pa 2906 900}%
\special{pa 2910 924}%
\special{pa 2888 936}%
\special{pa 2958 946}%
\special{fp}%
\end{picture}%
\end{center}
Above diagram depend only the numbers 
$(\alpha_1,\alpha_2,b_{40},b_{04})$.
\begin{proposition}
Suppose that $C_{2}$ is a union of two lines and the tangent cone of the
quintic $C_5$ consists of  a double line and a single line. 
Then we have the following possibilities.
\begin{enumerate}
\item If $\iota=6$,
then we have \begin{enumerate}
\item $(C,O)\sim {(B_{3,2}^2)}^{B_{4,2}}\circ B_{2,6}$, 
\item $(C,O)\sim B_{8,4} \circ  B_{2,6}$ and 
$B_{k,2} \circ B_{4,2} \circ B_{2,6}\ (5\le k \le 12)$.
\end{enumerate}
\item If $\iota=7$,
then we have two cases:
$(\iota_1,\iota_2)=(4,3)$ or $(3,4)$.
\begin{enumerate}
\item If $(\iota_1,\iota_2)=(4,3)$, we have 
\begin{enumerate}
\item $(C,O)\sim {(B_{3,2}^2)}^{B_{4,2}} \circ B_{2,11}$,
\item $(C,O)\sim B_{8,4} \circ B_{2,11}$ and 
$B_{k,2} \circ B_{4,2} \circ  B_{2,11}\,(5\le k \le 10)$.
\end{enumerate}
\item If $(\iota_1,\iota_2)=(3,4)$, we have 
 $(C,O)\sim {(B_{3,2}^2)}^{B_{9,2}}\circ B_{2,6}$.
\end{enumerate}
\item If $\iota=8$,
then we have two cases: 
$(\iota_1,\iota_2)=(5,3)$ or $(4,4)$ or $(3,5)$.
\begin{enumerate}
\item If $(\iota_1,\iota_2)=(5,3)$, we have 
$(C,O)\sim (B_{3,2}^2)^{B_{4,2}}\circ B_{2,16}$.
\item If $(\iota_1,\iota_2)=(4,4)$, we have 
$(C,O)\sim {(B_{3,2}^2)}^{B_{9,2}} \circ B_{2,11}$. 
\item If $(\iota_1,\iota_2)=(3,5)$, we have 
\begin{enumerate}
\item  $(C,O)\sim (B_{4,2}^2)^{2B_{5,2}}\circ B_{2,6}$,
\item  $(C,O)\sim {(B_{5,2}^2)}^{B_{6,2}}\circ  B_{2,6}$, 
${(B_{6,2}^2)}^{2{B}_{1,2}} \circ B_{2,6}$ and $B_{13,4} \circ B_{2,6}$.
\end{enumerate}
\end{enumerate}
\item If $\iota=9$,
then we have two cases:
$(\iota_1,\iota_2)=(5,4)$ or $(4,5)$.
\begin{enumerate}
\item If $(\iota_1,\iota_2)=(5,4)$, we have 
$(C,O)\sim {(B_{3,2}^2)}^{B_{9,2}} \circ B_{2,16}$. 
\item If $(\iota_1,\iota_2)=(4,5)$, we have 
\begin{enumerate}
\item $(C,O)\sim (B_{4,2}^2)^{2B_{5,2}}\circ B_{2,11}$,
\item $(C,O)\sim {(B_{5,2}^2)}^{B_{6,2}}\circ  B_{2,11}$, 
${(B_{6,2}^2)}^{2{B}_{1,2}} \circ B_{2,11}$ and 
$ B_{13,4} \circ B_{2,11}$.
\end{enumerate}
\end{enumerate}
If $\iota=10$, 
then we have
\begin{enumerate}
\item $(C,O)\sim (B_{4,2}^2)^{2B_{5,2}}\circ B_{2,16}$,
\item $(C,O)\sim {(B_{5,2}^2)}^{B_{6,2}}\circ  B_{2,16}$, 
      ${(B_{6,2}^2)}^{2{B}_{1,2}} \circ B_{2,16}$ and $B_{13,4} \circ B_{2,16}$.
      \end{enumerate}
\end{enumerate}
\end{proposition} 
We omit the proof as it is parallel to that of Proposition \ref{caseII-2}.

\subsubsection{Case III-(c)} In this case, 
we may assume that  the tangent cone of $(C_5,O)$ is given by $y^3=0$.
\\\\
\noindent
\textbf{(c-1)} Assume that $m_2=1$.
If $\iota=3$, we have $(C,O)\sim B_{6,5}$ by Lemma \ref{cuspidal}.
Therefore we consider the case $\iota\ge4$.
If $\iota=4$, we have $(C,O)\sim B_{8,5}$ as in  case (b-1).
If $\iota\ge5$, we get the following possibilities.

\begin{proposition}
Suppose that $C_2$ is smooth, $m_5=3$ and
the tangent cone of $C_5$ is 
a line with multiplicity $3$.
Then
the germ $(C,O)$ can be of type $B_{2\iota,5}$ for $\iota=3,4$ and 
\newline
if $\iota\ge 5$, we have the following possibilities.
\begin{enumerate}
\item If $\iota=5$, we have 
 $(C,O)\sim B_{5,10}$ and $B_{k,2} \circ B_{6,3}\ \ (k=5,\cdots,12)$. 
\item If $\iota=6$, we have 
 $(C,O)\sim B_{9,2}  \circ B_{6,3}$ and $B_{12,5}$.
\item If $\iota=7$, we have 
      $(C,O)\sim B_{14,2} \circ B_{6,3}$, $B_{7,2}\circ B_{8,3}$ and $B_{14,5}$.
\item If $\iota=8$, we have 
 $(C,O)\sim B_{19,2} \circ B_{6,3}$, $B_{12,2}  \circ B_{8,3}$ and  $B_{16,5}$.
\item If $\iota=9$, we have 
      $(C,O)\sim B_{24,2}\circ B_{6,3}$, $B_{17,2} \circ B_{8,3}$, 
       $B_{10,2} \circ B_{10,3}$ and $B_{18,5}$.   
\item If $\iota=10$, we have 
       $(C,O)\sim B_{29,2}\circ B_{6,3}$, $B_{22,2} \circ B_{8,3}$,
       $B_{15,2} \circ B_{10,3}$ and $B_{20,5}$.
\end{enumerate}
\end{proposition} 
We omit 
the proof as it is  parallel to Lemma \ref{c1} and Lemma \ref{c2}.

\noindent \textbf{(c-2)}
Assume that $C_2$ is two lines passing through the origin  and 
the tangent cone of quintic $C_5$ 
is defined by $\{y^3=0\}$.
Then we have $\iota \geq 6$ and
we can list the possibilities as in the  following proposition.
\begin{proposition}
Suppose that $C_2$ is two lines passing through the origin
and the tangent cone of quintic $C_5$ 
is defined by $\{y^3=0\}$.
Then we have: 
\begin{enumerate}
\item If $\iota=6$, we have
 $(C,O)\sim{(B_{4,3}^2)}^{B_{6,2}}$,\
 $ B_{4,2}\circ {(B_{3,2}^2)}^{B_{2,2}}$, 
      $B_{10,6}$ or $B_{6,3}\circ B_{5,3}$.
\item If $\iota=7$ we have $(C,O)\sim{(B_{4,3}^2)}^{B_{11,2}}$. 
   
\item If $\iota=8$ we have 
 $(C,O)\sim B_{9,2}\circ {(B_{3,2}^2)}^{B_{7,2}}$ or 
 $(C,O)\sim {(B_{5,3}^2)}^{B_{10,2}}$. 
\end{enumerate}
\end{proposition}  
The proof is parallel to the previous computations.

\subsection{Case IV : $m_{5}$ = 4}
We divide Case IV into five subcases by type of tangent cone $T_OC_5$ of $C_{5}$.\\
(a)  $T_OC_5$ consists of distinct four lines. \\
(b)  $T_OC_5$ consists of a double line  and distinct two lines. \\
(c)  $T_OC_5$ consists of a triple line and a simple line.\\
(d)  $T_OC_5$ consists of a line with multiplicity 4. \\
(e)  $T_OC_5$ consists of two double lines.\\

First we remark that if $C_2$ is smooth and 
$C_2$ intersects transversely with $C_5$ at the origin $(\iota=4)$,
$(C,O) \sim B_{8,5}$ by
Lemma \ref{cuspidal}.
So hereafter, we consider the case $C_2$ and the
tangent cone of $C_5$ does not intersect transversely.
First we prepare the following Lemma.
\begin{lemma}\label{termination8} 
Suppose that the conic $C_2$ is smooth and 
let $(x,y_1)$ be a local coordinate system so that $C_2$ is defined 
by $y_1=0$.
We put 
$f_5(x,y_1)=y_1(y_1+c_1x)(y_1+c_2x)(y_1+c_3x)+{\text{(higher terms)}}$.   
Then 
\begin{enumerate}
\item If $c_{i}\ne0$ for $i=1,2,3$, then
$(C,O)\sim B_{2\iota,5}$ for $\iota=4,5$ and 
if $\iota\ge5$, we have two series
$B_{k,2}\circ B_{6,3}$, $5\le k \le10$ for $\iota=5$ 
and $B_{5\iota-21,2}\circ B_{6,3}$ for $\iota = 6,\cdots,10$.

\item If $c_{1}=0$ and  $ c_{i}\ne0$ for $i=2,3$, we have
$(C,O)\sim B_{2\iota,5}$ for $\iota=4,5,6$ and 
$(C,O)\sim B_{5\iota-28,2}\circ B_{8,3}$ for $\iota=7,\cdots, 10$.
\item If $c_{1}=c_2=0$ and  $c_{3}\ne0$, 
we have  $(C,O)\sim B_{2\iota,5}$ for $\iota=4,\cdots, 8$ and 
$(C,O)\sim B_{5\iota-35,2}\circ B_{10,3}$ for $\iota=9,10$.
\item If $c_{i}=0,\ (i=1,2,3)$, we have
$(C,O)\sim B_{2\iota,5}$ for $\iota=4,\cdots, 10$.
\end{enumerate}
\end{lemma}
\begin{proof}
We observe to the Newton boundary $\mathcal{N}(f,x,y_1)$.
We have
\[
\mathcal{N}(f,x,y_1)=y_1^5+c^2x^6y_1^2+2c\beta_{50}x^8y_1+
\beta_{50}^2x^{10},\quad c:=c_1c_2c_3
\]
and $\Gamma(f;x,y_1)$ consists of one face $\Delta$ with 
weight vector ${}^t(1,2)$.
We can see that the discriminant $R$ of 
the face function  $f_{\Delta}(x,y_1)$ can be written as 
$R=\beta_{50}\,\alpha$ where $\alpha$ is a polynomial of $c$ and $\beta_{50}$.
Then we have $(C,O)\sim B_{10,5}$ if $R\ne 0,\, \iota=5$ ($\iota=5$ if
 and only if $\beta_{50}\ne0$).
We observe that  $R\ne 0$, if $c=0$ and $\beta_{50}\ne 0$.

We first consider the case (1): $c_i\ne 0$, $i=1,2,3$. 
Note that $\alpha=0$ and $\beta_{50}=0$ is impossible.
If $\beta_{50}\alpha\ne 0$, we have $\iota=5$ and $(C,O)\sim B_{10,5}$
as is observed above.  Thus
we consider the case $\alpha=0$ or $\beta_{50}=0$.
In both cases,
by taking a suitable triangular change of coordinates,
we can get non-degenerate singularity.
If  $\alpha=0$, we have $(C,O)\sim B_{k,2}\circ B_{6,3}$, $5\le k \le10$. 
If $\beta_{50}=0$, we have  $B_{5\iota-21,2}\circ B_{6,3}$ for $\iota = 6,\cdots,10$.   
Hence we have assertion $(1)$.

 To consider
the cases (2) $\sim$ (4), we may assume $\beta_{50}=0$.
 Then we can write
\[
f_5(x,y_1)=y_1^4+(c_2+c_3)xy_1^3+c_2c_3x^2y_1^2+\beta_{41}x^4y_1+\beta_{60}x^6
+{\text{(higher terms)}}  \]
 and we have
$\mathcal{N}(f;x,y_1)=y_1^5+\beta_{60}^2x^{12}$
and  note that  $\beta_{60}\ne 0 $ if and only if $\iota=6$.
If $\beta_{60}\ne 0 $, we have $(C,O)\sim B_{12,5}$ and  
if $\beta_{60} = 0$,  we have $\iota \ge 7$.
Secondly we consider that (2): $c_{1}=0$, $ c_{i}\ne0,\ (i=2,3)$.
Then  
\[
f_5(x,y_1)=y_1^4+(c_2+c_3)xy_1^3+c_2c_3x^2y_1^2+\beta_{41}x^4y_1+\beta_{70}x^7
+{\text{(higher terms)}}\]
 and we have
$\mathcal{N}(f;x,y_1)=y_1^5+x^8(\beta_{41}y_1+\beta_{70}x^3)^2$.
The  assertion (2) follows easily by the Newton boundary argument. 
For the assertions (3) and (4), we can consider similarly.   
\end{proof}

\subsubsection{Case IV-(a)} 
Now we classify the singularities in case $IV$.
In this case, we have the following.
\begin{proposition}
Suppose that the tangent cone of the quintic $C_5$ consists of
 distinct four lines. 
\begin{enumerate}
\item If the conic $C_{2}$ is smooth, 
       $(C,O)\sim B_{2\iota,5}$ for $\iota=4$,
      $ B_{2\iota,5}$ or $B_{k,2}\circ B_{6,3}$ $(5\le k \le10)$, for $\iota=5$ 
      and $B_{5\iota-21,2}\circ B_{6,3}$ for $\iota = 6,\cdots,10$.

\item Assume that the conic \ $C_2$ is a union of two lines $\ell_1,\,\ell_2$.
      Putting $\iota_i=I(\ell_i,C_5;O)$ $(i=1,2)$, we have 
      $(C,O)\sim B_{5\iota_1-16,2}\circ {(B_{2,2}^2)}^{2B_{2,2}}
      \circ B_{2,5\iota_2-16}$ with $\iota_1+\iota_2=\iota$. 
\end{enumerate}
\end{proposition} 
\begin{proof}
The  assertion (1) immediately  follows from Lemma \ref{termination8}.
The assertion (2) follows from the Lemma \ref{termination1}. 
\end{proof}
\subsubsection{Case VI-(b)}
In this case, 
we denote components of the tangent tangent cone of $C_5$ by
$L_1$ which is a double line $L_2$ and $L_3$. 

\noindent {\bf{(b-1) }} Assume that $C_2$ is smooth ($m_2=1$). 
Then we have the following.
\begin{proposition}
Suppose that $C_2$ is smooth and 
the tangent cone of the quintic $C_5$ consists of  
a double line  and distinct two line. 
The germ $(C,O) \sim B_{8,5}$ for $\iota=4$.
If $\iota \ge 5$, we have possibility of $(C,O)$. 
\begin{enumerate}
\item We assume that $C_2$ is tangent to $L_2$ or $L_3$. 
         Then  $(C,O)\sim B_{5,10}$
or  $(C,O)\sim  B_{3,6}\circ B_{2,k}$, $(5 \le k \le 10)$ for $\iota=5$ and 
         $ B_{3,6} \circ  B_{2,5\iota-21}$ for $\iota = 6,\cdots,10$. 
\item We assume that $C_2$ is tangent to  $L_1$.
         Then $(C,O)\sim B_{2\iota,5}$ for $\iota=5,6$ and 
        $(C,O)\sim B_{5\iota-28,2} \circ B_{8,3}$ for $\iota=7, \cdots,10$.
\end{enumerate}
\end{proposition}
\begin{proof}
The assertions (1) and (2) are immediate by Lemma \ref{termination4} and Lemma \ref{termination8}.
\end{proof}
\noindent {\bf{(b-2)}}
Assume that $C_2$ is a union of two lines passing through the origin
$\ell_i: a_ix+b_iy, \ (i=1,2) $
 and  
we assume that 
$L_1=\{y=0\}$, $L_2=\{x=0\}$ and $L_3=\{y+cx=0\}$.
\begin{proposition}
Suppose  that $C_2$ is a union of two lines and 
the tangent cone of the quintic $C_5$ consists of  
a double line  and distinct two line. 
Then
\begin{eqnarray*}
\iota=8:\quad (C,O)\sim & B_{6,4}\circ {(B_{1,1}^2)}^{B_{2,2}}\circ B_{2,4},\quad
B_{4,2}\circ B_{3,2} \circ {(B_{1,1}^2)}^{B_{2,2}}\circ B_{2,4},\,\qquad\\
\iota=9:\quad  (C,O)\sim& {(B_{3,2}^2)}^{B_{5,2}}\circ {(B_{1,1}^2)}^{B_{2,2}}\circ B_{2,4},\,\qquad
\ell_1=L_1\qquad\text{or}\\
&B_{6,4}\circ {(B_{1,1}^2)}^{B_{2,2}}\circ B_{2,9},\quad
B_{4,2}\circ B_{3,2} \circ {(B_{1,1}^2)}^{B_{2,2}}\circ B_{2,9},\,
\,\ell_1=L_2\\
\iota=10:\quad (C,O)\sim & {(B_{3,2}^2)}^{B_{5,2}}\circ
 {(B_{1,1}^2)}^{B_{2,2}}\circ B_{2,9},\, \ell_1=L_1,\, \ell_2= L_2\\
&B_{6,4}\circ {(B_{1,1}^2)}^{B_{7,2}}\circ B_{2,9} ,\quad
B_{4,2}\circ B_{3,2} \circ  {(B_{1,1}^2)}^{B_{7,2}}\circ B_{2,9},
\, \ell_1=L_2,\, \ell_2=L_3
\end{eqnarray*}

\end{proposition}
We omit the proof as it follows from an  easy calculation.

\subsubsection{Case IV-(c)} We assume that
 the tangent tangent cone of $C_5$ consists of  a triple line
$L_1$ and a simple  $L_2$.

\noindent {\bf{(c-1)}} Assume that $C_2$ is smooth. Then  
we have  the  following.
\begin{proposition}
Suppose that $C_2$ is smooth and 
$C_5$ be as above.
Then $(C,O)\sim B_{8,5}$ for $\iota=4$.

If $\iota \ge 5$, the possibilities  of $(C,O)$ are: 
\begin{enumerate}
\item  
If $C_2$ is tangent to  $L_2$,
 $(C,O)\sim B_{5,10},\, \text{or}\,\,B_{3,6}\circ B_{2,k}$,
 $(5 \le k \le 10)$ for $\iota=5$ and 
      $(C,O)\sim  B_{3,6} \circ  B_{2,5\iota-21}$ for $\iota = 6,\cdots,10$. 
\item 
If $C_2$ is tangent to  $L_1$,
       $(C,O)\sim B_{2\iota,5}$
      for $\iota=4, \cdots, 8$ and 
$(C,O)\sim B_{5\iota-35,2} \circ B_{10,3}$ for $\iota=9, 10$.
\end{enumerate}
\end{proposition}

The proof of (1) and (2) is immediate from 
Lemma \ref{termination4} and
Lemma \ref{termination8}.  

\noindent {\bf{(c-2)}}
Assume that $C_2$ is a union two lines   $\ell_i,\,i=1,2$  
passing through the origin and put 
        $\ell_i:a_ix+b_iy=0$ $(i=1,2)$.
We assume 
that  $L_1=\{y=0\}$ and $L_2=\{x=0\}$.
We have $\iota \geq 8$. 
\begin{proposition} 
Suppose that $C_2$ is a union two lines $\ell_1$, $\ell_2$ and
the tangent cone of the quintic $C_5$ consists of  
a triple line $L_1$ and a simple line $L_2$. Then
\begin{enumerate}
\item If  $\iota=8$, 
$(C,O)\sim B_{8,6}\circ B_{2,4}$ and $B_{5,3}\circ B_{4,3} \circ B_{2,4}$.
\item Suppose  $\iota=9$.
\begin{enumerate}
\item If $\ell_1=L_1$,
       $(C,O)\sim {(B_{4,3}^2)}^{B_{5,2}}\circ B_{2,4}$ and 
      $B_{5,3}\circ B_{4,3}\circ B_{2,9}$. 
\item If $\ell_2 = L_2$,  $(C,O)\sim B_{8,6}\circ B_{2,9}$.
\end{enumerate}
\item If $\iota=10$, we have $\ell_1=L_1,\, \ell_2 = L_2$ and $(C,O)\sim {(B_{4,3}^2)}^{B_{5,2}}\circ B_{2,9}$.
\end{enumerate}
\end{proposition} 
\subsubsection {Case VI-(d)}

\begin{proposition}
Suppose that the tangent cone of the quintic $C_5$ is line with multiplicity $4$.  
 \begin{enumerate}
\item If $C_2$ is smooth, we have $(C,O)\sim B_{2\iota,5}$ 
      for $\iota = 4,5,\cdots, 10$.
\item Suppose  $C_2$ consists of  two lines.
\begin{enumerate}
\item   If $\iota=8$ we have  $(C,O)\sim B_{10,8}$ and $B_{2,1}\circ B_{4,3}\circ B_{5,4}$.
\item   If $\iota=9$, we have
$(C,O)\sim  {(B_{5,4}^2)}^{B_{5,2}}$. 
\end{enumerate}     
 \end{enumerate}
\end{proposition} 
Note that when $\iota =10$, $C$ is linear tours curve. See \S6.

\subsubsection {Case VI-(e)} In this case,
we may assume that each double line is $L_1$ and $L_2$.\\

\noindent {\bf{(e-1)}} Assume that $C_2$ is smooth case.
\begin{proposition}
Suppose that $C_2$ is smooth and 
the tangent cone of the quintic $C_5$ consists of  
is two double line.
The germ $(C,O)$ can be of type $B_{8,5}$ for $\iota=4$.
If $\iota \ge 5$, we have possibility of $(C,O)$. 
\begin{enumerate}
\item If $C_2$ is tangent to  $L_1$, $(C,O)\sim B_{2\iota,5}$ for $\iota=5,6$ 
      $B_{5\iota-28,2} \circ B_{8,3} $ for $\iota=7, \cdots,10$.
\item If $C_2$ is tangent to $L_2$, $(C,O)\sim B_{5,2\iota}$ for $\iota=5,6$ 
      $B_{3,8}\circ  B_{2,5\iota-28}$ for $\iota=7, \cdots,10$.        
\end{enumerate}
\end{proposition}

\noindent {\bf{(e-2)}} 
$C_2$ is two lines case. 

\begin{proposition}
Suppose that $C_2$ is two lines and 
the tangent cone of the quintic $C_5$ consists of  
is two double line.
Then 
\begin{enumerate}
\item If $\iota=8$, 
$(C,O)\sim B_{6,4} \circ B_{4,6}$, 
$B_{4,2}\circ B_{3,2}\circ B_{4,6}$ and
$B_{4,2}\circ B_{3,2}\circ B_{2,3}\circ B_{2,4}$.
\item If $\iota=9$, 
$(C,O)\sim B_{6,4} \circ {(B_{2,3}^2)}^{B_{5,2}}$ and 
$B_{4,2}\circ B_{3,2}\circ {(B_{2,3}^2)}^{B_{5,2}}$. 
\item If $\iota=10$, 
 $(C,O)\sim {(B_{3,2}^2)}^{B_{5,2}} \circ {(B_{2,3}^2)}^{B_{5,2}} $
\end{enumerate}
\end{proposition}

\subsection{Case V: $m_{5}$ = 5.}
Similarly we also divide this case into  
$C_2$ is smooth or two lines.

We have $\iota \geq 5$.
\begin{proposition}\label{caseV}
Suppose that  the multiplicity of the quintic $C_5$ is $5$, 
i.e., $C_5$ consists of five line components. Then
\begin{enumerate}
\item  If $C_2$ is an irreducible conic,  $(C,O)\sim B_{2\iota,5}$ for
       $\iota=5,\cdots 10$. 
\item  If $C_2$ consists of two lines,
        $f$ is a homogeneous polynomial of degree $10$ and 
  $(C,O) \sim B_{10,10}$, i.e., $C$ consists of 10 line components. 
\end{enumerate}
\end{proposition}

\section{List of Classification}
Now we  have  the following list of  local classification.
\begin{theorem}\label{res}
Let $C=\{f=f_2^5+f_5^2=0\}$ is a (2,5)-torus curve.
We assume that $C_2=\{f_2=0\}$ is a reduced conic. ($C_2$ 
is non-reduced case in next section). 
The topological type of $(C,O)$ can be one of the  following
where $ \dagger (C,O)$ has degenerate series of  $(C,O)$. \\

{\rm{I}}. If $C_5$ is smooth,
 $(C,O)\sim B_{5\iota,2}$, $\iota=1,\cdots, 10$.


{\rm{II}}. Assume $C_5$ is singular. \\

{\rm (II-1)} Assume  $C_2$ is an irreducible conic.

\vspace{0.3cm}
{\small{
\unitlength 0.1in
\begin{picture}( 54.0000, 38.3000)(  4.1000,-40.7000)
%
\special{pn 8}%
\special{pa 410 750}%
\special{pa 5810 750}%
\special{fp}%
\put(5.8000,-32.3000){\makebox(0,0)[lb]{9}}%
\put(9.3000,-7.1000){\makebox(0,0)[lb]{$B_{3,2}\circ B_{2,3}, \ \ B_{5,4}$}}%
%
\special{pn 8}%
\special{pa 410 500}%
\special{pa 5810 500}%
\special{fp}%
%
\special{pn 8}%
\special{pa 410 250}%
\special{pa 410 4060}%
\special{fp}%
%
\special{pn 8}%
\special{pa 810 250}%
\special{pa 810 4060}%
\special{fp}%
\put(5.8000,-6.8000){\makebox(0,0)[lb]{2}}%
\put(9.3000,-10.2000){\makebox(0,0)[lb]{$B_{8,2}\circ B_{2,3}, \ \ {(B_{3,2}^2)}^{B_{3,2}}, \ B_{6,5}$}}%
\put(9.3000,-13.8000){\makebox(0,0)[lb]{$B_{13,2}\circ B_{2,3}, \ \ B_{6,2}\circ B_{4,3}, \ \  B_{8,5}, \ \ \dagger {(B_{4,2}^2)}^{2B_{2,2}}$}}%
\put(9.3000,-17.4000){\makebox(0,0)[lb]{$B_{18,2}\circ B_{2,3}, \ \  B_{11,2}\circ B_{4,3}, \ \  \dagger  B_{10,5}, \ \  \dagger {(B_{4,2}^2)}^{B_{7,2}+B_{2,2}}$}}%
\put(8.7000,-21.2000){\makebox(0,0)[lb]{$B_{23,2}\circ B_{2,3},  \   B_{16,2}\circ B_{4,3}, \  \dagger  B_{9,2}\circ B_{6,3}, \  B_{12,5},\  \dagger {(B_{4,2}^2)}^{B_{12,2}+B_{2,2}}$}}%
\put(8.7000,-24.7000){\makebox(0,0)[lb]{$B_{28,2}\circ B_{2,3},  \   B_{21,2}\circ B_{4,3}, \  \dagger  B_{14,2}\circ B_{6,3}, \ \ \dagger {(B_{4,2}^2)}^{B_{17,2}+B_{2,2}},\ \ B_{7,2}\circ B_{8,3}, \ \ B_{14,5} $}}%
\put(5.8000,-9.7000){\makebox(0,0)[lb]{3}}%
\put(5.8000,-13.5000){\makebox(0,0)[lb]{4}}%
\put(5.8000,-17.5000){\makebox(0,0)[lb]{5}}%
\put(5.8000,-21.1000){\makebox(0,0)[lb]{6}}%
\put(5.8000,-24.8000){\makebox(0,0)[lb]{7}}%
\put(9.2000,-28.3000){\makebox(0,0)[lb]{$B_{33,2}\circ B_{2,3}, \ B_{26,2}\circ B_{4,3},  \  \dagger B_{19,2}\circ B_{6,3}, \  \dagger {(B_{4,2}^2)}^{B_{22,2}+B_{2,2}}, \  B_{12,2}\circ B_{8,3}, \ \ B_{16,5} $}}%
\put(5.8000,-28.4000){\makebox(0,0)[lb]{8}}%
\put(9.2000,-31.9000){\makebox(0,0)[lb]{$B_{38,2}\circ B_{2,3}, \ \ B_{31,2}\circ B_{4,3}, \ \  \dagger B_{24,2}\circ B_{6,3} , \ \ \dagger {(B_{4,2}^2)}^{B_{27,2}+B_{2,2}} $}}%
\put(9.2000,-37.2000){\makebox(0,0)[lb]{$B_{43,2}\circ B_{2,3}, \ \ B_{36,2}\circ B_{4,3}, \ \  \dagger B_{29,2}\circ B_{6,3}, \ \  \dagger {(B_{4,2}^2)}^{B_{32,2}+B_{2,2}}$}}%
\put(9.2000,-39.4000){\makebox(0,0)[lb]{$B_{22,2}\circ B_{8,3}, \ \ B_{15,2}\circ B_{10,3}, \ \ B_{20,5}$}}%
\put(5.3000,-38.3000){\makebox(0,0)[lb]{10}}%
%
\special{pn 8}%
\special{pa 5810 250}%
\special{pa 5810 4070}%
\special{fp}%
%
\special{pn 8}%
\special{pa 410 1080}%
\special{pa 5810 1080}%
\special{fp}%
%
\special{pn 8}%
\special{pa 410 1480}%
\special{pa 5810 1480}%
\special{fp}%
%
\special{pn 8}%
\special{pa 410 1840}%
\special{pa 5810 1840}%
\special{fp}%
%
\special{pn 8}%
\special{pa 410 2210}%
\special{pa 5810 2210}%
\special{fp}%
%
\special{pn 8}%
\special{pa 410 2560}%
\special{pa 5810 2560}%
\special{fp}%
%
\special{pn 8}%
\special{pa 410 2920}%
\special{pa 5810 2920}%
\special{fp}%
%
\special{pn 8}%
\special{pa 410 3460}%
\special{pa 5810 3460}%
\special{fp}%
%
\special{pn 8}%
\special{pa 410 4060}%
\special{pa 5810 4060}%
\special{fp}%
\put(9.2000,-34.1000){\makebox(0,0)[lb]{$B_{17,2}\circ B_{8,3}, \ \ B_{10,2}\circ B_{10,3}, \ \ B_{18,5}$}}%
%
\special{pn 8}%
\special{pa 410 250}%
\special{pa 5810 250}%
\special{fp}%
\put(5.8000,-4.1000){\makebox(0,0)[lb]{$\iota$}}%
\put(28.3000,-4.5000){\makebox(0,0)[lb]{$(C,O)$}}%
\end{picture}%
\\
}}

The singularities with $\dagger$ have further degenerations as is
 indicated below.

\renewcommand{\labelitemi}{$\dagger$}
\begin{itemize}
\item ${(B_{4,2}^2)}^{2B_{2,2}}\ \ : \ \ 
    B_{10,4}, \ \ B_{k,2} \circ B_{2,5} \ \ (6\le k\le 15 )$
\item ${(B_{4,2}^2)}^{B_{7,2}+B_{2,2}}\ \ : \ \
    {(B_{5,2}^2)}^{B_{5,2}} $
\item $ {(B_{4,2}^2)}^{B_{12,2}+B_{2,2}} \ \ : \ \ 
    {(B_{6,2}^2)}^{2B_{3,2}}, \ {(B_{7,2}^2)}^{B_{2,2}},\ B_{15,4}$
\item $ {(B_{4,2}^2)}^{B_{17,2}+B_{2,2}} \ \ : \ \ 
    {(B_{6,2}^2)}^{B_{8,2}+B_{3,2}}, \ {(B_{7,2}^2)}^{B_{7,2}}$
\item $ {(B_{4,2}^2)}^{B_{22,2}+B_{2,2}} \ \ : \ \ 
   {(B_{6,2}^2)}^{B_{13,2}+B_{3,2}}, \ {(B_{8,2}^2)}^{2B_{4,2}}, \  
   {(B_{9,2}^2)}^{B_{5,2}}$\\ \hspace{7cm} $B_{20,4}, \ B_{k,2}\circ B_{10,2} \ (k=11,12)$
\item ${(B_{4,2}^2)}^{B_{27,2}+B_{2,2}} \ \ : \ \ 
   {(B_{6,2}^2)}^{B_{18,2}+B_{3,2}}, \ {(B_{8,2}^2)}^{B_{9,2}+B_{4,2}}, \  
   {(B_{9,2}^2)}^{B_{10,2}} $
\item $ {(B_{4,2}^2)}^{B_{32,2}+B_{2,2}} \ \ : \ \ 
   {(B_{6,2}^2)}^{B_{23,2}+B_{3,2}}, \ 
   {(B_{8,2}^2)}^{B_{13,2}+B_{4,2}}, \  
   {(B_{10,2}^2)}^{2B_{5,2}}, 
   {(B_{11,2}^2)}^{B_{6,2}},\\
   \hspace{3cm}
   {(B_{12,2}^2)}^{2B_{1,2}},\ B_{25,4} $
\item $ B_{10,5}\ \ : \ \ 
    B_{k,2}\circ B_{6,3}\, (5\le k \le 12),\ 
    B_{k,3}\circ B_{4,2}\, (7\le k \le 11),\     
    B_{3,1}\circ B_{5,2}\circ B_{4,2},\\ 
    \hspace{1.7cm} B_{3,1}\circ B_{7,2}\circ B_{4,2},\
    B_{k,2}\circ B_{3,1}\circ B_{4,2}\,(k=7,8,9),\
    B_{k_2+4,2}\circ B_{2,1}\circ (B_{2,1}^2)^{B_{k_1,2}}$,\\
   \hspace{1.6cm} $(k_1,k_2)\in 
    \{(k_1,k_2)\mid k_2-4\le  k_1\le 13-k_2,\, 5\le k_2 \le 7\}\cup\{(4,8),(5,9)\}$.

\item $ B_{9,2}\circ B_{6,3}\ \ : \ \ 
   {(B_{5,2}^2)}^{B_{1,2}} \circ B_{2,1}, \
   B_{9,2}\circ B_{2,1}\circ {(B_{2,1}^2)}^{B_{k,2}} \ \ (1\le k \le 8 ) $
\item $ B_{14,2}\circ B_{6,3}\ \ : \ \ 
   {(B_{5,2}^2)}^{B_{1,2}}\circ B_{2,1}, \ B_{13,2}\circ B_{2,1}, \ 
  B_{14,2}\circ B_{2,1}\circ {(B_{2,1}^2)}^{B_{k,2}} \ \ (1\le k \le 7 ) $
\item $ B_{19,2}\circ B_{6,3}\ \ : \ \ 
   B_{12,2}\circ {(B_{3,1}^2)}^{B_{1,2}} \circ B_{2,1}, \
   {(B_{7,2}^2)}^{B_{4,2} }\circ  B_{2,1},\
 B_{19,2}\circ B_{2,1}\circ {(B_{2,1}^2)}^{B_{k,2}} \ \ (1\le k \le 6 ) $
 
   \item $B_{24,2}\circ B_{6,3}\ \ : \ \ 
   B_{17,2}\circ {(B_{3,1}^2)}^{B_{1,2}} \circ B_{2,1}, \
   {(B_{8,2}^2)}^{2B_{2,2}}\circ  B_{1,2}, \ B_{18,4} \circ B_{2,1}, \\ 
   \hspace{2.5cm}    B_{24,2}\circ B_{2,1}\circ {(B_{2,1}^2)}^{B_{k,2}} \ \ (1\le k \le 5 ) $
  \item $ B_{29,2}\circ B_{6,3}\ \ : \ \ 
   B_{22,2}\circ {(B_{3,1}^2)}^{B_{1,2}} \circ B_{2,1}, \
   B_{22,2}\circ {(B_{4,1}^2)}^{B_{2,2}}\circ B_{2,1}, \
   {(B_{9,2}^2)}^{B_{5,2}}\circ B_{2,1}, \\
\hspace{2.5cm}   B_{29,2}\circ B_{2,1}\circ {(B_{2,1}^2)}^{B_{k,2}} \ \  (1\le k \le 5, \ k\ne 4) $
    \end{itemize}

{\rm (II-2)}  Assume $C_2$  is two distinct lines.\\

{\small{
\unitlength 0.1in
\begin{picture}( 54.0000, 44.6300)(  4.1000,-46.1600)
%
\special{pn 8}%
\special{pa 410 170}%
\special{pa 5810 170}%
\special{fp}%
%
\special{pn 8}%
\special{pa 410 170}%
\special{pa 410 4616}%
\special{fp}%
%
\special{pn 8}%
\special{pa 780 170}%
\special{pa 780 4616}%
\special{fp}%
\put(5.7000,-5.8400){\makebox(0,0)[lb]{4}}%
\put(8.8000,-8.6300){\makebox(0,0)[lb]{$B_{13,2}\circ B_{2,8}, \ \ {(B_{3,2}^2)}^{B_{13,2}}$}}%
\put(5.7000,-8.2700){\makebox(0,0)[lb]{5}}%
\put(8.8000,-11.1500){\makebox(0,0)[lb]{$B_{18,2}\circ B_{2,8}, \ \ B_{13,2}\circ B_{2,13}, \ \  \dagger {(B_{4,2}^2)}^{2B_{7,2}}$}}%
\put(8.8000,-13.1300){\makebox(0,0)[lb]{$B_{6,2}\circ {(B_{1,1}^2)}^{B_{4,2}}\circ B_{2,6}, \ \ \dagger {(B_{3,2}^2)}^{B_{4,2}}\circ B_{2,6}, \ \  \dagger {(B_{4,3}^2)}^{B_{6,2}}$}}%
\put(5.7000,-12.0500){\makebox(0,0)[lb]{6}}%
\put(8.8000,-15.8300){\makebox(0,0)[lb]{$B_{23,2}\circ B_{2,8}, \ \ B_{18,2}\circ B_{2,13}, \ \ \dagger B_{16,2}\circ{(B_{2,1}^2)}^{B_{7,2}}$}}%
\put(8.8000,-17.9900){\makebox(0,0)[lb]{$B_{11,2}\circ {(B_{1,1}^2)}^{B_{4,2}}\circ B_{2,6}, \ \ {(B_{3,2}^2)}^{B_{9,2}}\circ B_{2,6}, \ \ \dagger {(B_{3,2}^2)}^{B_{4,2}}\circ B_{2,11},\,\ {(B_{4,3}^2)}^{B_{11,2}}$}}%
\put(5.6000,-16.6400){\makebox(0,0)[lb]{7}}%
\put(8.8000,-21.0500){\makebox(0,0)[lb]{$B_{23,2}\circ B_{2,13}, \ B_{18,2}\circ B_{2,18}, \ B_{11,2}\circ{(B_{1,1}^2)}^{B_{4,2}}\circ B_{2,11},\ B_{6,4}\circ B_{6,4}, B_{10,8}$}}%
\put(8.8000,-25.6400){\makebox(0,0)[lb]{$B_{4,2}\circ {(B_{2,2}^2)}^{2B_{2,2}}\circ B_{2,4}, \ \dagger B_{6,4}\circ {(B_{1,1}^2)}^{2B_{2,2}}\circ B_{2,4}, \ \dagger B_{8,6}\circ B_{2,4}, \ B_{4,2}\circ {(B_{3,2}^2)}^{B_{7,2}}$}}%
%
\special{pn 8}%
\special{pa 410 3644}%
\special{pa 5810 3644}%
\special{fp}%
\put(5.5000,-23.3900){\makebox(0,0)[lb]{8}}%
\put(8.7000,-31.0400){\makebox(0,0)[lb]{${(B_{3,2}^2)}^{B_{9,2}}\circ B_{2,16}, \ \ \dagger {(B_{4,2}^2)}^{2B_{5,2}}\circ B_{2,11}, \ \ B_{9,2}\circ {(B_{2,2}^2)}^{2B_{2,2}}\circ B_{2,4}$}}%
\put(8.7000,-33.1100){\makebox(0,0)[lb]{$\dagger B_{6,4}\circ {(B_{1,1}^2)}^{B_{2,2}}\circ B_{2,9}, \ \ {(B_{3,2}^2)}^{B_{5,2}}\circ {(B_{1,1}^2)}^{B_{2,2}}\circ B_{2,4}, \ \ \dagger B_{6,4}\circ {(B_{1,1}^2)}^{B_{7,2}}\circ B_{2,4}$}}%
\put(5.4000,-31.6700){\makebox(0,0)[lb]{9}}%
\put(8.4000,-38.7800){\makebox(0,0)[lb]{$B_{23,2}\circ B_{2,23}, \ \ B_{16,2}\circ{(B_{1,1}^2)}^{B_{4,2}}\circ B_{2,16}$}}%
\put(8.4000,-40.7600){\makebox(0,0)[lb]{$\dagger {(B_{4,2}^2)}^{2B_{5,2}}\circ B_{2,16}, \ \ B_{9,2}\circ{(B_{2,2}^2)}^{2B_{2,2}}\circ B_{2,9}, \ \ \dagger B_{6,4}\circ {(B_{1,1}^2)}^{B_{7,2}}\circ B_{2,9}$}}%
\put(8.4000,-43.0100){\makebox(0,0)[lb]{${(B_{3,2}^2)}^{B_{5,2}}\circ {(B_{1,1}^2)}^{B_{2,2}}\circ B_{2,9}, \ \ B_{6,4}\circ {(B_{1,1}^2)}^{B_{7,2}}\circ B_{2,4}$}}%
\put(5.1000,-42.1100){\makebox(0,0)[lb]{10}}%
\put(8.7000,-28.8800){\makebox(0,0)[lb]{$B_{23,2}\circ B_{2,18}, \ \ B_{16,2}\circ{(B_{1,1}^2)}^{B_{4,2}}\circ B_{2,11}$}}%
%
\special{pn 8}%
\special{pa 410 648}%
\special{pa 5810 648}%
\special{fp}%
%
\special{pn 8}%
\special{pa 5810 170}%
\special{pa 5810 3734}%
\special{fp}%
\special{pa 5810 3734}%
\special{pa 5810 4616}%
\special{fp}%
%
\special{pn 8}%
\special{pa 410 900}%
\special{pa 5810 900}%
\special{fp}%
%
\special{pn 8}%
\special{pa 410 1368}%
\special{pa 5810 1368}%
\special{fp}%
%
\special{pn 8}%
\special{pa 410 1862}%
\special{pa 5810 1862}%
\special{fp}%
%
\special{pn 8}%
\special{pa 410 2654}%
\special{pa 5810 2654}%
\special{fp}%
\put(8.7000,-35.4500){\makebox(0,0)[lb]{$\dagger {(B_{4,3}^2)}^{B_{5,2}}\circ B_{2,4}, \ \  {(B_{5,4}^2)}^{B_{5,2}}, \ \ \dagger B_{6,4}\circ {(B_{2,3}^2)}^{B_{5,2}},\ \ B_{8,6}\circ B_{2,9}$}}%
\put(8.3000,-45.3500){\makebox(0,0)[lb]{${(B_{4,3}^2)}^{B_{5,2}}\circ B_{2,9}, \ \ {(B_{3,2}^2)}^{B_{5,2}}\circ {(B_{2,3}^2)}^{B_{5,2}} $}}%
%
\special{pn 8}%
\special{pa 410 4616}%
\special{pa 5810 4616}%
\special{fp}%
\put(8.8000,-23.3000){\makebox(0,0)[lb]{$B_{16,2}\circ {(B_{1,1}^2)}^{B_{4,2}}\circ B_{2,6},  \ {(B_{3,2}^2)}^{B_{9,2}}\circ B_{2,11}, \ {(B_{3,2}^2)}^{B_{4,2}}\circ B_{2,16}, \  \dagger {(B_{4,2}^2)}^{2B_{5,2}}\circ B_{2,6}$}}%
\put(8.8000,-6.1100){\makebox(0,0)[lb]{$B_{8,2}\circ B_{2,8}, \ \ \dagger {(B_{3,2}^2)}^{B_{8,2}}$}}%
%
\special{pn 8}%
\special{pa 410 396}%
\special{pa 5810 396}%
\special{fp}%
\put(5.7000,-3.2300){\makebox(0,0)[lb]{$\iota$}}%
\put(28.5000,-3.4100){\makebox(0,0)[lb]{$(C,O)$}}%
\end{picture}%
\\
}}

The singularities with $\dagger$ have further degenerations as is
  indicated below.
\begin{itemize} 
 \item $ {(B_{3,2}^2)}^{B_{8,2}}\ \ : \ \ 
    {(B_{4,2}^2)}^{ 2B_{2,2}}, \ \ B_{10,4}, \ \ B_{k,2} \circ B_{2,5} \ \ (6\le k \le 15 ) \hspace{4cm} $
 \item $  {(B_{4,2}^2)}^{ 2B_{7,2}}\ \ : \ \ 
    {(B_{5,2}^2)}^{B_{10,2}}, \ {(B_{6,2}^2)}^{2B_{3,2}}, \
    {(B_{7,2}^2)}^{B_{2,2}}, \ B_{15,4} $

 \item $  B_{16,2}\circ {(B_{2,1}^2)}^{B_{7,2}}\ \ : \ \ 
    {(B_{5,2}^2)}^{B_{15,2}} $

 \item $ {(B_{3,2}^2)}^{B_{4,2}}\circ B_{2,6},\quad B_{8,4}\circ B_{2,6},\quad 
 B_{k,2}\circ B_{4,2}\circ B_{2,6}\ \ (5\le k\le 12)$
 
 \item $  {(B_{3,2}^2)}^{B_{4,2}}\circ B_{2,11}\ \ : \ \ 
   B_{8,4}\circ B_{2,11}, \quad B_{k,2}\circ B_{4,2}\circ B_{2,11}\ \ (5\le k\le 10)$

 \item $ {(B_{4,2}^2)}^{2B_{5,2}}\circ B_{2,6}\ \ : \ \ 
   {(B_{5,2}^2)}^{B_{6,2}}\circ B_{2,6}, \ 
    {(B_{6,2}^2)}^{2B_{1,2}} \circ B_{2,6}, \quad B_{13,4}\circ B_{2,6}$
    
 \item $ {(B_{4,2}^2)}^{2B_{5,2}}\circ B_{2,11}\ \ : \ \ 
   {(B_{5,2}^2)}^{B_{6,2}}\circ B_{2,11}, \ 
   {(B_{6,2}^2)}^{2B_{1,2}} \circ B_{2,11}, \quad B_{13,4}\circ B_{2,11}$
   
 \item $  {(B_{4,2}^2)}^{2B_{5,2}}\circ B_{2,16}\ \ : \ \ 
   {(B_{5,2}^2)}^{B_{6,2}}\circ B_{2,16}, \ 
   {(B_{6,2}^2)}^{2B_{1,2}} \circ B_{2,16}, \quad B_{13,4}\circ B_{2,16}$
   
 \item $   {(B_{4,3}^2)}^{B_{6,2}}\ \ : \ \ 
                 B_{10,6}, \quad B_{6,3}\circ B_{5,3}$
 \item $  B_{6,4}\circ {(B_{1,1}^2)}^{B_{2,2}}\circ B_{2,9}\ \ : \ \ 
B_{4,3}\circ B_{3,2}\circ {(B_{1,1}^2)}^{B_{2,2}}\circ B_{2,9}$               
 \item $  B_{6,4}\circ {(B_{1,1}^2)}^{B_{7,2}}\circ B_{2,4}\ \ : \ \ 
B_{4,3}\circ B_{3,2}\circ {(B_{1,1}^2)}^{B_{7,2}}\circ B_{2,4}$
 \item $  B_{6,4}\circ {(B_{1,1}^2)}^{B_{7,2}}\circ B_{2,9}\ \ : \ \ 
B_{4,3}\circ B_{3,2}\circ {(B_{1,1}^2)}^{B_{7,2}}\circ B_{2,9}$
 \item $  B_{8,6}\circ B_{2,4}\ \ : \ \ 
B_{5,4}\circ B_{4,3}\circ B_{2,4}$
 \item $  B_{8,6}\circ B_{2,9}\ \ : \ \ 
 B_{5,4}\circ B_{4,3}\circ B_{2,9}$
 \item $  B_{6,4}\circ B_{4,6}\ \ : \ \ 
B_{4,2}\circ B_{3,2}\circ B_{2,6},\quad B_{4,2}\circ B_{3,2}\circ B_{2,3}\circ B_{2,4}
$
 \item $B_{6,4}\circ {(B_{2,3}^2)}^{B_{5,2}}\ \ : \ \ 
B_{4,2}\circ B_{3,2}\circ {(B_{2,3}^2)}^{B_{5,2}}$
    \end{itemize}                             
                 
\end{theorem}
\renewcommand{\labelitemi}{$\bullet$}
 
\section{LINEAR TORUS CURVE OF TYPE (2,5)}
\begin{definition}
Let $C$ be a torus curve of type $(2,5)$ 
which has a defining polynomial $f$ which can be 
written as $f(x,y)={f_2(x,y)}^5+{f_5(x,y)}^2$ where $\deg f_j=j$ $(j=2,5)$.
If $f_2(x,y)=\ell(x,y)^2$ for some linear form $\ell$,
then the curve $C$ is called a linear torus curve. 
\end{definition}
Let $\ell (x,y) = (ax+by+c)^2$. We may assume that 
$f_2(x,y)=-y^2$ by a linear change of coordinates so that 
$f(x,y)$ is a product of quintic forms $f(x,y)=(f_5(x,y)+y^5))(f_5(x,y)-y^5)$.  
It is easy to observe that 
the inner singularities of $C$ are on $\{y=0\} \cap C_5$.
\subsection{Local Classification}
In this section we determine local singularity of   
linear torus curve type (2,5). 
Similarly, 
we divide into five cases.
If $C_5$ is smooth, we already have the singularity $(C,O)$
form Lemma \ref{Lemma-BT}.
Hence we consider that 
the multiplicity of $C_5$ is larger than 2. 

\subsection{Case L-II: $m_{5}$ = 2.}
In this case, the tangent cone of $C_5$ has two types and we have $\iota=4,6,8,10$. 
\begin{proposition}
Suppose  $m_5=2$. 

\begin{enumerate}
\item If the tangent cone of $C_5$ consists of  two distinct  lines,
$(C,O)\sim B_{5\iota-12,2}\circ B_{2,8}$ for $\iota=4,6,8,10$.

\item Assume that  the tangent cone of $C_5$ is a line with multiplicity 2. 
\begin{enumerate}
\item If $(C_5,O)\sim B_{3,2}$,   
\begin{enumerate}
\item $(C,O)\sim {(B_{3,2}^2)}^{B_{8,2}}$ for $\iota=4$ and 
\item $(C,O)\sim {(B_{3,2}^2)}^{B_{18,2}}$ for $\iota=6$.
\end{enumerate}
\item If $(C_5,O)\sim B_{4,2}$,   
\begin{enumerate}
\item  $(C,O)\sim {(B_{4,2}^2)}^{2B_{2,2}}$ for $\iota=4$,
\item  $(C,O)\sim {(B_{4,2}^2)}^{2B_{12,2}}$ for $\iota=8$ and
\item  $(C,O)\sim {(B_{4,2}^2)}^{B_{22,2}+B_{12,2}}$ for $\iota=10$.
\end{enumerate}
\item If $(C_5,O)\sim B_{5,2}$,   
\begin{enumerate}
\item  $(C,O)\sim B_{10,4}$ and $B_{k,2}\circ B_{2,5}\ (6\le k \le 15)$ for $\iota=4$
\item  $(C,O)\sim{(B_{5,2}^2)}^{B_{20,2}}$ for $\iota=8$. 
\item  $(C,O)\sim{(B_{5,2}^2)}^{B_{30,2}}$ for $\iota=10$. 
\end{enumerate}
\item If $(C_5,O)\sim B_{6,2}$,   
$(C,O)\sim {(B_{6,2}^2)}^{2B_{8,2}}$ for $\iota=8$.
\item If $(C_5,O)\sim B_{7,2}$,   
$(C,O)\sim {(B_{7,2}^2)}^{B_{12,2}}$ for $\iota=8$.
\item If $(C_5,O)\sim B_{8,2}$,   
$(C,O)\sim {(B_{8,2}^2)}^{2B_{4,2}}$ for $\iota=8$.
\item If $(C_5,O)\sim B_{9,2}$,   
$(C,O)\sim {(B_{9,2}^2)}^{B_{4,2}}$ for $\iota=8$.
\item If $(C_5,O)\sim B_{10,2}$,   
$(C,O)\sim B_{20,5}$ and $B_{k,2}\circ B_{10,3}\ (11\le k \le 13)$ for $\iota=8$.
\end{enumerate}
\end{enumerate}
\end{proposition}
\begin{proof}
The assertion (1) is shown by the  Newton boundary argument
(cf. Lemma \ref{termination1}).
The assertion (2) is mainly computational.
See that Proposition \ref{caseII-2}. 
 \end{proof}
\subsection{Case L-III: $m_{5}$ = 3.}
In this case,  the tangent cone of $C_5$ has three types and we have 
$\iota=2k, (k=3, 4 ,5)$.

\begin{proposition}
Suppose $m_5=3$. 

\begin{enumerate}
\item If the tangent cone of $C_5$ consists of three  distinct  lines.
$(C,O)\sim {(B_{3,3}^2)}^{3B_{4,2}}$ and $B_{16,2}\circ {(B_{2,2}^2)}^{2B_{4,2}}$
and $B_{26,2}\circ {(B_{2,2}^2)}^{2B_{4,2}}$
for $\iota=6,8,10$.

\item Suppose the tangent cone of $C_5$ is a double line and a single line. 
\begin{enumerate}
\item If $ \iota =6$,
      $(C,O)\sim B_{6,2}\circ {(B_{2,3}^2)}^{B_{4,2}}$, 
      $B_{8,4}\circ B_{2,6}$ and 
      $B_{k,2}\circ B_{4,2}\circ B_{2,6}, \ (5\le k \le 10)$.
\item $\iota=8$.
   \begin{enumerate}
   \item If the tangent cone of $C_5$ is $\{xy^2=0\}$, then 
   $(C,O)\sim {(B_{3,2}^2)}^{B_{14,2}} \circ B_{2,6}$.
   \item If the tangent cone of $C_5$ is $\{x^2y=0\}$, then 
   $(C,O)\sim B_{16,2}\circ {(B_{2,3}^2)}^{B_{4,2}}$,
   $B_{16,2}\circ B_{4,8}$ and 
   $B_{16,2}\circ B_{2,4}\circ B_{2,k},\ ({5\le k\le 10})$.
   \end{enumerate}
\item  $\iota=10$. 
   \begin{enumerate}
   \item If the tangent cone of $C_5$ is $\{xy^2=0\}$, then 
   $(C,O)\sim {(B_{4,2}^2)}^{2B_{10,2}} \circ B_{2,6}$,
   ${(B_{5,2}^2)}^{B_{16,2}} \circ B_{2,6}$, 
   ${(B_{6,2}^2)}^{2B_{6,2}} \circ B_{2,6}$,
   ${(B_{8,2}^2)}^{2B_{2,2}} \circ B_{2,6}$, 
   $B_{18,2} \circ B_{2,6}$ and $B_{19,2} \circ B_{2,6}$.  
   \item  If the tangent cone of $C_5$ is $\{x^2y=0\}$, then 
   $(C,O)\sim B_{26,2}\circ {(B_{2,3}^2)}^{B_{4,2}}$,
   $B_{26,2}\circ B_{4,8}$ and $B_{26,2}\circ B_{2,4}\circ B_{2,k},\ ({5\le k\le 9})$.
   \end{enumerate}
   \end{enumerate}
\item Suppose  the tangent cone of $C_5$ is a line with triple line.
\begin{enumerate} 
   \item If $\iota=6$, 
    $(C,O)\sim {(B_{4,3}^2)}^{B_{6,2}}$,
   $B_{4,2}\circ {(B_{3,2}^2)}^{B_{2,2}}$,
   $B_{10,6}$ and $B_{3,6}\circ B_{5,3}$.
   \item If $\iota=8$, $(C,O)\sim {(B_{4,3}^2)}^{B_{16,2}}$.
   \item If $\iota=10$,  $(C,O)\sim B_{4,2}\circ {(B_{3,2}^2)}^{B_{12,2}}$
   and ${(B_{5,3}^2)}^{B_{20,2}}$.
 \end{enumerate}    
\end{enumerate}
\end{proposition}
\subsection{Case L-IV: $m_{5}$ = 4.}
\begin{proposition}
Suppose that the multiplicity of the quintic $C_5$ is 4. 

\begin{enumerate}
\item Assume  the tangent cone of $C_5$ is distinct four lines then
\begin{enumerate}
     \item If $\iota =8 $, 
     then $(C,O)\sim {(B_{4,4}^2)}^{ 4B_{2,2}}$.
     \item If $\iota=10$, then
     $(C,O)\sim B_{14,2}\circ {(B_{3,3}^2)}^{3B_{2,2}}$.
     \end{enumerate}
\item Assume the tangent cone of $C_5$ is a double line and distinct two
      lines.
     \begin{enumerate}
     \item If $\iota=8$, 
       $(C,O)\sim {(B_{2,2}^2)}^{2B_{2,2}}\circ B_{4,6}$ and 
     ${(B_{2,2}^2)}^{2B_{2,2}}\circ B_{2,3}\circ B_{2,4}$
     \item Suppose  $\iota=10$. 
\begin{enumerate}
    \item If the tangent cone of $C_5$ is $\{x^2y(y+cx)=0\}$,
            $(C,O)\sim B_{14,2}\circ {(B_{1,1}^2)}^{B_{2,2}}\circ B_{4,6}$ and 
           $B_{14,2}\circ {(B_{1,1}^2)}^{B_{2,2}}\circ B_{2,3}\circ B_{2,4}$.
    \item  If the tangent cone of $C_5$ is $\{xy^2(y+cx)=0\}$,
           $(C,O)\sim {(B_{3,2}^2)}^{B_{10,2}}\circ {(B_{2,2}^2)}^{ 2B_{2,2}}\circ B_{2,4}$. 
   \end{enumerate}     
   \end{enumerate}
\item  Assume the  tangent cone of $C_5$ is a triple line and another line. 
 \begin{enumerate}
     \item If $\iota =8$, 
           then $(C,O)\sim B_{4,2}\circ B_{6,8}$ and 
           $B_{4,2}\circ B_{3,4}\circ B_{2,3}\circ B_{1,2}$.
     \item If $\iota=10$,
     \begin{enumerate}
     \item The tangent cone of $C_5$ is $\{x^3y=0\}$,
           we have $B_{14,2}\circ B_{6,8}$ and $B_{14,2}\circ B_{3,4}\circ B_{2,3}\circ B_{1,2}$.
     \item The tangent cone of $C_5$ is $\{xy^3=0\}$,
           we have ${(B_{4,3}^2)}^{B_{10,2}}\circ B_{2,4}$.
     \end{enumerate}
 \end{enumerate}  
    \item   Assume the tangent cone of $C_5$ is a line with multiplicity 4.
 \begin{enumerate}
    \item If $\iota=8$,
           then $(C,O)\sim B_{8,10}$ and $B_{4,5}\circ B_{3,4}\circ B_{1,2}$.
    \item If $\iota=10$, we have $(C,O)\sim {(B_{4,5}^2)}^{B_{10,2}}$.
 \end{enumerate}     
    \item  Assume the tangent cone of $C_5$ consists of  two double lines. 
\begin{enumerate}
     \item If $\iota=8$, 
      then $(C,O)\sim {(B_{2,2}^4)}^{2B_{10,2}}$, 
           ${(B_{1,1}^4)}^{B_{10,2}}\circ B_{4,6}$ and 
           ${(B_{1,1}^4)}^{B_{10,2}}\circ B_{2,3}\circ B_{2,4}$.
     \item If $\iota=10$, we have 
           $(C,O)\sim {(B_{3,2}^2)}^{B_{10,2}}\circ B_{4,6}$ and 
           ${(B_{3,2}^2)}^{B_{10,2}}\circ B_{2,3}\circ B_{2,4}$.  
\end{enumerate}     
\end{enumerate}
\end{proposition}
\renewcommand{\labelitemi}{$\dagger$}
Putting together the above classifications, we have:
\begin{theorem}
Let $C$ be a linear torus curve type (2,5).
The $(C,O)$ is described as follows.

\vspace{0.3cm}
\begin{center}
 \small{
\unitlength 0.1in
\begin{picture}( 48.0000, 32.2000)(  4.1000,-34.0000)
\put(4.9000,-3.5000){\makebox(0,0)[lb]{$\iota$}}%
\put(7.0000,-6.0000){\makebox(0,0)[lb]{$B_{10,2}$}}%
%
\special{pn 8}%
\special{pa 410 180}%
\special{pa 5210 180}%
\special{fp}%
%
\special{pn 8}%
\special{pa 410 400}%
\special{pa 5210 400}%
\special{fp}%
%
\special{pn 8}%
\special{pa 410 180}%
\special{pa 410 3400}%
\special{fp}%
%
\special{pn 8}%
\special{pa 610 180}%
\special{pa 610 3400}%
\special{fp}%
\put(4.8000,-5.8000){\makebox(0,0)[lb]{2}}%
%
\special{pn 8}%
\special{pa 410 3400}%
\special{pa 5210 3400}%
\special{fp}%
%
\special{pn 8}%
\special{pa 5210 180}%
\special{pa 5210 3400}%
\special{fp}%
%
\special{pn 8}%
\special{pa 410 650}%
\special{pa 5210 650}%
\special{fp}%
\put(7.0000,-8.6000){\makebox(0,0)[lb]{$B_{20,2}, \ \ B_{8,2}\circ B_{2,8}, \ \ \dagger {(B_{3,2}^2)}^{B_{8,2}}$}}%
%
\special{pn 8}%
\special{pa 410 900}%
\special{pa 5210 900}%
\special{fp}%
\put(7.0000,-11.3000){\makebox(0,0)[lb]{$B_{30,2}, \ \ B_{18,2}\circ B_{2,8}, \ \ \dagger {(B_{3,2}^2)}^{B_{18,2}}$}}%
\put(7.0000,-13.7000){\makebox(0,0)[lb]{$ \dagger B_{6,2}\circ  {(B_{2,3}^2)}^{B_{4,2}} , \ \  \dagger {(B_{4,3})}^{B_{6,2}},\ \ \dagger {(B_{3,3}^2)}^{3B_{4,2}}$}}%
%
\special{pn 8}%
\special{pa 410 1430}%
\special{pa 5210 1430}%
\special{fp}%
\put(7.0000,-16.7000){\makebox(0,0)[lb]{$B_{40,2} , \ \ B_{28,2}\circ B_{2,8} , \ \ \dagger {(B_{4,2}^2)}^{2B_{12,2}} , \ \ \dagger B_{8,10}, \ \  {(B_{4,3}^2)}^{B_{16,2}}, \ \ {(B_{4,4}^2)}^{4B_{2,2}} $}}%
\put(7.0000,-19.0000){\makebox(0,0)[lb]{${(B_{3,2}^2)}^{B_{14,2}}\circ B_{2,6}, \ \ \dagger B_{16,2} \circ {(B_{2,3}^2)}^{B_{4,2}},\ \ \dagger {(B_{2,2}^2)}^{2B_{2,2}}\circ B_{4,6}$}}%
\put(7.0000,-21.3000){\makebox(0,0)[lb]{$\dagger B_{4,2}\circ B_{6,8}, \ \ \dagger {(B_{2,2}^4)}^{2B_{10,2}},\ \ B_{16,2}\circ {(B_{2,2}^2)}^{2B_{4,2}}$}}%
%
\special{pn 8}%
\special{pa 410 2210}%
\special{pa 5210 2210}%
\special{fp}%
\put(7.0000,-24.3000){\makebox(0,0)[lb]{$B_{50,2} , \ \ B_{38,2}\circ B_{2,8}, \ \  \dagger {(B_{4,2}^2)}^{B_{22,2}+B_{12,2}}, \ \ {(B_{4,3}^2)}^{B_{10,2}}\circ B_{2,4}$}}%
\put(7.0000,-27.3000){\makebox(0,0)[lb]{$\dagger {(B_{4,2}^2)}^{2B_{10,2}}\circ B_{2,6}, \ \ B_{26,2} \circ {(B_{2,3}^2)}^{B_{4,2}}, \ \ B_{10,10},\ \ \dagger B_{14,2}\circ {(B_{1,1}^2)}^{B_{2,2}}\circ B_{4,6}$}}%
\put(7.0000,-30.1000){\makebox(0,0)[lb]{$ \dagger B_{4,2}\circ {(B_{3,2}^2)}^{B_{12,2}}, \ \ B_{14,2}\circ {(B_{3,3}^2)}^{3B_{2,2}}, \ \ {(B_{4,5}^2)}^{B_{10,2}}, \ \  \dagger B_{14,2}\circ B_{6,8}$}}%
\put(4.8000,-8.3000){\makebox(0,0)[lb]{4}}%
\put(4.8000,-12.4000){\makebox(0,0)[lb]{6}}%
\put(4.8000,-18.5000){\makebox(0,0)[lb]{8}}%
\put(4.3000,-27.2000){\makebox(0,0)[lb]{10}}%
\put(24.9000,-3.5000){\makebox(0,0)[lb]{$(C,O)$}}%
\put(6.9000,-33.1000){\makebox(0,0)[lb]{$\dagger {(B_{3,2}^2)}^{B_{10,2}}\circ {(B_{2,2}^2)}^{2B_{2,2}}\circ B_{2,4}, \ \ B_{26,2}\circ {(B_{2,2}^2)}^{2B_{4,2}}, \ \  {(B_{4,5}^2)}^{B_{10,2}}    $}}%
\end{picture}%
}
\end{center}
\begin{itemize} 
 \item ${(B_{3,2}^2)}^{B_{8,2}}\ \ : \ \ 
         B_{10,4}, \ \ {(B_{4,2}^2)}^{2B_{2,2}}, \ \  
         B_{k,2} \circ B_{5,2} \ \ (6\le k \le 15)$
 \item ${(B_{3,2}^2)}^{B_{18,2}}\ \ : \ \ 
       {(B_{5,2}^2)}^{B_{20,2}}, \ \ {(B_{6,2}^2)}^{2B_{8,2}}, \ \
       {(B_{7,2}^2)}^{B_{12,2}}, \ \ {(B_{8,2}^2)}^{2B_{4,2}}, \ \
       {(B_{9,2}^2)}^{B_{4,2}},\\ \ B_{20,5}, \ \
        B_{k,2}\circ B_{10,3}\ (1\le k \le 13)$
 \item ${(B_{4,2}^2)}^{B_{22,2}+B_{12,2}}\ \ : \ \
        {(B_{5,2}^2)}^{B_{30,2}}$
 \item $B_{6,2}\circ {(B_{2,3}^2)}^{B_{4,2}}\ \ : \ \
 B_{8,4}\circ B_{2,6},\quad B_{k,2}\circ B_{4,2}\circ B_{2,6}\ (5\le k \le 10)$
 \item $B_{16,2}\circ {(B_{2,3}^2)}^{B_{4,2}}\ \ : \ \
 B_{16,2}\circ B_{4,8},\quad B_{16,2}\circ B_{2,4}\circ B_{2,k}\ (5\le k \le 10) $
 \item $B_{26,2}\circ {(B_{2,3}^2)}^{B_{4,2}}\ \ : \ \
 B_{26,2}\circ B_{4,8},\quad B_{26,2}\circ B_{2,4}\circ B_{2,k}\ (5\le k \le 9) $
 \item $  {(B_{4,2}^2)}^{2B_{10,2}}\circ B_{2,6}\ \ : \ \ 
          {(B_{5,2}^2)}^{B_{16,2}}\circ B_{2,6}, \ \
          {(B_{6,2}^2)}^{2B_{6,2}}\circ B_{2,6}, \ \ 
          {(B_{8,2}^2)}^{2B_{2,2}}\circ B_{2,6}, \ \ 
          \\ B_{18,2}\circ B_{2,6}, \ \ B_{19,2}\circ B_{2,6}$
 \item $  {(B_{4,3}^2)}^{B_{6,2}}\ \ : \ \
  B_{4,2}\circ {(B_{3,2}^2)}^{B_{2,2}}, \ \ B_{10,6}, 
  \ \  B_{6,3}\circ B_{5,3}$
 \item $  B_{4,2}\circ {(B_{3,2}^2)}^{B_{12,2}}\ \ : \ \
{(B_{5,3}^2)}^{B_{20,2}}$
 \item ${(B_{2,2}^2)}^{2B_{2,2}}\circ B_{4,6}\ \ : \ \
 B_{2,4}\circ B_{2,3}\circ {(B_{2,2}^2)}^{2B_{2,2}}\circ B_{2,3}\circ B_{2,4}$
\item $B_{14,2}\circ {(B_{1,1}^2)}^{B_{2,2}}\circ B_{4,6}
\ \ : \ \
B_{14,2}\circ {(B_{1,1}^2)}^{B_{2,2}}\circ B_{2,3}\circ B_{2,4}$
\item $B_{4,2}\circ B_{6,8}\ \ : \ \
      B_{4,2}\circ B_{3,4}\circ B_{2,3}\circ B_{1,2}$.
\item $B_{14,2}\circ B_{6,8}\ \ : \ \    
B_{14,2}\circ B_{3,4}\circ B_{2,3}\circ B_{1,2}$ 
\item $B_{8,10}\ \ : \ \ B_{4,5}\circ B_{3,4}\circ B_{1,2}$
\item ${(B_{2,2}^4)}^{2B_{10,2}}\ \ : \ \ 
    ,{(B_{1,1}^4)}^{B_{10,2}}\circ B_{4,6}, \ \ {(B_{1,1}^4)}^{B_{10,2}}\circ B_{2,3}\circ B_{2,4}$
\item ${(B_{3,2}^2)}^{B_{10,2}}\circ B_{4,6}\ \ : \ \
       {(B_{3,2}^2)}^{B_{10,2}}\circ B_{2,3}\circ B_{2,4}$      
\end{itemize}
\end{theorem}
\renewcommand{\labelitemi}{$\bullet$}
\section{Appendix}
In this section, 
we give some examples of singularities which were obtained in the
previous section.\\

\noindent{\bf{Example-Case I.}}
We assume that the quintic $C_5$ is smooth at $O$.
Then we have $(C,O)\sim B_{5\iota,2},\,1\le \iota \le 10$.
The following example
\[C:\,f(x,y)=(y^5+y+x^2)^2+(y+x^2)^5\]
corresponds to $\iota=10$ and $(C,O)\sim A_{49}$ and Milnor number $\mu=49$.
This is due to \cite{BenoitTu}.\\

\noindent{\bf{Example-Case II-(a).}}
We assume that the tangent cone of the quintic $C_5$
consists of distinct two lines.
\begin{equation*}
\begin{split}
 &{\bf{(a{\text{-}}1)}}\quad C:\,f(x,y)=(y-x^2)^5+(y^5+xy-x^3)^2, 
\ \ (C,O)\sim B_{43,2}\circ B_{2,3},\  \iota=10, \ \mu=51.\\&
{\bf{(a{\text{-}}2)}}\quad C:\,f(x,y)=x^5y^5+(y^5+yx+x^5)^2,
\quad (C,O)\sim B_{23,2}\circ B_{2,23},
\ \iota=10,\ \mu=51.
\end{split}
\end{equation*}

\noindent{\bf{Example-Case II-(b).}}
We assume that the tangent cone of the quintic $C_5$ 
consists of a line of multiplicity 2.
\begin{equation*}
\begin{split}
&{\bf{(b{\text{-}}1)}}
\quad C_t:\,f(x,y)=(y+x^2)^5+(-y^5+2xy^3+(2x^3+1)y^2+(2+t)x^2y+(1+t)x^4)^2,\\&
\qquad\qquad \ (C_t,O)\sim {(B_{4,2}^2)}^{B_{32,2}+B_{2,2}},\ \iota=10, \ \mu=55.
 {\text { This singularity 
 degenerate into:}}\\&
\qquad\quad\  C_0:\,f(x,y)=(y+x^2)^5+(-y^5+2xy^3+(2x^3+1)y^2+2x^2y+x^4)^2,\\&
\qquad\qquad\ (C_0,O)\sim B_{25,4},
 \iota=10, \ \mu=72.\\&
 {\bf{(b{\text{-}}2)}}\quad C:\,f(x,y)=(y^2+xy)^5+(y^2+x^5)^2,\quad
 (C,O)\sim {(B_{5,2}^2)}^{ B_{15,2}},\quad \iota=7,\ \mu=42.
\end{split}
\end{equation*}

\noindent{\bf{Example-Case III-(a).}}
We assume that the tangent cone of the quintic $C_5$ 
consists of distinct three lines.
\begin{equation*}
\begin{split}
&{\bf{(a{\text{-}}1)}}
\quad C:\,f(x,y)=(y+x^2)^5+(y^5+xy^3+(x^3+x)y^2+(x^3+x^2)y+x^4)^2
\\&\qquad \qquad\ \qquad 
 (C,O)\sim B_{36,2}\circ B_{4,3},\quad \iota=10,\ \ \mu=56.
\\&{\bf{(a{\text{-}}2)}}
\quad  C:\,f(x,y)=x^5y^5+(y^5+xy^2+x^2y+x^5)^2\\&\qquad \qquad\ \qquad 
(C,O)\sim B_{16,2}\circ {(B_{1,1}^2)}^{ B_{4,2}}\circ B_{2,16},
\quad \iota=10,\ \ \mu=57.
\end{split}
\end{equation*}

\noindent{\bf{Example-Case III-(b).}}
We assume that the tangent cone of the quintic $C_5$ 
consists of a double lines and a single line.
\begin{equation*}
\begin{split}
&{\bf{(b{\text{-}}1)}}
\quad C:\,f(x,y)=(y+x^2)^5+(y^5+y^2x-x^5)^2,\ (C,O)\sim B_{29,2}\circ B_{6,3},\ \iota=10,\ \mu=61.
\\&{\bf{(b{\text{-}}2)}}
\ \ C:\,f(x,y)=x^5y^5+(y^5+y^2x+x^5)^2,\  (C,O)\sim {(B_{4,2}^2)}^{2B_{5,2}}\circ B_{2,16},\ \iota=10,\ \mu=61.
\end{split}
\end{equation*}

\noindent{\bf{Example-Case III-(c).}}
We assume that the tangent cone of the quintic $C_5$ 
consists of a line with multiplicity 3.
\begin{equation*}
\begin{split}
&{\bf{(c{\text{-}}1)}}
\quad C:\,f(x,y)=(y+x^2)^5+(y^5+y^3+y^2x^2+yx^3+x^5)^2,\\&\qquad \qquad\ \qquad 
(C,O)\sim B_{29,2}\circ B_{6,3},\quad \iota=10, \ \mu=61.
\\&{\bf{(c{\text{-}}2)}}
\quad C:\,f(x,y)=x^5y^5+(y^3+x^5)^2,\quad (C,O)\sim {(B_{5,3}^2)}^{ B_{10,2}}
,\quad \iota=8, \ \mu=55.
\end{split}
\end{equation*}

\noindent{\bf{Example-Case IV-(a).}}
We assume that the tangent cone of the quintic $C_5$ 
consists of distinct four lines.
\begin{equation*}
\begin{split}
&{\bf{(a{\text{-}}1)}}
\quad C:\,f(x,y)=(y^5+y^4x+xy^3+2y^2x^3+x^3y+x^5)^2+(y^2+y+x^2)^5,
\\&
\qquad \qquad  (C,O)\sim B_{29,2}\circ B_{6,3},\quad \iota=10, \ \mu=61.
\\&{\bf{(a{\text{-}}2)}}\quad C:\,f(x,y)=x^5y^5+(y^5+y^3x+yx^3+x^5)^2,\\&
\qquad \qquad 
(C,O)\sim B_{14,2}\circ {(B_{2,2}^2)}^{2B_{2,2}}\circ B_{2,14}
,\quad \iota=10,\ \mu=67.
\end{split}
\end{equation*}

\noindent{\bf{Example-Case IV-(b).}}
We assume that the tangent cone of the quintic $C_5$ 
consists of a double line and distinct two lines.
\begin{equation*}
\begin{split}
&{\bf{(b{\text{-}}1)}}
\quad C:\,f(x,y)=(y^5+y^4+x^2y^3+x^2y^2+x^4y)^2+(y+x^2)^5,\\&
\qquad \qquad 
(C,O)\sim B_{22,2}\circ B_{6,3},\quad \iota=10,\ \mu=66.
\\&{\bf{(b{\text{-}}2)}}\quad C:\,
f(x,y)=(2y^5+(x^2+x)y^3-y^2x^2+2x^5)^2+(yx-x^2)^5,\\&
\qquad \qquad 
(C,O)\sim B_{6,4}\circ {(B_{1,1}^2)}^{ B_{7,2}}\circ B_{2,9}
,\quad \iota=10,\ \mu=69.
\end{split}
\end{equation*}

\noindent{\bf{Example-Case IV-(c).}}
We assume that the tangent cone of the quintic $C_5$ 
consists of a triple line and a single line.
\begin{equation*}
\begin{split}
&{\bf{(c{\text{-}}1)}}
\quad C:\,f(x,y)=(2y^5+y^4x+(x^2+x)y^3+3y^2x^3)^2+(y^2+(x+1)y+3x^2)^5
\\& \qquad\qquad \qquad 
(C,O)\sim B_{15,2}\circ B_{10,3},\quad \iota=10,\ \mu=71.
\\&{\bf{(c{\text{-}}2)}}\quad C:\,
f(x,y)=(y^5+y^3x+x^5)^2+y^5x^5,
\ (C,O)\sim {(B_{4,3}^2)}^{B_{5,2}}\circ B_{2,9},\ \iota=10,\ \mu=71.
\end{split}
\end{equation*}

\noindent{\bf{Example-Case IV-(d).}}
We assume that the tangent cone of the quintic $C_5$ 
consists of line with multiplicity 4.
\begin{equation*}
\begin{split}
&{\bf{(d{\text{-}}1)}}
\quad C:\,f(x,y)=(2y^5+(x+1)y^4+x^2y^3)^2+(y^2+(x+1)y+x^2)^5
\\&\qquad \qquad  \qquad 
(C,O)\sim B_{20,5}, \quad \iota=10,\ \mu=76. 
\\&{\bf{(d{\text{-}}2)}}\quad C:\,f(x,y)=(y^4+x^5)^2+(y^2+yx)^5,
\quad (C,O)\sim {(B_{5,4}^2)}^{B_{5,2}}, \quad \iota=9,\ \mu=68. 
\end{split}
\end{equation*}

\noindent{\bf{Example-Case IV-(e).}}
We assume that the tangent cone of the quintic $C_5$ 
consists of two double lines.
\begin{equation*}
\begin{split}
\quad C:\,f(x,y)=(y^5+y^2x^2+x^5)^2+y^5x^5,\
 (C,O)\sim {(B_{3,2}^2)}^{B_{5,2}}\circ {(B_{2,3}^2)}^{B_{5,2}},
 \ \iota=10,\ \mu=71. 
\end{split}
\end{equation*}


$Acknowledgement.$
I would like to express my deepest gratitude to 
Professor Mutsuo Oka
who has proposed this problem and navigated me during the preparation
of this paper.
\def\cprime{$'$} \def\cprime{$'$} \def\cprime{$'$} \def\cprime{$'$}
  \def\cprime{$'$}

\end{document}